\documentclass[11pt,a4paper]{article}
\usepackage[naustrian,UKenglish]{babel}
\usepackage[utf8x]{inputenc}
\usepackage[T1]{fontenc}
\usepackage[shortcuts]{extdash}
\usepackage{latexsym,amsfonts,amsmath,amsthm,amssymb}
\usepackage{fullpage}
\usepackage{graphicx}
\usepackage{xcolor,bm,bbm}
\usepackage{array}
\usepackage{paralist}
\usepackage{mathrsfs}

\newcommand*{\zalmeng}[1]{\mathbb{#1}}
\newcommand*{\NZ}[0]{\zalmeng{N}}

\newcommand*{\RZ}[0]{\zalmeng{R}}
\newcommand*{\CZ}[0]{\zalmeng{C}}
\newcommand*{\HZ}[0]{\zalmeng{H}}
\newcommand*{\KZ}[0]{\zalmeng{K}}
\newcommand*{\Kug}[2]{\mathbb{B}_{#1}^{#2}} 

\newcommand*{\Sph}[2]{\mathbb{S}_{#1}^{#2}}
\newcommand*{\ez}[0]{\mathsf{e}}
\newcommand*{\ie}[0]{\mathsf{i}}

\newcommand*{\Kfs}[1]{\xrightarrow[#1]{\text{\upshape a.s.}}}
\newcommand*{\KiWsk}[1]{\xrightarrow[#1]{\mathbb{P}}}
\newcommand*{\KiVert}[1]{\xrightarrow[#1]{\text{\upshape d}}}
\newcommand*{\GlVert}[0]{\stackrel{\text{\upshape d}}{=}}
\newcommand*{\ellpe}[2]{\ell_{#1}^{#2}}
\newcommand*{\KugVol}[2]{\omega_{#1}^{#2}}
\newcommand*{\KegM}[2]{\kappa_{#1}^{#2}}
\newcommand*{\Normal}[1]{\mathcal{N}_{#1}}

\newcommand*{\trapo}[1]{#1^{\mathsf{T}}}
\newcommand*{\adj}[1]{#1^*}
\newcommand*{\inv}[1]{#1^{-1}}
\newcommand*{\konj}[1]{\overline{#1}}
\newcommand*{\Colon}[0]{\,:\,} 
\newcommand*{\Rand}[1]{\partial #1}
\newcommand*{\diff}[0]{\mathrm{d}}

\newcommand*{\Schatten}[2]{S_{#1}^{#2}}
\newcommand*{\BetaMatr}[1]{B_{#1}^\text{\upshape I}}
\newcommand*{\DirichletMatr}[1]{D_{#1}^\text{\upshape I}}
\newcommand*{\Perm}[1]{\mathfrak{S}_{#1}}
\newcommand*{\offen}[1]{#1^{\circ}}
\newcommand*{\schluss}[1]{\overline{#1}}
\newcommand*{\WMasz}[0]{\mathcal{M}_1}
\newcommand*{\CeBe}[0]{C_{\text{\upshape b}}}
\newcommand*{\CeCe}[0]{C_{\text{\upshape c}}}
\newcommand*{\dw}[0]{d_{\text{\upshape w}}}
\newcommand*{\Sym}[0]{\text{\upshape Sym}}
\newcommand*{\Pos}[0]{\text{\upshape Pos}}
\newcommand*{\Orth}[0]{\mathrm{U}}
\newcommand*{\OrthVol}[0]{\Omega}
\newcommand*{\Borel}[0]{\mathcal{B}}

\renewcommand{\epsilon}{\varepsilon}
\renewcommand{\phi}{\varphi}

\DeclareMathOperator{\Spur}{tr}

\DeclareMathOperator{\diag}{diag}
\DeclareMathOperator{\Wsk}{\mathbb{P}}
\DeclareMathOperator{\Erw}{\mathbb{E}}

\DeclareMathOperator{\Cov}{Cov}
\DeclareMathOperator{\Rel}{Rel}
\DeclareMathOperator{\vol}{\mathit{v}}
\DeclareMathOperator{\BigO}{O}
\DeclareMathOperator{\smallO}{o}

\DeclareMathOperator{\Hausd}{\mathcal{H}}
\DeclareMathOperator{\Adj}{adj}
\DeclareMathOperator{\Gleichv}{\mathcal{U}}
\DeclareMathOperator{\Vek}{vec}
\DeclareMathOperator{\supp}{supp}
\DeclareMathOperator{\Ind}{1}

\newlength{\absatz}
\newcommand*{\Absatz}[0]{\hspace*{\absatz}}

\theoremstyle{plain}
\newtheorem{Sa}{Proposition}[section]

\newtheorem{Lem}[Sa]{Lemma}
\newtheorem{Thm}{Theorem}

\theoremstyle{definition}
\newtheorem{Def}[Sa]{Definition}

\theoremstyle{remark}
\newtheorem{Bem}[Sa]{Remark}

\makeatletter
\renewcommand{\@upn}{} 
\makeatother

\allowdisplaybreaks

\begin{document}


\title{\bfseries Asymptotic theory of Schatten classes}

\author{Michael L.\ Juhos, Zakhar Kabluchko, and Joscha Prochno}



\date{}

\maketitle

\begin{abstract}
\small
The study of Schatten classes has a long tradition in geometric functional analysis and related fields. In this paper we study a variety of geometric and probabilistic aspects of finite-dimensional Schatten classes of not necessarily square matrices. Among the main results are the exact and asymptotic volume of the Schatten\-/\(\infty\) unit ball, the boundedness of its isotropy constant, a Poincaré\--Maxwell\--Borel lemma for the uniform distribution on the Schatten\-/\(\infty\) ball, and Sanov\-/type large deviations principles for the singular values of matrices sampled uniformly from the Schatten\-/\(p\) unit ball for any \(0 < p \leq \infty\).

\vspace{0.5\baselineskip}
\end{abstract}

\tableofcontents

\section{Introduction}
\label{sec:einfuehrung}

\subsection{General overview}
\label{sec:ueberblick}

The Schatten classes (or Schatten\--von Neumann classes) of compact operators are a classical object of study in functional analysis, and their finite\-/dimensional counterparts, that is, spaces of matrices endowed with the \(\ellpe{p}{}\)\=/norm of their singular values, have attracted much attention within the last few decades and many results have been produced in a number of disciplines, among them local theory of Banach spaces, random matrix theory, and (asymptotic) convex geometry. 

The Schatten\=/\(p\) class \(\Schatten{p}{}\), where \(p \in (0, \infty]\), consists of all compact linear operators between two given Hilbert spaces whose sequence of singular values lies in the sequence space \(\ellpe{p}{}\), and thus it subsumes the special cases of nuclear or trace\-/class operators (\(p = 1\)) and of Hilbert\--Schmidt operators (\(p = 2\)). It was introduced by Schatten in~\cite[Chapter~6, pp.\ 71]{Schatten10960} (though he said ``completely continuous'' for ``compact''), having its roots in earlier works by himself~\cite{Schatten1946} and with von Neumann~\cite{SchN1946, SchN1948} concerning nuclear operators on Hilbert spaces; those works were generalized by Ruston~\cite{Ruston1951} to Banach spaces and by Grothendieck~\cite{Groth1951} to locally convex spaces.

From the point of view of the geometry of Banach spaces, Gordon and Lewis~\cite{GordonLewis1974} obtained that the space \(\Schatten{p}{}\) for \(p \neq 2\) does not have local unconditional structure and therefore does not admit an unconditional basis,
Tomczak\-/Jaegerman~\cite{Tomczak1974} showed that this space has Rademacher cotype~\(2\), and
Szarek and Tomczak\-/Jaegerman~\cite{SzarTomcz1980} provided bounds for the volume ratio of \(\Schatten{1}{n}\).
Saint Raymond~\cite{Saintraymond1984} computed the asymptotic volumes of real and complex Schatten balls up to a yet undetermined factor,
and König, Meyer, and Pajor~\cite{KMP1998} established boundedness of the isotropic constant for all \(p \in [1, \infty)\).
Guédon and Paouris~\cite{GuedonPaouris2007} proved concentration of mass properties,
Barthe and Cordero\-/Erausquin~\cite{BartheCordero2013} studied variance estimates,
Chávez\-/Domínguez and Kuzarova~\cite{ChavezKutzarova2015} determined the Gelfand widths of certain identity mappings between finite\-/dimensional Schatten classes,
Radtke and Vritsiou~\cite{RadtkeVritsiou2020} proved the thin\-/shell conjecture and
Vritsiou~\cite{Vritsiou2018} the variance conjecture for the spectral norm, i.e.\ \(p = \infty\);
Hinrichs, Prochno, and Vybíral~\cite{HinProVyb2017} computed the entropy numbers for natural embeddings \(\Schatten{p}{n} \hookrightarrow \Schatten{q}{n}\), and in \cite{HinProVyb2021} the same authors determined the corresponding Gelfand numbers;
Prochno and Strzelecki~\cite{ProStrz2022} improved the last stated results and complemented them by the approximation and Kolmogorov numbers.
In a series of papers, Kabluchko, Prochno, and Thäle \cite{KPTh2020_3, KPTh2020_1, KPTh2020_2} refined some of the previous works by giving the precise asymptotic volume of Schatten balls, they computed the volume ratio, considered the volume of intersections in the spirit of Schechtman and Schmuckenschläger~\cite{SS1991}, and derived Sanov\-/type large deviations principles for the empirical measures of singular values of random matrices sampled from Schatten balls.
This last results stands in line with the study of large deviations for spectral measures of random matrices initiated by Ben Arous and Guionnet~\cite{BenAG1997}, and Kaufmann and Thäle~\cite{KaufThaele2022} have generalized the underlying distribution to a certain mixture of uniform distribution and cone measure on the sphere in the spirit of Barthe, Guédon, Mendelson, and Naor~\cite{BGMN2005}.
The question of volume of intersections has also been taken up by Sonnleitner and Thäle~\cite{STh2023}.
Very recently Dadoun, Fradelizi, Guédon, and Zitt~\cite{DFGZ2023} determined the asymptotics of the moments of inertia of Schatten balls, thereby also providing the limit of the isotropic constant, and confirmed the variance conjecture for \(p > 3\).

One of the main deficiencies of most results mentioned above is that they have only been obtained for Schatten classes of \emph{square} matrices, some even more specially for Hermitian matrices, which is in a sense quite restrictive. The aim of the present article is to extend the scope of some of the results to Schatten classes of \(m\)\-/by\-/\(n\) matrices for arbitrary \(m, n \in \NZ\), in as far as that is possible or sensible, thereby complementing and generalizing the existing literature. As we shall see, this is far from being trivial in several situations. 

Let us be more specific about the results obtained in this paper. First we provide the \emph{exact} and asymptotic volume of the Schatten\-/\(\infty\) unit ball of \(m\)\-/by\-/\(n\) matrices (Theorem~\ref{sa:volumen_b}); the former is done via a Weyl\-/type integral transformation formula (similar to the one already used by Saint Raymond~\cite{Saintraymond1984}) and evaluating the Selberg integral that appears. 
Next we provide bounds for its isotropic constant (Theorem~\ref{sa:lb_beschraenkt}), thereby adding to the work of König, Meyer, and Pajor~\cite{KMP1998}; notably, we are also able to give the precise limit as \(n \to \infty\) in the case when \(\lim_{n \to \infty} \frac{m}{n}\) exists. As the last major purely geometrical result, we clarify the relation between the Hausdorff and cone measures on the Schatten\-/\(p\) unit sphere for arbitrary \(p\) (Theorem~\ref{sa:tildes_volumen}); in the case of the classical \(\ellpe{p}{}\)\-/spheres this was done by Naor and Romik~\cite[Lemma~2]{NR2003}, and, interestingly, the formulas for the Radon\--Nikod\'ym derivatives turn out to be formally equal.

We then continue to explore asymptotic properties of the Schatten\=/\(\infty\) balls via a probabilistic representation of the uniform distribution on the ball (Proposition~\ref{sa:gleichv_b_probdarst1}) in the spirit of Schechtman and Schmuckenschläger~\cite{SS1991} and of Rachev and Rüschendorf~\cite{RR1991}, who introduced those for the classical \(\ellpe{p}{}\)\-/balls, leading us to a Poincaré\--Maxwell\--Borel result (Theorem~\ref{sa:pmb}) stating that, fixing the number \(m\) of rows and letting the number \(n\) of columns tend to infinity, the projection onto the first \(k\) columns converges to an \(m\)\-/by\-/\(k\) matrix with i.i.d.\ Gaussian components. Under the same premises we also prove a central limit theorem for the inner product of two points sampled independently from the unit ball or sphere (Theorem~\ref{sa:zgs_inneresprodukt}).

Then we introduce an apparently new notion of symmetry of (probability) measures w.r.t.\ a given convex body, and apply this notion, together with two\-/sided unitary invariance, to Schatten classes. As the main result (Theorem~\ref{sa:gleichvert_schattenq_kugel_sphaere}), we produce a probabilistic representation of those random matrices whose distribution follows this new symmetry w.r.t.\ Schatten\=/\(p\) balls, thereby subsuming the representations for the singular values applied by Kabluchko, Prochno, and Thäle~\cite{KPTh2020_3, KPTh2020_1, KPTh2020_2} in the non\-/Hermitian case.

Finally, we present Sanov\-/type large deviations principles for random matrices sampled from Schatten\-/\(p\) balls for any \(p \in (0, \infty]\) (Theorems~\ref{sa:ldp_singulaerwerte} and \ref{sa:ldp_singulaerwerte_unendl}). This extends the result for non\-/Hermitian random matrices given in \cite{KPTh2020_2}. It is important to note here that for non\-/square matrices the joint density of singular values contains an additional dimension\-/dependent nonconstant factor, which renders the adaption of the proof in places highly nontrivial.

\subsection{Organization of the paper}
\label{sec:organisation}

Since the present paper contains a wide range of results it is structured according to the general themes of those results, and the (mostly short) proofs directly follow their propositions; the exception is Section~\ref{sec:sanovldp} where the central results are collected at the beginning and the corresponding (extensive) proofs are deferred to subsections.

Section~\ref{sec:setup} lays out the mathematical setup including the norms and spaces, as well as frequently used auxiliary results.

Section~\ref{sec:geometrie} treats the elementary geometric properties of Schatten balls, including Theorems~\ref{sa:volumen_b}, \ref{sa:lb_beschraenkt}, and \ref{sa:tildes_volumen}.

Section~\ref{sec:wahrscheinlichkeit} contains the results pertaining to the weak limit theorems for the Schatten\=/\(\infty\) ball, that is Proposition~\ref{sa:gleichv_b_probdarst1} and Theorems~\ref{sa:pmb} and~\ref{sa:zgs_inneresprodukt}.

Section~\ref{sec:symmetrischemasze} deals with \(K\)\=/symmetric measures and concludes with Theorem~\ref{sa:gleichvert_schattenq_kugel_sphaere}.

Finally Section~\ref{sec:sanovldp} is dedicated to the Sanov\-/type large deviations principles for the empirical spectral measures stated in Theorems~\ref{sa:ldp_singulaerwerte} and \ref{sa:ldp_singulaerwerte_unendl}.

\section{Mathematical setup}
\label{sec:setup}

\subsection{Norms and spaces}
\label{sec:normenraeume}

Let \(\KZ \in \{\RZ, \CZ\}\), and by \(\beta \in \{1, 2\}\) denote the \(\RZ\)\-/dimension of \(\KZ\). (Most of the present results remain valid even for \(\KZ = \HZ\), the skew field of quaternions; then \(\beta = 4\).) The nonnegative and positive real numbers are written \(\RZ_{\geq 0}\) and \(\RZ_{> 0}\), resp.; \(\ie\) is the imaginary unit, and \(\Re(x)\), \(\Im(x)\), \(\konj{x}\) denote the real part, imaginary part, and complex conjugate of \(x\) (applied entrywise for vectors and matrices). For the \(d\)\-/dimensional Lebesgue measure we write \(\vol_d\). For \(m, n \in \NZ\) with \(m \leq n\) consider \(\KZ^{m \times n}\), the \(\beta m n\)\-/dimensional real vector space of \((m \times n)\)\-/matrices. This is endowed with a Euclidean structure via the inner product (the \emph{Hilbert\--Schmidt} or \emph{Frobenius inner product})
\begin{equation*}
\langle x, y \rangle := \Re\Spur(x \adj{y}) = \Re\biggl( \sum_{i = 1}^m \sum_{j = 1}^n x_{i, j} \konj{y_{i, j}} \biggr)\!,
\end{equation*}
where \(\Spur\) means the trace of a square matrix. \(I_m\) denotes the \((m \times m)\)\-/identity matrix. The collection of self\-/adjoint \((m \times m)\)\-/matrices is denoted by \(\Sym_{m, \beta}\), and that of self\-/adjoint and positive\-/definite \((m \times m)\)\-/matrices by \(\Pos_{m, \beta}\). On \(\Sym_{m, \beta}\) define the usual partial order (also called \emph{Loewner\-/order}): for \(a, b \in \Sym_{m, \beta}\), \(a \leq b\) iff \(b - a\) is positive\-/semidefinite, or equivalently \(\adj{\xi} a \xi \leq \adj{\xi} b \xi\) for all \(\xi \in \KZ^m\). (Elements of \(\KZ^m\) are always considered column\-/vectors here.)

Recall that for any \(x \in \KZ^{m \times n}\) the \((m \times m)\)\-/matrix \(x \adj{x}\) is symmetric and positive\-/semidefinite, and hence all the latter's eigenvalues are nonnegative real numbers, whose nonnegative square\-/roots are called the \emph{singular values} of \(x\). Let \(s(x) := (s_i(x))_{i \leq m} \in \RZ^m\) be the vector of the singular values of \(x\), arranged non\-/increasingly.

For \(p \in (0, \infty]\) call \(\lVert \cdot \rVert_p\) the \(\ellpe{p}{}\)\-/norm on \(\KZ^m\) (which is a proper quasinorm iff \(p < 1\)), i.e., for \(\xi = (\xi_i)_{i \leq m}\),
\begin{equation*}
\lVert \xi \rVert_p := \begin{cases} \bigl( \sum_{i = 1}^m \lvert \xi_i \rvert^p \bigr)^{1/p} & \text{if } p < \infty, \\ \max\{\lvert \xi_i \rvert \Colon i \leq m\} & \text{if } p = \infty; \end{cases}
\end{equation*}
the corresponding space is written \(\ellpe{p, \beta}{m} := (\KZ^m, \lVert \cdot \rVert_p)\), the unit ball is denoted by \(\Kug{p, \beta}{m}\), the unit sphere by \(\Sph{p, \beta}{m}\),\footnote{Our notation differs from the convention of writing \(\Sph{p}{m - 1}\), where in the real case \(m - 1\) is the dimension indeed. But already in the comlex case this is contradictory because \(\Rand{\Kug{p, 2}{m}}\) has real dimension \(2 m - 1\). Therefore we have opted to only notate the number of components.} and \(\KugVol{p, \beta}{m} := \vol_{\beta m}(\Kug{p, \beta}{m})\). On \(\KZ^{m \times n}\) define the \emph{Schatten\=/\(p\)\-/norm} \(\lVert \cdot \rVert_{\Schatten{p}{}}\) by
\begin{equation*}
\lVert x \rVert_{\Schatten{p}{}} := \lVert s(x) \rVert_p = \Spur((x \adj{x})^{p/2})^{1/p} \quad \text{for \(x \in \KZ^{m \times n}\).}
\end{equation*}
In particular, \(\lVert x \rVert_{\Schatten{\infty}{}} = s_1(x)\) is the \emph{spectral norm} of \(x\) and \(\lVert x \rVert_{\Schatten{2}{}} = \lVert s(x) \rVert_2 = \langle x, x \rangle^{1/2}\) is the \emph{Hilbert\--Schmidt} or \emph{Frobenius norm.} The space \(\Schatten{p}{} := \Schatten{p, \beta}{m \times n} := (\KZ^{m \times n}, \lVert \cdot \rVert_{\Schatten{p}{}})\) is called a \emph{(finite\-/dimensional) Schatten\=/\(p\) class;} we denote its unit ball by \(\Kug{\Schatten{p}{}, \beta}{m \times n}\) and its unit sphere by \(\Sph{\Schatten{p}{}, \beta}{m, n}\); additionally define \(\KugVol{\Schatten{p}{}, \beta}{m \times n} := \vol_{\beta m n}(\Kug{\Schatten{p}{}, \beta}{m \times n})\).

\begin{Bem}\label{rem:dimensionen}
The assumption \(m \leq n\) is no severe restriction: given arbitrary \(m, n \in \NZ\) and \(x \in \KZ^{m \times n}\), the Hermitian positive\-/semidefinite matrices \(\adj{x} x\) and \(x \adj{x}\) have the same positive eigenvalues with matching multiplicities and the rest are zero, hence singular values are defined also for ``tall'' matrices, and we have \((s_i(\adj{x}))_{i \leq r} = (s_i(x))_{i \leq r}\) and consequently \(\lVert \adj{x} \rVert_{\Schatten{p}{}} = \lVert x \rVert_{\Schatten{p}{}}\), where \(r = \min\{m, n\}\). There are results in the literature that also take the zero singular values into account, e.g.\ \cite[Theorem~5.5.10]{HiaiPetz2000}.
\end{Bem}

\subsection{Auxiliary notions and results}
\label{sec:auxiliar}

Recall that a matrix \(u \in \KZ^{n \times m}\) is called \emph{semiunitary} (\emph{semiorthogonal} if \(\KZ = \RZ\)) iff \(\adj{u} u = I_m\). The collection of semiunitary (semiorthogonal) \((n \times m)\)\-/matrices is called the \emph{Stiefel\-/manifold} and notated by \(\Orth_{n, m; \beta}\) (it is indeed a submanifold of \(\KZ^{n \times m}\)); in particular \(\Orth_{m; \beta} := \Orth_{m, m; \beta}\) denotes the set of unitary (orthogonal) matrices, they satisfy \(\adj{u} u = u \adj{u} = I_m\).

For any \(x \in \KZ^{m \times n}\) there exist \(u \in \Orth_{n; \beta}\) and \(v \in \Orth_{m; \beta}\) such that \(x = v \diag(s(x)) \adj{u}\), the so\-/called \emph{singular\-/value\-/decomposition} of \(x\); by slight abuse of notation, here we mean \(\diag(\xi) := (\xi_i \delta_{i, j})_{i \leq m, j \leq n} \in \KZ^{m \times n}\) for a vector \(\KZ^m\), i.e., the matrix whose entries on the main diagonal are the components of \(\xi\) and are zero else.

The conventions for integration (which can be justified by exploiting that all domains of integration concerned are Riemannian submanifolds of some full matrix space) are as follows, where for \(\beta = 1\) ignore all imaginary parts: on \(\KZ^{m \times n}\) the volume form is (up to sign, as are the subsequent) \(\bigwedge_{i = 1}^m \bigwedge_{j = 1}^n (\diff \Re(x_{i, j}) \wedge \diff \Im(x_{i, j}))\) (which corresponds to \(\beta m n\)\=/dimensional Lebesgue measure); on \(\Sym_{m, \beta}\) and \(\Pos_{m, \beta}\) it is \(\bigwedge_{i = 1}^m \bigl( \bigwedge_{j = 1}^i \diff\Re(x_{i, j}) \wedge \bigwedge_{j = 1}^{i - 1} \diff\Im(x_{i, j}) \bigr)\); and on \(\Orth_{n, m; \beta}\), \(\bigwedge_{j = 1}^m \bigl( \bigwedge_{i = j + 1}^n \Re\bigl(\sum_{k = 1}^n \konj{x_{k, i}} \, \diff x_{k, j} \bigr) \wedge \bigwedge_{i = j}^n \Im\bigl( \sum_{k = 1}^n \konj{x_{k, i}} \, \diff x_{k, j} \bigr) \bigr)\) (where \(x_{i, j}\) for \(j \geq m + 1\) are such that for the augmented matrix, \((x_{i, j})_{i, j \leq n} \in \Orth_{n; \beta}\)); in all cases we write shortly \(\diff x\). Note that the induced measure on \(\Orth_{n, m; \beta}\) is invariant under left and right multiplication with unitary matrices; in particular, on \(\Orth_{m; \beta}\) this gives a Haar measure.

For \(z \in \CZ\) with \(\Re(z) > \frac{\beta (m - 1)}{2}\) define the multivariate gamma function \(\Gamma_{m, \beta}\) by
\begin{equation*}
\Gamma_{m, \beta}(z) = \int_{\Pos_{m, \beta}} \det(x)^{z - 1 - \beta (m - 1)/2} \, \ez^{-\Spur(x)} \, \diff x = \pi^{\beta m (m - 1)/4} \prod_{k = 0}^{m - 1} \Gamma\Bigl( z - \frac{\beta k}{2} \Bigr).
\end{equation*}

The Riemannian volume of \(\Orth_{n; \beta}\) is known at least since Hurwitz~\cite{Hurwitz1897} (note that actually Hurwitz considers \(\{u \in \Orth_{n; \beta} \Colon \det(u) = 1\}\), the special unitary group) and is a special case of the Riemannian volume of \(\Orth_{n, m; \beta}\) which equals
\begin{equation}\label{eq:vol_stiefel}
\begin{split}
\OrthVol_{n, m; \beta} &:= 2^{\beta m (m - 1)/4} \int_{\Orth_{n, m; \beta}} \diff u\\
&= \frac{2^{m + \beta m (m - 1)/4} \, \pi^{\beta (2 m n - m (m - 1))/4}}{\prod_{k = 0}^{m - 1} \Gamma\bigl( \frac{\beta (n - k)}{2} \bigr)} = \frac{2^{m + \beta m (m - 1)/4} \, \pi^{\beta m n/2}}{\Gamma_{m, \beta}\bigl( \frac{\beta n}{2} \bigr)}.
\end{split}
\end{equation}
In particular we set \(\OrthVol_{n; \beta} := \OrthVol_{n, n; \beta}\), the volume of \(\Orth_{n; \beta}\); then \(\OrthVol_{n, m; \beta} = 2^{-\beta m (n - m)/2} \frac{\OrthVol_{n; \beta}}{\OrthVol_{n - m; \beta}}\).

The \emph{Selberg integral} (see~\cite{Selberg1944}) will be of some importance to our results: let \(m \in \NZ\) and let \(\alpha, \beta, \gamma \in \CZ\) with \(\Re(\alpha), \Re(\beta) > 0\) and \(\Re(\gamma) > -\min\bigl\{ \frac{1}{m}, \frac{\Re(\alpha)}{m - 1}, \frac{\Re(\beta)}{m - 1} \bigr\}\), then
\begin{equation*}
\int_{(0, 1)^m} \prod_{i = 1}^m \bigl( x_i^{\alpha - 1} \, (1 - x_i)^{\beta - 1} \bigr) \prod_{1 \leq i < j \leq m} \lvert x_i - x_j \rvert^{2 \gamma} \, \diff x = \prod_{k = 0}^{m - 1} \frac{\Gamma(\alpha + k \gamma) \Gamma(\beta + k \gamma) \Gamma(\gamma + 1 + k \gamma)}{\Gamma(\alpha + \beta + (m - 1) \gamma + k \gamma) \Gamma(\gamma + 1)}.
\end{equation*}

Often we are going to need the asymptotics of the product of gamma\-/functions as they appear in \(\OrthVol_{n, m; \beta}\) and the Selberg integral. Here we also remind the reader of the \emph{Landau symbols:} for sequences \((a_n)_{n \in \NZ} \in \KZ^\NZ\) and \((b_n)_{n \in \NZ} \in \RZ_{> 0}^\NZ\), we write \(a_n = \BigO(b_n)\) to mean, \(\lvert a_n \rvert \leq C b_n\) for eventually all \(n \in \NZ\), with some \(C \in \RZ_{> 0}\); and \(a_n = \smallO(b_n)\) to mean \(\lim_{n \to \infty} \frac{a_n}{b_n} = 0\).

\begin{Lem}\label{lem:produkt_gamma}
For any \(\beta \in \RZ_{> 0}\) we have
\begin{equation*}
\begin{split}
\log\biggl( \prod_{k = 1}^n \Gamma\Bigl( \frac{\beta k}{2} \Bigr) \biggr) &= \frac{\beta n^2}{4} \log\Bigl( \frac{\beta n}{2} \Bigr) - \frac{3 \beta n^2}{8} + \frac{(\beta - 2) n}{4} \log\Bigl( \frac{\beta n}{2} \Bigr)\\
&\quad + \frac{n}{4} \log(4 \pi^2 \, \ez^{2 - \beta}) + \frac{\beta^2 - 6 \beta + 4}{24 \beta} \log\Bigl( \frac{\beta n}{2} \Bigr) + A_\beta + \BigO\Bigl( \frac{1}{n} \Bigr),
\end{split}
\end{equation*}
where \(A_\beta \in \RZ\) is some constant.
\end{Lem}

\begin{proof}
Let \(\beta \in \RZ_{> 0}\). We notice \(\log\bigl( \prod_{k = 1}^n \Gamma(\frac{\beta n}{2}) \bigr) = \sum_{k = 1}^n \log(\Gamma(\frac{\beta k}{2}))\), hence we apply the Euler\--Maclaurin summation formula to the function \(f \colon \RZ_{> 0} \to \RZ\), \(x \mapsto \log(\Gamma(\frac{\beta x}{2}))\), that is,
\begin{equation*}
\sum_{k = 1}^n f(k) = \int_1^n f(x) \, \diff x + \frac{f(n) + f(1)}{2} + \frac{f'(n) - f'(1)}{12} + \int_1^n H_3(x) f'''(x) \, \diff x,
\end{equation*}
where \(H_3(x) := \frac{1}{6} B_3(x - \lfloor x \rfloor)\) with the Bernoulli\-/polynomial \(B_3(x) = x^3 - \frac{3}{2} x^2 + \frac{1}{2} x\).

We already know the asymptotics of the gamma\-/function through Stirling's formula,
\begin{equation*}
\log(\Gamma(x)) = x \log(x) - x - \frac{1}{2} \log(x) + \frac{1}{2} \log(2 \pi) + \frac{1}{12 x} + R(x),
\end{equation*}
where \(R\) satisfies \(R(x) = \BigO(x^{-3})\) as \(x \to \infty\). Therewith we get
\begin{equation*}
\begin{split}
\int_1^n f(x) \, \diff x &= \frac{\beta n^2}{4} \log\Bigl( \frac{\beta n}{2} \Bigr) - \frac{\beta n^2}{8} - \frac{\beta n^2}{4} - \frac{n}{2} \log\Bigl( \frac{\beta n}{2} \Bigr) + \frac{n}{2}\\
&\quad + \frac{n}{2} \log(2 \pi) + \frac{1}{6 \beta} \log\Bigl( \frac{\beta n}{2} \Bigr) + c + \int_1^n R(x) \, \diff x,
\end{split}
\end{equation*}
where we already have gathered all constant terms. Now since \(\lvert R(x) \rvert \leq C x^{-3}\), \(R\) is integrable over \([1, \infty)\), hence the limit \(c_1 := \lim_{n \to \infty} \bigl( c + \int_1^n R(x) \, \diff x \bigr)\) exists in \(\RZ\), and
\begin{equation*}
\biggl\lvert c_1 - c - \int_1^n R(x) \, \diff x \biggr\rvert = \biggl\lvert \int_n^\infty R(x) \, \diff x \biggr\rvert \leq \frac{C}{2 n^2};
\end{equation*}
this results in
\begin{equation*}
\begin{split}
\int_1^n f(x) \, \diff x &= \frac{\beta n^2}{2} \log\Bigl( \frac{\beta n}{2} \Bigr) - \frac{3 \beta n^2}{8} - \frac{n}{2} \log\Bigl( \frac{\beta n}{2} \Bigr)\\
&\quad + \frac{n}{2} \log(2 \pi \ez) + \frac{1}{6 \beta} \log\Bigl( \frac{\beta n}{2} \Bigr) + c_1 + \BigO\Bigl( \frac{1}{n^2} \Bigr).
\end{split}
\end{equation*}
Obviously
\begin{equation*}
\frac{f(n)}{2} = \frac{\beta n}{4} \log\Bigl( \frac{\beta n}{2} \Bigr) - \frac{\beta n}{4} - \frac{1}{4} \log\Bigl( \frac{\beta n}{2} \Bigr) + c_2 + \BigO\Bigl( \frac{1}{n} \Bigr).
\end{equation*}
We also have, via the polygamma\-/functions,
\begin{equation*}
\frac{f'(n)}{12} = \frac{\beta}{24} \, \psi_0\Bigl( \frac{\beta n}{2} \Bigr) = \frac{\beta}{24} \log\Bigl( \frac{\beta n}{2} \Bigr) + \BigO\Bigl( \frac{1}{n} \Bigr),
\end{equation*}
because of \(\psi_0(x) = \log(x) - \frac{1}{2x} + \BigO(x^{-2})\). Lastly we have \(f'''(x) = (\frac{\beta}{2})^3 \psi_2(\frac{\beta x}{2})\); as \(\psi_2(x) = \BigO(x^{-2})\) and \(H_3\) is bounded, the function \(H_3 f'''\) is integrable over \([1, \infty)\), thus the limit \(c_3 := \lim_{n \to \infty} \int_1^n H_3(x) f'''(x) \, \diff x\) exists in \(\RZ\), and the error satisfies
\begin{equation*}
\biggl\lvert c_3 - \int_1^n H_3(x) f'''(x) \, \diff x \biggr\rvert = \biggl\lvert \int_n^\infty H_3(x) f'''(x) \, \diff x \biggr\rvert \leq \frac{C}{n}
\end{equation*}
with some constant \(C\) (different from the previous occurrence). The statement of the lemma now follows by gathering terms and defining \(A_\beta := c_1 + c_3 + c_3\).
\end{proof}

\begin{Bem}
For \(\beta \in \{1, 2\}\) the statement of Lemma~\ref{lem:produkt_gamma} could also be obtained by expressing the product in terms of the \emph{Barnes\-/\(G\)\=/function} whose asymptotics are well\-/known. But this approach fails for general \(\beta\), and even for \(\beta = 1\) the calculations are tedious.
\end{Bem}

For the present investigations the following transformation formulae for integrals with respect to two important matrix factorizations turn out to be useful. (Others exist that we do not need.)

\begin{Sa}\label{sa:integraltrafo}
Let \(f \colon \KZ^{m \times n} \to \KZ\) be measurable and either nonnegative or integrable.
\begin{compactenum}
\item (Polar decomposition~\cite[Lemma~2.1]{Chikuse1990}, also~\cite[Proposition~2.27]{Zhang2017})\footnote{\label{fu:integraltrafo}Note that~\cite{Chikuse1990} works with ``tall'' matrices, i.e.\ \(m \geq n\), so the representation needs to be transposed; also the measure on \(\Orth_{n, m; \beta}\) is normalized differently.}
\begin{equation*}
\int_{\KZ^{m \times n}} f(x) \, \diff x = 2^{-m} \int_{\Pos_{m, \beta}} \int_{\Orth_{n, m; \beta}} f(r^{1/2} \, \adj{u}) \, \diff u \, \det(r)^{\beta (n - m + 1)/2 - 1} \, \diff r.
\end{equation*}
\item (Singular value decomposition~\cite[Lemma~1.5.3, (i)]{Chikuse2003})
\begin{align*}
\int_{\KZ^{m \times n}} f(x) \, \diff x &= \frac{2^{\beta n (n - 1)/4 + \beta m (m - 1)/4 + \beta m/2}}{2^{\beta m n/2} \, m! \OrthVol_{1; \beta}^m \OrthVol_{n - m; \beta}} \int_{\RZ_{> 0}^m} \int_{\Orth_{m; \beta}} \int_{\Orth_{n; \beta}} f(v \diag(s) \adj{u}) \, \diff u \, \diff v\\
&\mspace{230mu} \cdot \prod_{i = 1}^m s_i^{\beta (n - m + 1) - 1} \prod_{1 \leq i < j \leq m} \lvert s_i^2 -s_j^2 \rvert^\beta \, \diff s.
\end{align*}

In particular, if \(f = F \circ s\), where \(F \colon \RZ^m \to \KZ\) (``\(f\) depends only on the singular values of \(x\)''), then~\cite[Proposition~4.1.3]{AGZ2010}
\begin{equation*}
\int_{\KZ^{m \times n}} F(s(x)) \, \diff x = \frac{2^{\beta m/2} \, \OrthVol_{m; \beta} \OrthVol_{n; \beta}}{2^{\beta m n/2} \, m! \OrthVol_{1; \beta}^m \OrthVol_{n - m; \beta}} \int_{\RZ_{> 0}^m} F(s) \prod_{i = 1}^m s_i^{\beta (n - m + 1) - 1} \prod_{1 \leq i < j \leq m} \lvert s_i^2 - s_j^2 \rvert^\beta \, \diff s.
\end{equation*}
\end{compactenum}
\end{Sa}

\begin{Bem}\label{bem:weylkammer}
Some comments on Proposition~\ref{sa:integraltrafo}, part~2, seem in place. The form given in the present work does not completely match the one given in~\cite{Chikuse2003}. Apart from the same notational differences as mentioned in Footnote~\ref{fu:integraltrafo}, the major discrepancy is the domain of integration for the singular values: Chikuse assumes them to be ordered decreasingly, whereas we impose no such restriction. But this poses no serious obstacle, for the following argument: define \(W^m := \{x \in \RZ_{> 0}^m \Colon x_1 > \dotsb > x_m\}\), then almost every \(s \in \RZ_{> 0}^m\) can be uniquely written \(s = \tau s'\), where \(\tau \in \Orth_{m; \beta}\) is a permutation matrix and \(s' \in W^m\); therewith \(v \diag(s) \adj{u} = (v \tau) \diag(s') \adj{(u \diag(\tau, I_{n - m}))}\); also write \(h(s) := \prod_{i = 1}^m s_i^{\beta (n - m + 1) - 1} \prod_{i < j} \lvert s_i^2 - s_j^2 \rvert^\beta\), then clearly \(h(\tau s) = h(s)\). Now decompose \(\RZ_{> 0}^m = \bigcup_{\tau \text{ perm.\ m.}} \tau W^m\), neglecting the set of vectors for which at least two coordinates are equal. Then for any permutation matrix \(\tau\),
\begin{multline*}
\int_{\tau W^m} \int_{\Orth_{m; \beta}} \int_{\Orth_{n; \beta}} f(v \diag(s) \adj{u}) h(s) \, \diff u \, \diff v \, \diff s\\
\begin{aligned}
&= \int_{W^m} \int_{\Orth_{m; \beta}} \int_{\Orth_{n; \beta}} f(v \diag(\trapo{\tau} s) \adj{u}) h(\trapo{\tau} s) \, \diff u \, \diff v \, \diff s\\
&= \int_{W^m} \int_{\Orth_{m; \beta}} \int_{\Orth_{n; \beta}} f(v \trapo{\tau} \diag(s) \adj{(u \diag(\trapo{\tau}, I_{n - m}))}) h(s) \, \diff u \, \diff v \, \diff s\\
&= \int_{W^m} \int_{\Orth_{m; \beta}} \int_{\Orth_{n; \beta}} f(v \diag(s) \adj{u}) h(s) \, \diff u \, \diff v \, \diff s,
\end{aligned}
\end{multline*}
where we also have used that on \(\Orth_{n; \beta}\) and \(\Orth_{m; \beta}\) we are integrating with respect to the (scaled) Haar measures. This leads to
\begin{multline*}
\int_{\RZ_{> 0}^m} \int_{\Orth_{m; \beta}} \int_{\Orth_{n; \beta}} f(v \diag(s) \adj{u}) h(s) \, \diff u \, \diff v \, \diff s\\
\begin{aligned}
&= \sum_{\tau \text{ perm.\ m.}} \int_{\tau W^m} \int_{\Orth_{m; \beta}} \int_{\Orth_{n; \beta}} f(v \diag(s) \adj{u}) h(s) \, \diff u \, \diff v \, \diff s\\
&= m! \int_{W^m} \int_{\Orth_{m; \beta}} \int_{\Orth_{n; \beta}} f(v \diag(s) \adj{u}) h(s) \, \diff u \, \diff v \, \diff s.
\end{aligned}
\end{multline*}
Summing up, an additional factor \(m!\) has to be taken into account.
\end{Bem}

Lastly, for any set \(A\) we denote the (set\-/theoretic) indicator function by \(\Ind_A\), that is, \(\Ind_A(x) = 1\) if \(x \in A\), and \(\Ind_A(x) = 0\) else. Further notions and notations shall be introduced on the fly as needed.

\section{Geometry of Schatten\=/\(p\) balls}
\label{sec:geometrie}

The first result is an explicit expression for the exact and asymptotic \(\beta m n\)\=/dimensional Lebesgue volume of \(\Kug{\Schatten{\infty}{}, \beta}{m \times n}\). In the case \(m = n\) this has already been accomplished in principle by Saint Raymond~\cite[Th\'eor\`eme~1]{Saintraymond1984}, though he does not evaluate the Selberg integral encountered. A statement concerning \(\Kug{\Schatten{p}{}, \beta}{m \times n}\) for \(p < \infty\) is found in Proposition~\ref{sa:volumen_bsp}.

\begin{Thm}\label{sa:volumen_b}
The exact volume of \(\Kug{\Schatten{\infty}{}, \beta}{m \times n}\) is given by
\begin{equation*}
\begin{split}
\KugVol{\Schatten{\infty}{}, \beta}{m \times n} &= \frac{\prod_{k = 0}^{m - 1} \Gamma(1 + \frac{\beta k}{2}) \prod_{k = 0}^{n - 1} \Gamma(1 + \frac{\beta k}{2})}{\prod_{k = 0}^{m + n - 1} \Gamma(1 + \frac{\beta k}{2})} \, \pi^{\beta m n/2}\\
&= \frac{\Gamma_{m, \beta}(1 + \frac{\beta (m - 1)}{2}) \, \Gamma_{n, \beta}(1 + \frac{\beta (n - 1)}{2})}{\Gamma_{m + n, \beta}(1 + \frac{\beta (m + n - 1)}{2})} \, \pi^{\beta m n}.
\end{split}
\end{equation*}
Hence follow the asymptotics of the volume radius,
\begin{equation*}
(\KugVol{\Schatten{\infty}{}, \beta}{m \times n})^{1/(\beta m n)} = \Bigl( \frac{2 \pi \ez^{3/2}}{\beta (m + n)} \Bigr)^{1/2} \Bigl( \frac{m}{m + n} \Bigr)^{(m + 2/\beta - 1)/(4 n)} \Bigl( \frac{n}{m + n} \Bigr)^{(n + 2/\beta - 1)/(4 m)} \bigl( 1 + r_\beta(m, n) \bigr),
\end{equation*}
where \(r_\beta \colon \NZ^2 \to \RZ\) satisfies \(C_1 \leq m n \, r_\beta(m, n) + \frac{\beta^2 - 6 \beta + 4}{24 \beta^2} \log\bigl( \frac{1}{m} + \frac{1}{n} \bigr) \leq C_2\) for all \(m, n \in \NZ\) and \(\beta \in \{1, 2\}\), with global constants \(C_1, C_2 \in \RZ\).

In particular, consider \(m\) as depending on \(n\), then if \(c := \lim_{n \to \infty} \frac{m}{n} \in [0, 1]\) exists, then
\begin{equation*}
\lim_{n \to \infty} (\beta n)^{1/2} (\KugVol{\Schatten{\infty}{}, \beta}{m \times n})^{1/(\beta m n)} = \Bigl( \frac{2 \pi \ez^{3/2}}{1 + c} \Bigr)^{1/2} (1 + c)^{-1/(4 c)} \Bigl( 1 + \frac{1}{c} \Bigr)^{-c/4},
\end{equation*}
where in the case \(c = 0\) we interpret \((1 + c)^{1/c} = \ez\) and \((1 + \frac{1}{c})^c = 1\).
\end{Thm}

\begin{proof}
We apply Proposition~\ref{sa:integraltrafo}, part~2, to \(f := \Ind_{\Kug{\Schatten{\infty}{}, \beta}{m \times n}} = \Ind_{\Kug{\infty, 1}{m}}{} \circ s\) and get
\begin{equation*}
\KugVol{\Schatten{\infty}{}, \beta}{m \times n} = \frac{2^{\beta m/2} \, \OrthVol_{m; \beta} \OrthVol_{n; \beta}}{2^{\beta m n/2} \, m! \OrthVol_{1; \beta}^m \OrthVol_{n - m; \beta}} \int_{[0, 1]^m} \prod_{i = 1}^m s_i^{\beta (n - m + 1) - 1} \prod_{1 \leq i < j \leq m} \lvert s_i^2 - s_j^2 \rvert^\beta \, \diff s.
\end{equation*}
Via the transformation \(s_i \mapsto s_i^{1/2}\) the last integral is identified as Selberg's integral whose exact value equals
\begin{align*}
\int_{[0, 1]^m} \prod_{i = 1}^m s_i^{\beta (n - m + 1) - 1} \prod_{1 \leq i < j \leq m} \lvert s_i^2 - s_j^2 \rvert^\beta \, \diff s &= 2^{-m} \prod_{k = 0}^{m - 1} \frac{\Gamma(\frac{\beta (n - m + 1)}{2} + \frac{\beta k}{2}) \Gamma(1 + \frac{\beta k}{2}) \Gamma(\frac{\beta}{2} + 1 + \frac{\beta k}{2})}{\Gamma(\frac{\beta n}{2} + 1 + \frac{\beta k}{2}) \Gamma(\frac{\beta}{2} + 1)}.
\end{align*}
Plugging in and rearranging terms yields the claimed result.

The asymptotic formulas are an application of Lemma~\ref{lem:produkt_gamma}; in order to apply it, it helps to rearrange terms a bit via \(\Gamma(x + 1) = x \Gamma(x)\), like this,
\begin{equation*}
\frac{\prod_{k = 0}^{m - 1} \Gamma(1 + \frac{\beta k}{2}) \prod_{k = 0}^{n - 1} \Gamma(1 + \frac{\beta k}{2})}{\prod_{k = 0}^{m + n - 1} \Gamma(1 + \frac{\beta k}{2})} = \frac{2}{\beta} \, \frac{\Gamma(m) \Gamma(n) \Gamma(\frac{\beta (m + n)}{2})}{\Gamma(\frac{\beta m}{2}) \Gamma(\frac{\beta n}{2}) \Gamma(m + n)} \, \frac{\prod_{k = 1}^m \Gamma(\frac{\beta k}{2}) \prod_{k = 1}^n \Gamma(\frac{\beta k}{2})}{\prod_{k = 1}^{m + n} \Gamma(\frac{\beta k}{2})}.
\end{equation*}
Now thence follows, after gathering terms,
\begin{equation*}
\begin{split}
(\KugVol{\Schatten{\infty}{}, \beta}{m \times n})^{1/(\beta m n)} &= \Bigl( \frac{2 \pi \ez^{3/2}}{\beta (m + n)} \Bigr)^{1/2} \Bigl( \frac{m}{m + n} \Bigr)^{(m + 2/\beta - 1)/(4 n)} \Bigl( \frac{n}{m + n} \Bigr)^{(n + 2/\beta - 1)/(4 m)}\\
&\quad \cdot \Bigl( \frac{m n}{m + n} \Bigr)^{(\beta^2 - 6 \beta + 4)/(24 \beta^2 m n)} \ez^{(c_\beta + \BigO(1/m + 1/n))/(m n)},
\end{split}
\end{equation*}
where \(c_\beta := \frac{A_\beta}{\beta} + \frac{\beta^2 - 18 \beta + 4}{24 \beta^2} \log\bigl( \frac{\beta}{2} \bigr)\). Note \(\frac{m + n}{m n} = \frac{1}{m} + \frac{1}{n} \leq 2\) because of \(m, n \geq 1\); furthermore by the inequality of arithmetic and geometric means \(\frac{1}{m} + \frac{1}{n} \geq 2 (m n)^{-1/2}\) and hence \((\frac{1}{n} + \frac{1}{m})^{1/(m n)} \to 1\) as \(n \to \infty\). Therefore
\begin{multline*}
\Bigl( \frac{m n}{m + n} \Bigr)^{(\beta^2 - 6 \beta + 4)/(24 \beta^2 m n)} \ez^{(c_\beta + \BigO(1/m + 1/n))/(m n)}\\
= 1 + \frac{1}{m n} \Bigl\{ -\frac{\beta^2 - 6 \beta + 4}{24 \beta^2} \log\Bigl( \frac{1}{m} + \frac{1}{n} \Bigr) + c_\beta + \BigO\Bigl( \frac{1}{m} + \frac{1}{n} \Bigr) \Bigr\} + R_\beta(m, n),
\end{multline*}
where \(R_\beta(m, n)\) is of the order of the square of the first\-/order term; from this it also follows that \(m n \, R_\beta(m, n)\) is bounded. Naming
\begin{equation*}
r_\beta(m, n) := \frac{1}{m n} \Bigl\{ -\frac{\beta^2 - 6 \beta + 4}{24 \beta^2} \log\Bigl( \frac{1}{m} + \frac{1}{n} \Bigr) + c_\beta + \BigO\Bigl( \frac{1}{m} + \frac{1}{n} \Bigr) \Bigr\} + R_\beta(m, n),
\end{equation*}
we arrive at the claimed result.

Assume now that \(c = \lim_{n \to \infty} \frac{m}{n} \in [0, 1]\) exists, then we define \(c_n := \frac{m}{n} \in (0, 1)\) and we can write \(m = c_n n\). Using this yields
\begin{equation*}
\Bigl( \frac{m}{m + n} \Bigr)^{(m + 2/\beta - 1)/n} \Bigl( \frac{n}{m + n} \Bigr)^{(n + 2/\beta - 1)/m} = \Bigl( 1 + \frac{1}{c_n} \Bigr)^{-c_n + (1 - 2/\beta)/n} (1 + c_n)^{-1/c_n + (1 - 2/\beta)/m}.
\end{equation*}
Now \(\lim_{n \to \infty} (1 + c_n)^{-1/c_n} = (1 + c)^{-1/c}\) (\(= \ez^{-1}\) if \(c = 0\)); furthermore \(\lim_{n \to \infty} (1 + c_n)^{(1 - 2/\beta)/m} = 1\), which is clear if \(c = 0\), irrespective of whether \(m\) is bounded or not, and in the case \(c > 0\) we have \(\lim_{n \to \infty} m = \infty\) and the result follows again; the limit \(\lim_{n \to \infty} (1 + \frac{1}{c_n})^{-c_n} = (1 + \frac{1}{c})^{-c}\) (\(= 1\) if \(c = 0\)) is unremarkable; and concerning the remaining term observe \(c_n n = m \geq 1\), thus \(1 \leq 1 + \frac{1}{c_n} \leq 1 + n\); hence irrespective of the value of \(\beta\), \((1 + \frac{1}{c_n})^{(1 - 2/\beta)/n}\) is sandwiched between \(1\) and \((1 + n)^{(1 - 2/\beta)/n}\), either of which converges to \(1\) as \(n \to \infty\), and therefore also
\begin{equation*}
\Bigl( 1 + \frac{1}{c_n} \Bigr)^{(1 - 2/\beta)/n} \xrightarrow[n \to \infty]{} 1. \qedhere
\end{equation*}
\end{proof}

\begin{Bem}
In the quadratic case (i.e.\ \(m = n\)) Saint Raymond has derived the following asymptotics~\cite[Corollaire~4]{Saintraymond1984}:
\begin{equation*}
(\KugVol{\Schatten{\infty}{}, \beta}{n \times n})^{1/(\beta n^2)} \sim \frac{1}{2} \sqrt{\frac{2 \pi \ez^{3/2}}{\beta n}},
\end{equation*}
where \(a_n \sim b_n\) means \(\lim_{n \to \infty} \frac{a_n}{b_n} = 1\). It is easily seen that Theorem~\ref{sa:volumen_b} reproduces his result when taking \(c = 1\).
\end{Bem}

Recall that a \emph{convex body} is a subset \(K \subset \RZ^d\) (\(d \in \NZ\)) that is compact, convex, and has nonempty interior. A convex body \(K\) is said to be \emph{isotropic} iff \(\vol_d(K) = 1\), it is centred, i.e.,
\begin{equation*}
\int_K x \, \diff x = 0,
\end{equation*}
and
\begin{equation*}
\exists L \in [0, \infty) \forall \theta \in \RZ^d \colon \int_K \langle x, \theta \rangle^2 \, \diff x = L^2 \lVert \theta \rVert_2^2.
\end{equation*}
The last condition can be expressed equivalently as
\begin{equation*}
\exists L \in [0, \infty) \forall i \in [1, d] \colon \int_K x_i^2 \, \diff x = L^2 \quad \wedge \quad \forall i, j \in [1, d] \colon \biggl( i \neq j \Longrightarrow \int_K x_i x_j \, \diff x = 0 \biggr).
\end{equation*}
For convex bodies in \(\CZ^d\) one would have to check the above conditions for the \(2 d\) real components \(\Re(x_1), \Im(x_1), \dotsc, \Re(x_d), \Im(x_d)\). We provide below an equivalent characterization using directly the complex components.

\begin{Lem}\label{lem:komplexe_isotropie}
A centred convex body \(K \subset \CZ^d\) with unit volume is isotropic iff both of the following conditions are satisfied:
\begin{equation}\label{eq:komplexe_momente1}
\exists L \in \RZ \forall i \in [1, d] \colon \int_K \lvert x_i \rvert^2 \, \diff x = L^2
\end{equation}
and
\begin{equation}\label{eq:komplexe_momente2}
\forall i, j \in [1, d], i \neq j \colon \int_K x_i^2 \, \diff x = \int_K x_i x_j \, \diff x = \int_K x_i \konj{x_j} \, \diff x = 0.
\end{equation}
\end{Lem}

\begin{proof}
Let \(i, j \in [1, d]\) with \(i \neq j\), then using \(\Re(x) = \frac{x + \konj{x}}{2}\) and \(\Im(x) = \frac{x - \konj{x}}{2 \ie}\) we can rewrite,
\begin{align}
\int_K \Re(x_i)^2 \, \diff x &= \int_K \frac{1}{4} (x_i^2 + 2 x_i \konj{x_i} + \konj{x_i^2}) \, \diff x = \int_K \frac{\Re(x_i^2)}{2} \, \diff x + \int_K \frac{\lvert x_i \rvert^2}{2} \, \diff x, \label{eq:gem_moment1}\\
\intertext{and similarly,}
\int_K \Re(x_i) \Re(x_j) \, \diff x &= \int_K \frac{\Re(x_i x_j)}{2} \, \diff x + \int_K \frac{\Re(x_i \konj{x_j})}{2} \, \diff x, \label{eq:gem_moment2}\\
\int_K \Im(x_i)^2 \, \diff x &= \int_K \frac{1}{-4} (x_i^2 - 2 x_i \konj{x_i} + \konj{x_i^2}) \, \diff x = -\int_K \frac{\Re(x_i^2)}{2} \, \diff x + \int_K \frac{\lvert x_i \rvert^2}{2} \, \diff x, \label{eq:gem_moment3}\\
\int_K \Im(x_i) \Im(x_j) \, \diff x &= -\int_K \frac{\Re(x_i x_j)}{2} \, \diff x + \int_K \frac{\Re(x_i \konj{x_j})}{2} \, \diff x, \label{eq:gem_moment4}\\
\intertext{and without the restriction \(i \neq j\),}
\int_K \Re(x_i) \Im(x_j) \, \diff x &= \int_K \frac{\Im(x_i x_j)}{2} \, \diff x - \int_K \frac{\Im(x_i \konj{x_j})}{2} \, \diff x. \label{eq:gem_moment5}
\end{align}

\(\Rightarrow\): Let \(K\) be isotropic, then we know from \eqref{eq:gem_moment1} and \eqref{eq:gem_moment3}, for some \(L \in \RZ\),
\begin{equation*}
\int_K \frac{\Re(x_i^2)}{2} \, \diff x + \int_K \frac{\lvert x_i \rvert^2}{2} \, \diff x = L^2 = -\int_K \frac{\Re(x_i^2)}{2} \, \diff x + \int_K \frac{\lvert x_i \rvert^2}{2} \, \diff x;
\end{equation*}
substract the two equations to get
\begin{equation*}
\int_K \Re(x_i^2) \, \diff x = 0,
\end{equation*}
and add them to get
\begin{equation*}
\int_K \lvert x_i \rvert^2 \, \diff x = 2 L^2.
\end{equation*}
In like manner we obtain from \eqref{eq:gem_moment2} and \eqref{eq:gem_moment4} that
\begin{equation*}
\int_K \frac{\Re(x_i x_j)}{2} \, \diff x + \int_K \frac{\Re(x_i \konj{x_j})}{2} \, \diff x = 0 = -\int_K \frac{\Re(x_i x_j)}{2} \, \diff x + \int_K \frac{\Re(x_i \konj{x_j})}{2} \, \diff x,
\end{equation*}
hence
\begin{equation*}
\int_K \Re(x_i x_j) \, \diff x = 0 = \int_K \Re(x_i \konj{x_j}) \, \diff x.
\end{equation*}
Swap \(i\) and \(j\) in Equation~\eqref{eq:gem_moment5}, then \(\Im(x_j \konj{x_i}) = \Im(\konj{x_i \konj{x_j}}) = -\Im(x_i \konj{x_j})\), and therefore
\begin{equation*}
\int_K \frac{\Im(x_i x_j)}{2} \, \diff x - \int_K \frac{\Im(x_i \konj{x_j})}{2} \, \diff x = 0 = \int_K \frac{\Im(x_i x_j)}{2} \, \diff x + \int_K \frac{\Im(x_i \konj{x_j})}{2} \, \diff x,
\end{equation*}
and again from both subtracting and adding there follow
\begin{equation*}
\int_K \Im(x_i x_j) \, \diff x = 0 = \int_K \Im(x_i \konj{x_j}) \, \diff x.
\end{equation*}
Putting all things together yields the claim.

\(\Leftarrow\): Separate the real and imaginary parts in Condition~\eqref{eq:komplexe_momente2} (all are zero), and together with \eqref{eq:komplexe_momente1} substitute them in Equations~\eqref{eq:gem_moment1}\--\eqref{eq:gem_moment5} to get the desired isotropy.
\end{proof}

\begin{Bem}
In relatively recent years interest in the study of complex convex bodies has grown, in the sense that the ambient space \(\CZ^d\) is considered over \(\CZ\), not over \(\RZ\); examples are the works \cite{KolZy2003, KoKoeZy2008, Abardia2012, HuangHeWang2012, KoPaZy2013, HuangHe2015, EllHof2023}. This also leads to a different notion of isotropy; more specifically, a compact convex set \(K \subset \CZ^d\) is called \emph{(complex) isotropic} iff it has unit volume, it is centred, and it possesses a constant \(L \in [0, \infty)\) such that,
\begin{equation*}
\forall \theta \in \CZ^d \colon \int_K \lvert \adj{x} \theta \rvert^2 \, \diff x = L^2 \lVert \theta \rVert_2^2;
\end{equation*}
contrast this to \emph{real isotropy} where the last condition reads,
\begin{equation*}
\forall \theta \in \CZ^d \colon \int_K \Re(\adj{x} \theta)^2 \, \diff x = L^2 \lVert \theta \rVert_2^2.
\end{equation*}
These two notions are not equivalent, the class of complex isotropic bodies is strictly larger; see \cite[Theorem~2.1]{DucHuangHe2014}.

The authors stress that the present article considers \emph{real} isotropy only.
\end{Bem}

\begin{Sa}\label{sa:b_isotropisch}
The volume\-/normalized ball, \((\KugVol{\Schatten{p}{}, \beta}{m \times n})^{-1/(\beta m n)} \, \Kug{\Schatten{p}{}, \beta}{m \times n}\), is isotropic.
\end{Sa}

\begin{proof}
Since \((\KugVol{\Schatten{p}{}, \beta}{m \times n})^{-1/(\beta m n)} \, \Kug{\Schatten{p}{}, \beta}{m \times n}\) differs from \(\Kug{\Schatten{p}{}, \beta}{m \times n}\) only by a dilation, it suffices to consider the latter. Being a norm unit ball, \(\Kug{\Schatten{p}{}, \beta}{m \times n}\) is centred; thence it remains to check the second moments, and we will exploit the unitary invariance of \(\Kug{\Schatten{p}{}, \beta}{m \times n}\) throughout (inherited from the singular values). Let \((i, j) \in m \times n\), then let \(v \in \Orth_{m; \beta}\) swap the first and \(i\)\textsuperscript{th} row and \(u \in \Orth_{n; \beta}\) swap the first and \(j\)\textsuperscript{th} column, effecting
\begin{equation*}
\int_{\Kug{\Schatten{p}{}, \beta}{m \times n}} x_{i, j}^2 \, \diff x = \int_{v \Kug{\Schatten{p}{}, \beta}{m \times n} u} x_{1, 1}^2 \, \diff x = \int_{\Kug{\Schatten{p}{}, \beta}{m \times n}} x_{1, 1}^2 \, \diff x,
\end{equation*}
so the diagonal moments are equal to a common value.

Now let \((i, j), (i', j') \in m \times n\) with \((i, j) \neq (i', j')\). If \(i \neq i'\), define \(v \in \Orth_{m; \beta}\) to multiply the \(i'\)\textsuperscript{th} row by \(-1\); therewith
\begin{equation*}
\int_{\Kug{\Schatten{p}{}, \beta}{m \times n}} x_{i, j} x_{i', j'} \, \diff x = \int_{v \Kug{\Schatten{p}{}, \beta}{m \times n}} x_{i, j} (-x_{i', j'}) \, \diff x = -\int_{\Kug{\Schatten{p}{}, \beta}{m \times n}} x_{i, j} x_{i', j'} \, \diff x,
\end{equation*}
which results in the integral being zero. If \(i = i'\), then \(j \neq j'\) by assumption; and then multiply the \(j'\)\textsuperscript{th} column by \(-1\). Eventually all off\-/diagonal moments are zero, and in the case \(\beta = 1\) this finishes the proof.

In the case \(\beta = 2\) we apply Lemma~\ref{lem:komplexe_isotropie}; performing the same steps as before with the integral \(\int_{\Kug{\Schatten{p}{}, 2}{m \times n}} x_{i, j} \konj{x_{i', j'}} \, \diff x\) we see, for \((i, j) = (i', j')\),
\begin{equation*}
\int_{\Kug{\Schatten{p}{}, 2}{m \times n}} \lvert x_{i, j} \rvert^2 \, \diff x = \int_{\Kug{\Schatten{p}{}, 2}{m \times n}} \lvert x_{1, 1} \rvert^2 \, \diff x,
\end{equation*}
and for \((i, j) \neq (i', j')\),
\begin{equation*}
\int_{\Kug{\Schatten{p}{}, 2}{m \times n}} x_{i, j} \konj{x_{i', j'}} \, \diff x = 0.
\end{equation*}
It remains to show \(\int_{\Kug{\Schatten{p}{}, 2}{m \times n}} x_{1, 1}^2 \, \diff x = 0\), but in order to do so let \(v \in \Orth_{m; 2}\) multiply the first row by \(\ie\), therewith
\begin{equation*}
\int_{\Kug{\Schatten{p}{}, 2}{m \times n}} x_{1, 1}^2 \, \diff x = \int_{\adj{v} \Kug{\Schatten{p}{}, 2}{m \times n}} (\ie x_{1, 1})^2 \, \diff x = -\int_{\Kug{\Schatten{p}{}, 2}{m \times n}} x_{1, 1}^2 \, \diff x,
\end{equation*}
which leads to the desired statement.
\end{proof}

For a convex body \(K \subset \RZ^d\) which is isotropic up to dilation, the isotropy constant \(L_K \in (0, \infty)\) is defined via
\begin{align*}
L_K^2 &= \frac{1}{d \vol_d(K)^{1 + 2/d}} \int_K \lVert x \rVert_2^2 \, \diff x\\
&= \frac{1}{d \vol_d(K)^{2/d}} \, \frac{\int_K \lVert x \rVert_2^2 \, \diff x}{\int_K \diff x}.
\end{align*}

\begin{Thm}\label{sa:lb_beschraenkt}
The isotropy constant \(L_{\Kug{\Schatten{\infty}{}, \beta}{m \times n}}\) remains bounded in \(m\) and \(n\); to be precise,
\begin{equation*}
L_{\Kug{\Schatten{\infty}{}, \beta}{m \times n}}^2 = \frac{1}{2 \pi \ez^{3/2}} \, c_\beta(m, n),
\end{equation*}
where we know
\begin{equation*}
\ez^{1/2} \Bigl( 1 + \frac{C_1}{m n} \Bigr) \leq c_\beta(m, n) \leq \ez^2 \Bigl( 1 + C_2 \Bigl( \frac{1}{m} + \frac{1}{n} \Bigr) \Bigr)
\end{equation*}
for all \(m, n \in \NZ\), with global constants \(C_1, C_2 \in \RZ\).

In particular, consider \(m\) as depending on \(n\), then if \(c := \lim_{n \to \infty} \frac{m}{n} \in [0, 1]\) exists, then
\begin{equation*}
\lim_{n \to \infty} L_{\Kug{\Schatten{\infty}{}, \beta}{m \times n}}^2 = \frac{1}{2 \pi \ez^{3/2}} (1 + c)^{1/(2 c)} \Bigl( 1 + \frac{1}{c} \Bigr)^{c/2},
\end{equation*}
with the same interpretations for the case \(c = 0\) as in Theorem~\ref{sa:volumen_b}.
\end{Thm}

\begin{proof}
Note that in our case \(\langle x, x \rangle = \lVert s(x) \rVert_2^2\) and \(\Ind_{\Kug{\Schatten{\infty}{}, \beta}{m \times n}} = \Ind_{\Kug{\infty, 1}{m}} \circ s\), therefore
\begin{equation*}
\frac{\int_{\Kug{\Schatten{\infty}{}, \beta}{m \times n}} \langle x, x \rangle \, \diff x}{\int_{\Kug{\Schatten{\infty}{}, \beta}{m \times n}} \diff x} = \frac{\int_{\RZ_{> 0}^m} \Ind_{\Kug{\infty, 1}{m}}(s) \lVert s \rVert_2^2 \prod_{i = 1}^m s_i^{\beta (n - m + 1) - 1} \prod_{i < j} \lvert s_i^2 - s_j^2 \rvert^\beta \, \diff s}{\int_{\RZ_{> 0}^m} \Ind_{\Kug{\infty, 1}{m}}(s) \prod_{i = 1}^m s_i^{\beta (n - m + 1) - 1} \prod_{i < j} \lvert s_i^2 - s_j^2 \rvert^\beta \, \diff s};
\end{equation*}
but via the transformation \(s_i \mapsto s_i^{1/2}\) this quotient equals
\begin{equation*}
\frac{\int_{[0, 1]^m} \lVert s \rVert_1 \prod_{i = 1}^m s_i^{\beta (n - m + 1)/2 - 1} \prod_{i < j} \lvert s_i - s_j \rvert^\beta \, \diff s}{\int_{[0, 1]^m} \prod_{i = 1}^m s_i^{\beta (n - m + 1)/2 - 1} \prod_{i < j} \lvert s_i - s_j \rvert^\beta \, \diff s},
\end{equation*}
and after observing that the integrand of the enumerator is symmetric, we can apply a theorem of Aomoto's~\cite{Aomoto1987} and obtain the exact value
\begin{equation*}
\frac{\int_{\Kug{\Schatten{\infty}{}, \beta}{m \times n}} \langle x, x \rangle \, \diff x}{\int_{\Kug{\Schatten{\infty}{}, \beta}{m \times n}} \diff x} = \frac{m n}{m + n + \frac{2}{\beta} - 1}.
\end{equation*}
Using this together with the asymptotic volume from Theorem~\ref{sa:volumen_b}, we have
\begin{equation*}
L_{\Kug{\Schatten{\infty}{}, \beta}{m \times n}}^2 = \frac{m + n}{2 \ez^{3/2} \, \pi (m + n + \frac{2}{\beta} - 1)} \Bigl( \frac{m + n}{m} \Bigr)^{(m + 2/\beta - 1)/(2 n)} \Bigl( \frac{m + n}{n} \Bigr)^{(n + 2/\beta - 1)/(2 m)} \bigl( 1 + r_\beta(m, n) \bigr)^{-2},
\end{equation*}
with \(r_\beta(m, n)\) as in Theorem~\ref{sa:volumen_b}. In order to find suitable bounds, we consider
\begin{equation*}
\begin{split}
\log(c_\beta(m, n)) &:= \log\Bigl( \frac{m + n}{m + n + \frac{2}{\beta} - 1} \Bigr) + \frac{m + \frac{2}{\beta} - 1}{2 n} \log\Bigl( \frac{m + n}{m} \Bigr)\\
&\quad + \frac{n + \frac{2}{\beta} - 1}{2 m} \log\Bigl( \frac{m + n}{n} \Bigr) - 2 \log\bigl(1 + r_\beta(m, n)\bigr).
\end{split}
\end{equation*}
Recall the general estimates for the logarithm: for any \(x > 0\) we have
\begin{equation*}
1 - \frac{1}{x} \leq \log(x) \leq x - 1.
\end{equation*}
On the one hand, together with \(\frac{c_1}{m n} \log(\frac{1}{m} + \frac{1}{n}) \leq \log(1 + r_\beta(m, n)) \leq \frac{c_2}{m n}\) with suitable global constants \(c_1, c_2 \in \RZ\), this yields the lower bound
\begin{align*}
\log(c_\beta(m, n)) &\geq 1 - \frac{m + n + \frac{2}{\beta} - 1}{m + n} + \frac{m + \frac{2}{\beta} - 1}{2 n} \Bigl( 1 - \frac{m}{m + n} \Bigr) + \frac{n + \frac{2}{\beta} - 1}{2 m} \Bigl( 1 - \frac{n}{m + n} \Bigr) - \frac{2 c_2}{m n} \\
&= \frac{1}{2} - \frac{2 c_2}{m n}.
\end{align*}
On the other hand for an upper bound we get
\begin{align*}
\log(c_\beta(m, n)) &\leq \frac{m + n}{m + n + \frac{2}{\beta} - 1} - 1 + \frac{m + \frac{2}{\beta} - 1}{2 n} \Bigl( \frac{m + n}{m} - 1 \Bigr)\\
&\quad + \frac{n + \frac{2}{\beta} - 1}{2 m} \Bigl( \frac{m + n}{n} - 1 \Bigr) - \frac{2 c_1}{m n} \log\Bigl( \frac{1}{m} + \frac{1}{n} \Bigr)\\
&= \frac{m + n}{m + n + \frac{2}{\beta} - 1} + \frac{2 - \beta}{2 \beta m} + \frac{2 - \beta}{2 \beta n} - \frac{2 c_1}{m n} \log\Bigl( \frac{1}{m} + \frac{1}{n} \Bigr);
\end{align*}
obviously \(\frac{m + n}{m + n + \frac{2}{\beta} - 1} \leq \frac{m + n}{m + n - 1} \leq 2\), and it remains to show that \(\frac{1}{m} + \frac{1}{n}\) dominates \(\frac{1}{m n} \log(\frac{1}{m} + \frac{1}{n})\) when at least \(n\) grows indefinitely. But we have
\begin{equation*}
\frac{\frac{1}{m n} \log(\frac{1}{m} + \frac{1}{n})}{\frac{1}{m} + \frac{1}{n}} = \frac{\log(m + n)}{m + n} - \frac{\log(m n)}{m + n};
\end{equation*}
clearly the first summand is bounded, and via the inequality of arithmetic and geometric means we know \(0 \leq \log(m n) \leq 2 \log(m + n) - 2 \log(2)\) and therefore also the second summand is bounded.

The value of \(\lim_{n \to \infty} L_{\Kug{\Schatten{\infty}{}, \beta}{m \times n}}^2\) in the case \(\lim_{n \to \infty} \frac{m}{n} = c\) is a direct consequence of the corresponding values of \(\lim_{n \to \infty} (\beta n)^{1/2} (\KugVol{\Schatten{\infty}{}, \beta}{m \times n})^{1/(\beta m n)}\) given in Theorem~\ref{sa:volumen_b}.
\end{proof}

\begin{Bem}
Although K\"onig, Meyer, Pajor~\cite{KMP1998} claim that their theorem includes the case \(p = \infty\) (that is, the spectral norm), their method of proof actually does not cover that, because the integral in question contains the exponential \(\ez^{-\lVert x \rVert_p^p}\) which in the limit case equals just \(\Ind_{\Kug{\infty, 1}{d}}(x)\), being an exponential no longer and rendering their integral transfrom moot.
\end{Bem}

Define a \emph{star body} to be a compact set \(K \subset \RZ^d\) with zero in its interior which is star\-/shaped w.r.t.\ zero, i.e., \(t K \subset K\) for all \(t \in [0, 1]\), and whose Minkowski\-/functional \(\lvert \cdot \rvert_K\) is continuous. Any convex body with zero in its interior is a star body. For such a star body \(K\) define the associated \emph{(normalized) cone measure} \(\KegM{}{} = \KegM{K}{}\) on its boundary \(\Rand{K}\) by
\begin{equation*}
\KegM{}{}(A) := \frac{\vol_d\bigl( \{t x \Colon t \in [0, 1], x \in A\} \bigr)}{\vol_d(K)}
\end{equation*}
for any measurable \(A \subset \Rand{K}\); it is the unique measure on \(\Rand{K}\) which satisfies the following polar integration formula (see~\cite[Proposition~1]{NR2003}): for any \(f \colon \RZ^d \to \RZ\) measurable and nonnegative,
\begin{equation}\label{eq:kegmasz_polarint}
\int_{\RZ^d} f(x) \, \diff x = d \vol_d(K) \int_{\RZ_{\geq 0}} \int_{\Rand{K}} f(r \theta) \, \diff\KegM{}{}(\theta) \, r^{d - 1} \, \diff r.
\end{equation}
A useful identity for evaluating integrals w.r.t.\ the cone measure can be obtained by integrating \(x \mapsto f(\frac{x}{\lvert x \rvert_K})\) over \(K\), to wit,
\begin{equation}\label{eq:kegmasz_int}
\int_{\Rand{K}} f(\theta) \, \diff\KegM{}{}(\theta) = \frac{1}{\vol_d(K)} \int_K f\Bigl( \frac{x}{\lvert x \rvert_K} \Bigr) \, \diff x.
\end{equation}
The surface measure on \(\Rand{K}\) is defined to be the \((d - 1)\)\=/dimensional Hausdorff measure restricted to \(\Rand{K}\). The two measures are related according to the following statement.

\begin{Lem}\label{lem:kegel_gg_hausdorff}
Let \(\nu(x)\) be the outer unit normal vector in \(x \in \Rand K\), then for \(\Hausd^{d - 1}\)\=/almost every \(x \in \Rand K\),
\begin{equation*}
\diff\!\Hausd^{d - 1}(x) = \frac{d \vol_d(K)}{\langle \nu(x), x \rangle} \, \diff\KegM{}{}(x).
\end{equation*}
\end{Lem}

\begin{proof}
This is a rewriting of Naor and Romik \cite[Lemma~1]{NR2003}. Adapting their notation, note that \(\Hausd^{d - 1} = \operatorname{area}(\Rand{K}) \sigma_K\); it remains to treat \(\lVert \nabla \lvert \cdot \rvert_K(x) \rVert_2\). But in their proof they derive \(\langle \nabla \lvert \cdot \rvert_K(x), x \rangle = 1\), and therewith
\begin{equation*}
\lVert \nabla \lvert \cdot \rvert_K(x) \rVert_2 = \frac{\lVert \nabla \lvert \cdot \rvert_K(x) \rVert_2}{\langle \nabla \lvert \cdot \rvert_K(x), x \rangle} = \frac{1}{\bigl\langle \frac{\nabla \lvert \cdot \rvert_K(x)}{\lVert \nabla \lvert \cdot \rvert_K(x) \rVert_2}, x \bigr\rangle},
\end{equation*}
from which the claim follows because of \(\nu(x) = \frac{\nabla \lvert \cdot \rvert_K(x)}{\lVert \nabla \lvert \cdot \rvert_K(x) \rVert_2}\).
\end{proof}

\begin{Bem}
Rewritten as \(\vol_d(K) \, \diff\KegM{}{}(x) = \frac{1}{d} \langle \nu(x), x \rangle \, \diff\!\Hausd^{d - 1}(x)\), the formula has a nice elementary interpretation: fix a point \(x \in \Rand{K}\), then this determines an ``infinitesimal cone'' with apex \(o\) and an ``infinitesimal base'' tangent to \(\Rand{K}\) at \(x\); the height of this cone equals \(\lVert \langle \nu(x), x \rangle \nu(x) \rVert_2 = \langle \nu(x), x \rangle\) (the length of the orthogonal projection of \(x\) onto \(\RZ \nu(x)\)), the base measures \(\diff\!\Hausd^{d - 1}(x)\), and therefore the cone has volume \(\frac{1}{d} \langle \nu(x), x \rangle \, \diff\!\Hausd^{d - 1}(x)\). But by the very definition of the cone measure this also equals \(\vol_d(K) \, \diff\KegM{}{}(x)\).
\end{Bem}

As a consequence of Lemma~\ref{lem:kegel_gg_hausdorff}, if \(x \mapsto \langle \nu(x), x \rangle\) is constant on \(\Rand K\), then the normalized measures coincide. This is known to be the case for \(\Sph{p, \beta}{m}\) iff \(p \in \{1, 2, \infty\}\), going back to~\cite[p.~1314]{RR1991}. Let us fix the notation \(\KegM{p, \beta}{m} := \KegM{\Kug{p, \beta}{m}}{}\) and \(\KegM{\Schatten{p}{}, \beta}{m \times n} := \KegM{\Kug{\Schatten{p}{}, \beta}{m \times n}}{}\).

\begin{Lem}\label{lem:potenz_ableitung}
Let \(p \in \RZ\), then the trace power map
\begin{equation*}
\tau_p \colon \left\{ \begin{aligned} \Pos_{m, \beta} &\to \RZ \\ x &\mapsto \Spur(x^p) \end{aligned} \right.
\end{equation*}
is continuously differentiable and for any \(x \in \Pos_{m, \beta}\),
\begin{equation*}
\nabla \tau_p(x) = p x^{p - 1}.
\end{equation*}
\end{Lem}

\begin{proof}
\textit{Case \(p \in \NZ_0\):} Continuous differentiability is obvious, as is the gradient, for \(p \in \{0, 1\}\). So let \(p \geq 2\), then the map \(f \colon \Pos_{m, \beta} \to \Pos_{m, \beta}\), \(x \mapsto x^p\), is continuously differentiable with derivative \(\diff f(x)(h) = \sum_{k = 1}^p x^{k - 1} h x^{p - k}\) for any \(x \in \Pos_{m, \beta}\) and \(h \in \Sym_{m, \beta}\). This implies, because of \(\tau_p = \Spur \circ f\),
\begin{align*}
\diff \tau_p(x)(h) &= \Spur(\diff f(x)(h)) = \Spur\biggl( \sum_{k = 1}^p x^{k - 1} h x^{p - k} \biggr)\\
&= \sum_{k = 1}^p \Spur(x^{k - 1} h x^{p - k}) = \sum_{k = 1}^p \Spur(x^{p - 1} h)\\
&= p \Spur(x^{p - 1} h) = \langle p x^{p - 1}, h \rangle.
\end{align*}
(Note \(\Spur(a b) \in \RZ\) for any \(a, b \in \Sym_{m, \beta}\).)

\textit{Case \(p\) arbitrary:} Write \(x^p = \ez^{p \log(x)}\), then this already suggests continuous differentiability, and expanding the exponential we have
\begin{equation*}
\tau_p(x) = \sum_{k = 0}^\infty \frac{p^k}{k!} \Spur(\log(x)^k).
\end{equation*}
Call \(f_k(x) := \Spur(\log(x)^k)\), then we have to determine \(\nabla f_k(x)\) and show that \(\sum_{k \geq 0} \frac{p^k}{k!} \nabla f_k\) converges locally uniformly on \(\Pos_{m, \beta}\); this will yield \(\nabla \tau_p = \sum_{k = 0}^\infty \frac{p^k}{k!} \nabla f_k\).

From the case of integer powers and the chain rule we know \(\diff f_k(x)(h) = \langle k \log(x)^{k - 1}, \diff \log(x)(h) \rangle\) for any \(x \in \Pos_{m, \beta}\) and \(h \in \Sym_{m, \beta}\). In order to determine \(\diff \log(x)(h)\) consider the exponential map again, that is, \(\ez^x = \sum_{k = 0}^\infty \frac{1}{k!} x^k\). Call \(g_k(x) := x^k\), then each \(g_k\) is differentiable with derivative \(\diff g_k(x)(h) = \sum_{l = 1}^k x^{l - 1} h x^{k - l}\); therefore \(\lVert \diff g_k(x)(h) \rVert_{\Schatten{2}{}} \leq k \lVert x \rVert_{\Schatten{2}{}}^{k - 1} \lVert h \rVert_{\Schatten{2}{}}\) (any submultiplicative norm would do) and for the operator norm we get \(\lVert \diff g_k(x) \rVert \leq k \lVert x \rVert_{\Schatten{2}{}}^{k - 1}\). Now given any \(R \in (0, \infty)\) we find
\begin{equation*}
\sum_{k = 0}^\infty \frac{1}{k!} \lVert \diff g_k(x) \rVert \leq \sum_{k = 0}^\infty \frac{1}{k!} \, k \lVert x \rVert_{\Schatten{2}{}}^{k - 1} \leq \sum_{k = 0}^\infty \frac{1}{k!} R^k = \ez^R < \infty
\end{equation*}
for all \(x \in \Sym_{m, \beta}\) with \(\lVert x \rVert_{\Schatten{2}{}} \leq R\); thus \(\sum_{k \geq 0} \frac{1}{k!} \diff g_k\) is locally uniformly convergent on \(\Sym_{m, \beta}\) and hence \(\diff \exp = \sum_{k = 0}^\infty \frac{1}{k!} \diff g_k\). From that and from the identity \(x = \ez^{\log(x)}\) follows, for any \(x \in \Pos_{m, \beta}\) and \(h \in \Sym_{m, \beta}\),
\begin{equation*}
h = \diff \exp(\log(x))(\diff \log(x)(h)) = \sum_{k = 0}^\infty \frac{1}{k!} \sum_{l = 1}^k \log(x)^{l - 1} \diff \log(x)(h) \log(x)^{k - l}.
\end{equation*}
Fix a spectral decomposition \(x = u \diag(\lambda) \adj{u}\) with \(\lambda \in \RZ_{> 0}^m\) and \(u \in \Orth_{m; \beta}\) and define \(\xi := \adj{u} h u\) and \(\eta := \adj{u} \diff \log(x)(h) u\), then the above equation is equivalent to
\begin{equation*}
\sum_{k = 0}^\infty \frac{1}{k!} \sum_{l = 1}^k \log(\diag(\lambda))^{l - 1} \eta \log(\diag(\lambda))^{k - l} = \xi.
\end{equation*}
Let \(i \in [1, n]\), then the \((i, i)\)\-/component of the equation reads
\begin{equation*}
\xi_{i, i} = \sum_{k = 0}^\infty \frac{1}{k!} \sum_{l = 1}^k \log(\lambda_i)^{l - 1} \eta_{i, i} \log(\lambda_i)^{k - l} = \eta_{i, i} \sum_{k = 0}^{\infty} \frac{1}{k!} \, k \log(\lambda_i)^{k - 1} = \eta_{i, i} \ez^{\log(\lambda_i)} = \lambda_i \eta_{i, i}
\end{equation*}
which yields the solution
\begin{equation*}
\eta_{i, i} = \frac{\xi_{i, i}}{\lambda_i}.
\end{equation*}
Therewith we can continue,
\begin{align*}
\diff f_k(x)(h) &= \langle k \log(x)^{k - 1}, \diff \log(x)(h) \rangle\\
&= \langle k \log(\diag(\lambda))^{k - 1}, \eta \rangle = \sum_{i = 1}^m k \log(\lambda_i)^{k - 1} \, \eta_{i, i}\\
&= \sum_{i = 1}^m k \log(\lambda_i)^{k - 1} \inv{\lambda_i} \, \xi_{i, i} = \langle k \log(\diag(\lambda))^{k - 1} \inv{\diag(\lambda)}, \xi \rangle\\
&= \langle k \log(x)^{k - 1} \inv{x}, h \rangle,
\end{align*}
and so \(\nabla f_k(x) = k \log(x)^{k - 1} \inv{x}\). Finally let \(R \in (1, \infty)\), then we have, for all \(x \in \Pos_{m, \beta}\) such that \(\frac{1}{R} \leq \lambda_m(x) \leq \lambda_1(x) \leq R\),
\begin{equation*}
\lVert \nabla f_k(x) \rVert_{\Schatten{2}{}} \leq k \lVert \log(x) \rVert_{\Schatten{2}{}}^{k - 1} \lVert \inv{x} \rVert_{\Schatten{2}{}} \leq k (m^{1/2} \log(R))^{k - 1} (m^{1/2} R),
\end{equation*}
and since
\begin{equation*}
\sum_{k = 0}^\infty \frac{p^k}{k!} \, k (m^{1/2} \log(R))^{k - 1} (m^{1/2} R) = p m^{1/2} R \ez^{p m^{1/2} \log(R)} = p m^{1/2} \, R^{p m^{1/2} + 1} < \infty,
\end{equation*}
the series \(\sum_{k \geq 0} \frac{p^k}{k!} \nabla f_k\) converges locally uniformly on \(\Pos_{m, \beta}\) as desired, so we obtain, for any \(x \in \Pos_{m, \beta}\),
\begin{equation*}
\nabla \tau_p(x) = \sum_{k = 0}^\infty \frac{p^k}{k!} k \log(x)^{k - 1} \inv{x} = p \ez^{p \log(x)} \inv{x} = p x^{p - 1}. \qedhere
\end{equation*}
\end{proof}

\begin{Thm}\label{sa:tildes_volumen}
For \(\Hausd^{\beta m n - 1}\)\-/almost all \(x \in \Sph{\Schatten{p}{}, \beta}{m, n}\),
\begin{equation*}
\langle \nu(x), x \rangle = \begin{cases} \lVert x \rVert_{\Schatten{2 p - 2}{}}^{1 - p} & \textit{if } p < \infty, \\ 1 & \textit{if } p = \infty; \end{cases}
\end{equation*}
in particular the cone measure \(\KegM{\Schatten{p}{}, \beta}{m \times n}\) coincides with the normalized Hausdorff\-/measure on \(\Sph{\Schatten{p}{}, \beta}{m, n}\) iff \(p \in \{1, 2, \infty\}\). Thereby also,
\begin{equation*}
\Hausd^{\beta m n - 1}(\Sph{\Schatten{p}{}, \beta}{m, n}) = \begin{cases} \beta m^{3/2} n \KugVol{\Schatten{1}{}, \beta}{m \times n} & \text{if } p = 1, \\ \beta m n \KugVol{\Schatten{p}{}, \beta}{m \times n} & \text{if } p \in \{2, \infty\}. \end{cases}
\end{equation*}
\end{Thm}

\begin{proof}
\textit{Case \(p < \infty\):} Define the map
\begin{equation*}
g \colon \left\{ \begin{aligned} \KZ^{m \times n} &\to \RZ \\ x &\mapsto \Spur((x \adj{x})^{p/2}) \end{aligned} \right.,
\end{equation*}
then \(\Sph{\Schatten{p}{}, \beta}{m, n} = \inv{g}\{1\}\), and via Lemma~\ref{lem:potenz_ableitung} \(g\) is differentiable at any \(x \in \KZ^{m \times n}\) such that \(x \adj{x} \in \Pos_{m, \beta}\), which is almost everywhere, and then, for any \(h \in \KZ^{m \times n}\),
\begin{align*}
\diff g(x)(h) &= \Bigl\langle \frac{p}{2} (x \adj{x})^{p/2 - 1}, h \adj{x} + x \adj{h} \Bigr\rangle\\
&= \langle p (x \adj{x})^{p/2 - 1} x, h \rangle,
\end{align*}
which means
\begin{equation*}
\nabla g(x) = p (x \adj{x})^{p/2 - 1} x.
\end{equation*}
From this follow, for almost any \(x \in \Sph{\Schatten{p}{}, \beta}{m, n}\), both
\begin{align*}
\langle \nabla g(x), x \rangle &= p \langle (x \adj{x})^{p/2 - 1} x, x \rangle = p \langle (x \adj{x})^{p/2 - 1} x \adj{x}, I_m \rangle\\
&= p \langle (x \adj{x})^{p/2}, I_m \rangle = p \Spur((x \adj{x})^{p/2}) = p,
\end{align*}
and
\begin{align*}
\lVert \nabla g(x) \rVert_{\Schatten{2}{}}^2 &= p^2 \langle (x \adj{x})^{p/2 - 1} x, (x \adj{x})^{p/2 - 1} x \rangle = p^2 \langle (x \adj{x})^{p - 2} x \adj{x}, I_m \rangle\\
&= p^2 \Spur((x \adj{x})^{p - 1}) = p^2 \lVert x \rVert_{\Schatten{2 p - 2}{}}^{2 p - 2},
\end{align*}
which imply
\begin{equation*}
\langle \nu(x), x \rangle = \frac{\langle \nabla g(x), x \rangle}{\lVert \nabla g(x) \rVert_{\Schatten{2}{}}} = \frac{1}{\lVert x \rVert_{\Schatten{2 p - 2}{}}^{p - 1}},
\end{equation*}
as claimed.

\textit{Case \(p = \infty\):} Define \(D := \{x \in \KZ^{m \times n} \Colon s_2(x) < 1\}\), then \(D\) is an open subset of \(\KZ^{m \times n}\), and we can write \(\Sph{\Schatten{\infty}{}, \beta}{m, n} = D \cap \{x \in \KZ^{m \times n} \Colon \det(I_m - x \adj{x}) = 0\}\) up to the null set \(\{x \in \KZ^{m \times n} \Colon s_1(x) = s_2(x) = 1\}\). Define the map
\begin{equation*}
g \colon \left\{ \begin{aligned} D &\to \RZ \\ x &\mapsto \det(I_m - x \adj{x}) \end{aligned} \right.,
\end{equation*}
then \(g\) is continuously differentiable, \(\Sph{\Schatten{\infty}{}, \beta}{m, n} = \inv{g}\{0\}\), and for any \(x, h \in \KZ^{m \times n}\) the derivative is
\begin{align*}
\diff g(x)(h) &= \Spur\bigl( \Adj(I_d - x \adj{x}) (-h \adj{x} - x \adj{h}) \bigr)\\
&= -2 \langle \Adj(I_m - x \adj{x}) x, h \rangle,
\end{align*}
where \(\Adj\) denotes the adjugate of a matrix, and hence,
\begin{equation*}
\nabla g(x) = -2 \Adj(I_m - x \adj{x}) x.
\end{equation*}
For any \(x \in \KZ^{m \times n}\) with singular value decomposition \(x = v \diag(s(x)) \adj{u}\) we have
\begin{equation*}
\Adj(I_m - x \adj{x}) = \sum_{i = 1}^m \biggl( \prod_{\substack{j = 1 \\ j \neq i}}^m \bigl( 1 - s_j(x)^2 \bigr) \biggr) v_i \adj{v_i},
\end{equation*}
where \(v_1, \dotsc, v_m \in \KZ^m\) are the columns of \(v\); in particular for all \(x \in \Sph{\Schatten{\infty}{}, \beta}{m, n}\) we get
\begin{equation*}
\Adj(I_m - x \adj{x}) x = \biggl( \prod_{j = 2}^m \bigl( 1 - s_j(x)^2 \bigr) \biggr) v_1 \adj{(\adj{x} v_1)} = \alpha v_1 \adj{u_1},
\end{equation*}
with \(\alpha := \prod_{j = 2}^m (1 - s_j(x)^2)\) and \(u_1 \in \KZ^n\) the first column of \(u\). Thus \(\nabla g(x) \neq 0\) for all \(x \in \Sph{\Schatten{\infty}{}, \beta}{m, n}\) which have simple maximal singular value \(1\), and this means that \(\Sph{\Schatten{\infty}{}, \beta}{m, n}\) has a continuous outer unit normal vector field almost everywhere and up to sign it is given by \(\frac{\nabla g(x)}{\lVert \nabla g(x) \rVert_{\Schatten{2}{}}}\) at \(x \in \Sph{\Schatten{\infty}{}, \beta}{m, n}\).

Let \(x \in \Sph{\Schatten{\infty}{}, \beta}{m, n}\) with simple singular value \(1\), then
\begin{align*}
\langle \nabla g(x), x \rangle &= -2 \alpha \langle v_1 \adj{u_1}, x \rangle\\
&= -2 \alpha \langle v_1, x u_1 \rangle\\
&= -2 \alpha \langle v_1, v_1 \rangle = -2 \alpha,
\end{align*}
and similarly,
\begin{align*}
\lVert \nabla g(x) \rVert_{\Schatten{2}{}}^2 &= 4 \alpha^2 \langle v_1 \adj{u_1}, v_1 \adj{u_1} \rangle\\
&= 4 \alpha^2 \langle \adj{v_1} v_1, \adj{u_1} u_1 \rangle\\
&= 4 \alpha^2 \langle 1, 1 \rangle = 4 \alpha^2,
\end{align*}
and this yields
\begin{equation*}
\frac{\lvert \langle \nabla g(x), x \rangle \rvert}{\lVert \nabla g(x) \rVert_{\Schatten{2}{}}} = 1.
\end{equation*}
This also implies that the cone and normalized Hausdorff\-/measures conincide.

In the case \(p = 1\) we have
\begin{equation*}
\Spur((x \adj{x})^{p - 1}) = \Spur(I_m) = m
\end{equation*}
which leads to \(\langle \nu(x), x \rangle = m^{-1/2}\), which is constant; and in the case \(p = 2\) we see
\begin{equation*}
\lVert x \rVert_{\Schatten{2 p - 2}{}}^{p - 1} = \lVert x \rVert_{\Schatten{2}{}} = 1
\end{equation*}
for all \(x \in \Sph{\Schatten{2}{}, \beta}{m, n}\), which is constant too. (Also notice that \(\Schatten{2, \beta}{m \times n}\) and \(\ellpe{2, 1}{\beta m n}\) are isomorphic Euclidean spaces.) For \(p \notin \{1, 2, \infty\}\) it suffices to show that \(\lVert \cdot \rVert_{\Schatten{2 p - 2}{}}^{2 p - 2}\) attains at least two different values on \(\Sph{\Schatten{p}{}, \beta}{m, n}\) at points with full rank, since then by continuity it cannot be constant outside of a null set. But this is achieved with, say, \(m^{-1/p} \diag(1, \dotsc, 1)\) (\(m\) ones) and \((m - \frac{1}{2})^{-1/p} \diag(1, \dotsc, 1, 2^{-1/p})\) (\(m - 1\) ones).

The volume of \(\Sph{\Schatten{p}{}, \beta}{m, n}\) follows from Lemma~\ref{lem:kegel_gg_hausdorff} and the concrete values of \(\langle \nu(x), x \rangle\).
\end{proof}

The ball \(\Kug{\Schatten{\infty}{}, \beta}{m \times n}\) can be described without direct reference to singular values, as the following proposition states.

\begin{Sa}\label{sa:b_gleich_spektralkugel}
The following identity holds true,
\begin{equation*}
\Kug{\Schatten{\infty}{}, \beta}{m \times n} = \{x \in \KZ^{m \times n} \Colon x \adj{x} \leq I_m\}.
\end{equation*}
\end{Sa}

\begin{proof}
Call \(B := \{x \in \KZ^{m \times n} \Colon x \adj{x} \leq I_m\}\). Let \(x \in \KZ^{m \times n}\), then \(x \in B\) iff for all \(\xi \in \KZ^m\), \(\adj{\xi} x \adj{x} \xi \leq \adj{\xi} \xi\), equivalently \(\lVert \adj{x} \xi \rVert_2^2 \leq \lVert \xi \rVert_2^2\), this itself means \(s_1(\adj{x}) = s_1(x) \leq 1\), which is equivalent to \(x \in \Kug{\Schatten{\infty}{}, \beta}{m \times n}\).
\end{proof}

The ``boundary'' \(S_\beta^{m, n} := \{x \in \KZ^{m \times n} \Colon x \adj{x} = I_m\}\) is easy to identify: a close glance at the definition of the Stiefel manifold (see beginning of Section~\ref{sec:auxiliar}) reveals \(S_\beta^{m, n} = \{\adj{x} \Colon x \in \Orth_{n, m; \beta}\} = \{\trapo{x} \Colon x \in \Orth_{n, m; \beta}\}\). Also note the chain of inclusions \(S_\beta^{m, n} \subset \Sph{\Schatten{\infty}{}, \beta}{m, n} \subset \Kug{\Schatten{\infty}{}, \beta}{m \times n}\) (concerning the first inclusion observe \(s(x) = (1)_{i \leq m}\) for any \(x \in S_\beta^{m, n}\)).

\begin{Sa}\label{sa:skalierte_volumen}
Let \(r \in \KZ^{m \times m}\) be regular, then the following hold true.
\begin{compactenum}
\item \(\{x \in \KZ^{m \times n} \Colon x \adj{x} \leq r \adj{r}\} = r \Kug{\Schatten{\infty}{}, \beta}{m \times n}\), and \(\vol_{\beta m n}(r \Kug{\Schatten{\infty}{}, \beta}{m \times n}) = \lvert \det(r) \rvert^{\beta n} \KugVol{\Schatten{\infty}{}, \beta}{m \times n}\).
\item \(\Rand{(r \Kug{\Schatten{\infty}{}, \beta}{m \times n})} = r \Sph{\Schatten{\infty}{}, \beta}{m \times n}\), and
\begin{equation*}
\Hausd^{\beta m n - 1}(r \Sph{\Schatten{\infty}{}, \beta}{m, n}) = \frac{\lvert \det(r) \rvert^{\beta (n - 1)} \Hausd^{\beta m - 1}(r \Sph{2, \beta}{m})}{\Hausd^{\beta m - 1}(\Sph{2, \beta}{m})} \Hausd^{\beta m n - 1}(\Sph{\Schatten{\infty}{}, \beta}{m, n}).
\end{equation*}
\item \(\{x \in \KZ^{m \times n} \Colon x \adj{x} = r \adj{r}\} = r S_\beta^{m, n}\), and
\begin{equation*}
\Hausd^d(r S_\beta^{m, n}) = 2^{-\beta m (m - 1)/4} \lvert \det(r) \rvert^{\beta (n - m + 1) - 1} \prod_{1 \leq j < i \leq m} (s_j(r)^2 + s_i(r)^2)^{\beta/2} \cdot \OrthVol_{n, m; \beta},
\end{equation*}
where \(d = \frac{\beta m (2 n - m + 1)}{2} - m\).
\end{compactenum}
\end{Sa}

\begin{proof}
Before proving the volumes, we argue the claimed set identities, and we fix the map \(f \colon \KZ^{m \times n} \to \KZ^{m \times n}\), \(x \mapsto r x\). \(f\) is linear and bijective, hence a homeomorphism, and this immediately gives \(\Rand{(r \Kug{\Schatten{\infty}{}, \beta}{m \times n})} = r \Rand{\Kug{\Schatten{\infty}{}, \beta}{m \times n}} = r \Sph{\Schatten{\infty}{}, \beta}{m \times n}\). Concerning \(r \Kug{\Schatten{\infty}{}, \beta}{m \times n}\) note that \(x \adj{x} \leq r \adj{r}\) iff \(\adj{\xi} x \adj{x} \xi \leq \adj{\xi} r \adj{r} \xi\) for all \(\xi \in \KZ^m\); now substitute \(\eta := \adj{r} \xi\), which is bijective, and then \(\adj{\eta} \inv{r} x \adj{x} \inv{(\adj{r})} \eta \leq \adj{\eta} \eta\) for all \(\eta \in \KZ^m\), which means \((\inv{r} x) \adj{(\inv{r} x)} \leq I_m\) and so, by Proposition~\ref{sa:b_gleich_spektralkugel}, \(\inv{r} x \in \Kug{\Schatten{\infty}{}, \beta}{m \times n}\), hence the claim. Concerning \(r S_\beta^{m, n}\), note \(x \adj{x} = r \adj{r}\) iff \(\inv{r} x \adj{x} \inv{(\adj{r})} = I_m\), and the claim follows.

We now turn to the volumes. The general tool is to calculate the determinant of \(f\) restricted to the tangent spaces of the respective manifolds; to that end we are going to choose orthonormal bases as fits our need.

\begin{asparaenum}
\item
Fix a singular value decomposition \(r = v \diag(s(r)) \adj{u}\). Here \(f\) is acting on the whole space \(\KZ^{m \times n}\) which is its own tangent space, and we choose as basis elements the matrices \(u e_i \trapo{e_j}\), additionally \(\ie u e_i \trapo{e_j}\) in the complex case, then
\begin{equation*}
f(u e_i \trapo{e_j}) = r u e_i \trapo{e_j} = s_i(r) v e_i \trapo{e_j}, \quad f(\ie u e_i \trapo{e_j}) = \ie r u e_i \trapo{e_j} = s_i(r) \ie v e_i \trapo{e_j}
\end{equation*}
for all \(i \in [1, m], j \in [1, n]\); and because also the \(v e_i \trapo{e_j}\)\---and \(\ie v e_i \trapo{e_j}\)\---together constitute an orthonormal basis of \(\KZ^{m \times n}\), we get
\begin{equation*}
\det(f) = \biggl( \prod_{i = 1}^m \prod_{j = 1}^n s_i(r) \biggr)^\beta = \biggl( \prod_{i = 1}^m s_i(r) \biggr)^{\beta n} = \lvert \det(r) \rvert^{\beta n},
\end{equation*}
concluding this part.

\item
From the proof of Theorem~\ref{sa:tildes_volumen} we know \(\nu(x) = v_1 \adj{u_1} = v e_1 \trapo{e_1} \adj{u}\) for \(x \in \Sph{\Schatten{\infty}{}, \beta}{m, n}\) with singular value decomposition \(x = v \diag(s(x)) \adj{u}\) such that \(s_1(x) = 1\) is isolated. Therefore an orthonormal base of the tangent space at \(x\) is given by \(v_i \adj{u_j}\) for all \((i, j) \in [1, m] \times [1, n]\) except \((i, j) = (1, 1)\)\---in the complex case additionally \(\ie v_i \adj{u_j}\) for all \((i, j) \in [1, m] \times [1, n]\). Because \(u\), \(v\) depend on \(x\) it makes no sense to fix a decomposition of \(r\) beforehand. Instead, given \(x\), define \(t := \adj{v} r v \in \KZ^{m \times m}\); then \(r v_i = \sum_{k = 1}^m t_{k, i} v_k\) for any \(i \in [1, m]\). Let \(j \in [1, n]\) and \(i \in [1, m]\), then
\begin{equation*}
f(v_i \adj{u_j}) = r v_i \adj{u_j} = \sum_{k = 1}^m t_{k, i} v_k \adj{u_j} = \sum_{k = 1}^m \Re(t_{k, i}) v_k \adj{u_j} + \sum_{k = 1}^m \Im(t_{k, i}) \ie v_k \adj{u_j},
\end{equation*}
and specially for the complex case,
\begin{equation*}
f(\ie v_i \adj{u_j}) = r v_i \adj{u_j} = -\sum_{k = 1}^m \Im(t_{k, i}) v_k \adj{u_j} + \sum_{k = 1}^m \Re(t_{k, i}) \ie v_k \adj{u_j}.
\end{equation*}
This means that for any \(j \in [1, n]\), the subspace spanned by \(v_i \adj{u_j}\) with \(i \in [1, m]\)\---additionally \(\ie v_i \adj{u_j}\) in the complex case\---is \(f\)\=/invariant and with respect to that orthonormal basis the matrix of \(f\) is given by \(t\)\---by \(\bigl( \begin{smallmatrix} \Re(t) & -\Im(t) \\ \Im(t) & \Re(t) \end{smallmatrix} \bigr)\) in the complex case. Also note that for distinct \(j\) those subspaces are mutually orthogonal. Thus for any \(j \in [2, n]\) the associated subspace makes the contribution \(\lvert \det(t) \rvert^\beta = \lvert \det(r) \rvert^\beta\) to the functional determinant, yielding a total of \(\lvert \det(r) \rvert^{\beta (n - 1)}\).

For \(j = 1\) though we lack \(f(v_1 \adj{u_1})\), therefore the matrix of \(f\) is the (\(\beta m \times (\beta m - 1)\))\-/matrix \(\mathbf{t} \bigl( \begin{smallmatrix} 0 \\ I_{\beta m - 1} \end{smallmatrix} \bigr)\), that is \(\mathbf{t}\) sans the first column; here we mean \(\mathbf{t} := t\) in the real case and \(\mathbf{t} := \bigl( \begin{smallmatrix} \Re(t) & -\Im(t) \\ \Im(t) & \Re(t) \end{smallmatrix} \bigr)\) in the complex case. The contribution to the functional determinant in this case is the Gramian
\begin{equation*}
G(x) := \det\left( (0, I_{\beta m - 1}) \trapo{\mathbf{t}} \mathbf{t} \begin{pmatrix} 0 \\ I_{\beta m - 1} \end{pmatrix} \right)^{1/2} = \det\left( (0, I_{\beta m - 1}) \trapo{\mathbf{v}} \trapo{\mathbf{r}} \mathbf{r} \mathbf{v} \begin{pmatrix} 0 \\ I_{\beta m - 1} \end{pmatrix} \right)^{1/2},
\end{equation*}
where we have used that \(a b\) corresponds to \(\mathbf{a} \mathbf{b}\), that \(\adj{a}\) corresponds to \(\trapo{\mathbf{a}}\), and that \(v \in \Orth_{m; \beta}\) corresponds to \(\mathbf{v} \in \Orth_{\beta m; 1}\). Summing up, the functional determinant of \(f\) at \(x\) is given by
\begin{equation*}
D(x) := \lvert \det(r) \rvert^{\beta (n - 1)} G(x);
\end{equation*}
note that it depends on \(x\) only through \(v\), not \(u\) nor \(s(x)\). This allows us to compute the volume of \(r \Sph{\Schatten{\infty}{}, \beta}{m, n}\) as follows, where we use Theorem~\ref{sa:tildes_volumen} to transform from Hausdorff- to cone measure,
\begin{align*}
\Hausd^{\beta m n - 1}(r \Sph{\Schatten{\infty}{}, \beta}{m, n}) &= \int_{\Sph{\Schatten{\infty}{}, \beta}{m, n}} D(x) \, \diff\!\Hausd^{\beta m n - 1}(x)\\
&= \Hausd^{\beta m n - 1}(\Sph{\Schatten{\infty}{}, \beta}{m, n}) \int_{\Sph{\Schatten{\infty}{}, \beta}{m, n}} D(x) \, \diff\KegM{\Schatten{\infty}{}, \beta}{m \times n}(x)\\
&= \frac{\Hausd^{\beta m n - 1}(\Sph{\Schatten{\infty}{}, \beta}{m, n})}{\KugVol{\Schatten{\infty}{}, \beta}{m \times n}} \int_{\Kug{\Schatten{\infty}{}, \beta}{m \times n}} D\Bigl( \frac{x}{\lVert x \rVert_{\Schatten{\infty}{}}} \Bigr) \, \diff x\\
&= \frac{\Hausd^{\beta m n - 1}(\Sph{\Schatten{\infty}{}, \beta}{m, n})}{\KugVol{\Schatten{\infty}{}, \beta}{m \times n}} \, C m! \int_{W^m} \int_{\Orth_{n; \beta}} \int_{\Orth_{m; \beta}} \Ind_{\Kug{\Schatten{\infty}{}, \beta}{m \times n}}(v \diag(\sigma) \adj{u})\\
&\qquad \cdot D\Bigl( \frac{v \diag(\sigma) \adj{u}}{\lVert v \diag(\sigma) \adj{u} \rVert_{\Schatten{\infty}{}}} \Bigr) h(\sigma) \, \diff v \, \diff u \, \diff\sigma\\
&= \Hausd^{\beta m n - 1}(\Sph{\Schatten{\infty}{}, \beta}{m, n}) \, \frac{2^{\beta m (m - 1)/4}}{\OrthVol_{m; \beta}} \, \lvert \det(r) \rvert^{\beta (n - 1)} \int_{\Orth_{m; \beta}} G(v \diag(e_1)) \, \diff v;
\end{align*}
we note a few points: in the fourth line we have abbreviated \(C := \frac{2^{\beta m/2 + \beta m (m - 1)/4 + \beta n (n - 1)/4}}{2^{\beta m n/2} m! \OrthVol_{1; \beta}^m \OrthVol_{n - m; \beta}}\) and \(h(\sigma) := \prod_{j = 1}^m \sigma_j^{\beta (n - m + 1) - 1} \prod_{1 \leq j < i \leq m} \lvert \sigma_i^2 - \sigma_j^2 \rvert^\beta\), and as we need the singular values to be in decreasing order for \(G\) we have integrated over \(W^m\), cf.\ Remark~\ref{bem:weylkammer}; for the last line we have invested the structure of \(D\) and the knowledge that \(G\) only depends on \(v\), therefore we have inserted some constant values for \(\sigma\) and \(u\), namely \(e_1\) and \(I_n\), resp. In order to identify the last integral we rely on the following transformation formula, which can be derived from the coarea formula in considering the map \(\Orth_{m; \beta} \to \Sph{2, \beta}{m}\) given by \(v \mapsto v e_1\):
\begin{equation}\label{eq:orth_und_sfaere}
\int_{\Orth_{m; \beta}} f(v) \, \diff v = \int_{\Sph{2, \beta}{m}} \int_{\Orth_{m - 1; \beta}} f\left( \hat{\theta} \begin{pmatrix} 1 & 0 \\ 0 & w \end{pmatrix} \right) \, \diff w \, \diff\!\Hausd^{\beta m - 1}(\theta),
\end{equation}
where \(\hat{\theta} \in \Orth_{m; \beta}\) is such that \(\hat{\theta} e_1 = \theta\). In our situation we need to consider, in the real case,
\begin{equation*}
\hat{\theta} \begin{pmatrix} 1 & 0 \\ 0 & w \end{pmatrix} \begin{pmatrix} 0 \\ I_{m - 1} \end{pmatrix} = \hat{\theta} \begin{pmatrix} 0 \\ I_{m - 1} \end{pmatrix} w,
\end{equation*}
and similarly in the complex case,
\begin{equation*}
\hat{\boldsymbol{\theta}} \begin{pmatrix} 1 & 0 & 0 & 0 \\ 0 & \Re(w) & 0 & -\Im(w) \\ 0 & 0 & 1 & 0 \\ 0 & \Im(w) & 0 & \Re(w) \end{pmatrix} \begin{pmatrix} 0 \\ I_{2 m - 1} \end{pmatrix} = \hat{\boldsymbol{\theta}} \begin{pmatrix} 0 \\ I_{2 m - 1} \end{pmatrix} \begin{pmatrix} \Re(w) & 0 & -\Im(w) \\ 0 & 1 & 0 \\ \Im(w) & 0 & \Re(w) \end{pmatrix}\!;
\end{equation*}
now we recall that a Gramian is orthogonally invariant, hence
\begin{align*}
\int_{\Orth_{m; \beta}} G(v \diag(e_1)) \, \diff v &= \int_{\Sph{2, \beta}{m}} \int_{\Orth_{m - 1; \beta}} \det\left( (0, I_{\beta m - 1}) \trapo{\hat{\boldsymbol{\theta}}} \trapo{\mathbf{r}} \mathbf{r} \hat{\boldsymbol{\theta}} \begin{pmatrix} 0 \\ I_{\beta m - 1} \end{pmatrix} \right)^{1/2} \, \diff w \, \diff\!\Hausd^{\beta m - 1}(\theta)\\
&= \frac{\OrthVol_{m - 1; \beta}}{2^{\beta (m - 1) (m - 2)/4}} \int_{\Sph{2, \beta}{m}} \det\left( (0, I_{\beta m - 1}) \trapo{\hat{\boldsymbol{\theta}}} \trapo{\mathbf{r}} \mathbf{r} \hat{\boldsymbol{\theta}} \begin{pmatrix} 0 \\ I_{\beta m - 1} \end{pmatrix} \right)^{1/2} \diff\!\Hausd^{\beta m - 1}(\theta).
\end{align*}
The conclusion of this part is reached through the identities
\begin{equation*}
\int_{\Sph{2, \beta}{m}} \det\left (0, I_{\beta m - 1}) \trapo{\hat{\boldsymbol{\theta}}} \trapo{\mathbf{r}} \mathbf{r} \hat{\boldsymbol{\theta}} \begin{pmatrix} 0 \\ I_{\beta m - 1} \end{pmatrix} \right)^{1/2} \diff\!\Hausd^{\beta m - 1}(\theta) = \Hausd^{\beta m - 1}(r \Sph{2, \beta}{m})
\end{equation*}
and
\begin{equation*}
\OrthVol_{m; \beta} = 2^{\beta (m - 1)/2} \, \OrthVol_{m - 1; \beta} \Hausd^{\beta m - 1}(\Sph{2, \beta}{m}),
\end{equation*}
the first of which can be proved by considering the map \(\Sph{2, \beta}{m} \to \KZ^m\) given by \(\theta \mapsto r \theta\), leading to calculations which essentially are the same as those above for \(\Sph{\Schatten{\infty}{}, \beta}{m, n}\) without writing the \(u_j\) and only considering \(j = 1\); the latter can be proved either from the volume formula~\eqref{eq:vol_stiefel} or from the transformation formula~\eqref{eq:orth_und_sfaere}.

\item
Fix a singular value decomposition \(r = v \diag(s(r)) \adj{u}\). We already know \(S_\beta^{m, n} = \{\adj{w} \Colon w \in \Orth_{n, m; \beta}\}\) and hence \(r S_\beta^{m, n} = \{\adj{(w \adj{r})} \Colon w \in \Orth_{n, m; \beta}\}\), and for any \(w \in \Orth_{n, m; \beta}\) the tangent space is spanned by the matrices \(E_{i, j}\)\---additionally \(F_{i, j}\) in the complex case\---given by
\begin{alignat*}{2}
F_{j, j} &:= \ie w u e_j \trapo{e_j} \adj{u} && \text{ for } j \in [1, m],\\
E_{i, j} &:= \frac{w}{\sqrt{2}} \, u (e_i \trapo{e_j} - e_j \trapo{e_i}) \adj{u} && \text{ for } 1 \leq j < i \leq m,\\
F_{i, j} &:= \frac{\ie w}{\sqrt{2}} \, u (e_i \trapo{e_j} + e_j \trapo{e_i}) \adj{u} && \text{ for } 1 \leq j < i \leq m,\\
E_{i, j} &:= \hat{w} e_i \trapo{e_j} \adj{u} && \text{ for } i \in [m + 1, n], j \in [1, m],\\
F_{i, j} &:= \ie \hat{w} e_i \trapo{e_j} \adj{u} && \text{ for } i \in [m + 1, n], j \in [1, m],
\end{alignat*}
where \(\hat{w} \in \Orth_{n, \beta}\) is an extension of \(w\) with suitable \(n - m\) columns. Their respective images are, in the same order,
\begin{align*}
f(F_{j, j}) &= s_j(r) \ie w u e_j \trapo{e_j} \adj{v},\\
f(E_{i, j}) &= \frac{w}{\sqrt{2}} \, u (s_j(r) e_i \trapo{e_j} - s_i(r) e_j \trapo{e_i}) \adj{v},\\
f(F_{i, j}) &= \frac{\ie w}{\sqrt{2}} \, u (s_j(r) e_i \trapo{e_j} + s_i(r) e_j \trapo{e_i}) \adj{v},\\
f(E_{i, j}) &= s_j(r) \hat{w} e_i \trapo{e_j} \adj{v},\\
f(F_{i, j}) &= s_j(r) \ie \hat{w} e_i \trapo{e_j} \adj{v}.
\end{align*}
Clearly the elements of each line are orthogonal to the other lines, and the elements of the first, fourth and fifth lines also are orthogonal among themselves and their norms are the indicated singular values. Orthogonality is even true for the second and third lines, although less obvious: let \(1 \leq j < i \leq m\) and \(1 \leq j' < i' \leq m\), then
\begin{equation*}
\langle f(E_{i, j}), f(E_{i', j'}) \rangle = \frac{s_i(r) s_{i'}(r) + s_j(r) s_{j'}(r)}{2} \delta_{i, i'} \delta_{j, j'} - \frac{s_i(r) s_{j'}(r) + s_{i'}(r) s_j(r)}{2} \delta_{i, j'} \delta_{i', j};
\end{equation*}
but for our choice of indices always \(\delta_{i, j'} \delta_{i', j} = 0\), and so orthgonality follows. The analogous calculation goes through for \(f(F_{i, j})\). Therefore the functional determinant equals
\begin{multline*}
\biggl( \prod_{j = 1}^m s_j(r) \biggr)^{\beta - 1} \biggl( \prod_{1 \leq j < i \leq m} \frac{(s_j(r)^2 + s_i(r)^2)^{1/2}}{\sqrt{2}} \biggr)^\beta \biggl( \prod_{1 \leq j \leq m < i \leq n} s_j(r) \biggr)^\beta\\
= 2^{-\beta m (m - 1)/4} \lvert \det(r) \rvert^{\beta (n - m + 1) - 1} \prod_{1 \leq j < i \leq m} (s_j(r)^2 + s_i(r))^{\beta/2}.
\end{multline*}
Since this is indpendent of \(w\), this implies directly
\begin{equation*}
\Hausd^d(\Orth_{n, m; \beta} \adj{r}) = 2^{-\beta m (m - 1)/4} \lvert \det(r) \rvert^{\beta (n - m + 1) - 1} \prod_{1 \leq j < i \leq m} (s_j(r)^2 + s_i(r))^{\beta/2} \cdot \OrthVol_{n, m; \beta},
\end{equation*}
and because adjoining is an isometry the result for \(S_\beta^{m, n}\) follows. \qedhere
\end{asparaenum}
\end{proof}

\begin{Bem}
The factor \(\Hausd^{\beta m - 1}(r \Sph{2, \beta}{m})\) in Proposition~\ref{sa:skalierte_volumen}, part~2, is the surface area of an ellipsoid, as such it cannot be expressed in terms of elementary functions of \(s(r)\) (the semiaxes of the ellipsoid).
\end{Bem}

\section{Probabilistic results and weak limit theorems for \(p = \infty\)}
\label{sec:wahrscheinlichkeit}

For all subsequent probabilistic results we suppose a sufficiently rich probability space \((\Omega, \mathcal{A}, \Wsk)\) on which all random variables that appear in the present paper are defined. The expected value of a random variable \(X\) w.r.t.\ \(\Wsk\) is denoted \(\Erw[X]\) and is to be understood component\-/wise for vector\-/valued \(X\). \(X \sim \mu\) means the law, or distribution, of \(X\) is \(\mu\), and \(\GlVert\) denotes equality in law. For \(A \subset \RZ^d\) with \(\vol_d(A) \in \RZ_{> 0}\) the uniform distribution on \(A\) (w.r.t.\ Lebesgue measure) is \(\Gleichv(A)\), and \(\Gleichv(\Orth_{n, m; \beta})\) means the normalized invariant measure on \(\Orth_{n, m; \beta}\) (see also Section~\ref{sec:auxiliar}). Convergence in law, in probability, and almost surely are denoted by \(\KiVert{}\), \(\KiWsk{}\), and \(\Kfs{}\) respectively.

The following can be considered a variant of the Schechtman\--Zinn representation of the uniform distribution on the \(\ellpe{p}{}\)\-/balls. It uses the \emph{matrix\-/variate type\=/I beta\-/distribution,} written \(\BetaMatr{m, \beta}(a, b)\) and defined as follows (see~\cite[Definition~5.2.1]{GN2000}): suppose \(a, b > \frac{\beta (m - 1)}{2}\), then \(\BetaMatr{m, \beta}(a, b)\) is the probability distribution on \(\Sym_{m, \beta}\) with density
\begin{equation*}
x \mapsto \frac{\Gamma_{m, \beta}(a + b)}{\Gamma_{m, \beta}(a) \Gamma_{m, \beta}(b)} \det(x)^{a - 1 - \beta (m - 1)/2} \det(I_m - x)^{b - 1 - \beta (m - 1)/2} \Ind_{(0, I_m)}(x),
\end{equation*}
where of course \((0, I_m) := \{x \in \Sym_{m, \beta} \Colon 0 < x < I_m\}\).

\begin{Sa}\label{sa:gleichv_b_probdarst1}
A random variable \(X \in \KZ^{m \times n}\) is distributed uniformly on \(\Kug{\Schatten{\infty}{}, \beta}{m \times n}\) iff there exist independent random variables \(R \in \KZ^{m \times m}\) and \(U \in \KZ^{n \times m}\) such that \(R \sim \BetaMatr{m, \beta}(\frac{\beta n}{2}, \frac{\beta (m - 1)}{2} + 1)\), \(U \sim \Gleichv(\Orth_{n, m; \beta})\), and
\begin{equation*}
X \GlVert R^{1/2} \, \adj{U}.
\end{equation*}
\end{Sa}

\begin{Bem}\label{bem:prob_darst}
Note that from \(X \GlVert R^{1/2} \, \adj{U}\) follow \(R \GlVert X \adj{X}\) and \(U \GlVert \adj{X} (X \adj{X})^{-1/2}\).
\end{Bem}

\begin{proof}
The statement is a simple consequence of Proposition~\ref{sa:integraltrafo}, part~1. We have \(X \sim \Gleichv(\Kug{\Schatten{\infty}{}, \beta}{m \times n})\) iff, for any \(f \colon \KZ^{m \times n} \to \RZ\) nonnegative and measurable,
\begin{align*}
\Erw[f(X)] &= \frac{1}{\KugVol{\Schatten{\infty}{}, \beta}{m \times n}} \int_{\KZ^{m \times n}} \Ind_{\Kug{\Schatten{\infty}{}, \beta}{m \times n}}(x) f(x) \, \diff x\\
&= \frac{2^{-m}}{\KugVol{\Schatten{\infty}{}, \beta}{m \times n}} \int_{\Pos_{m, \beta}} \int_{\Orth_{n, m; \beta}} \Ind_{\Kug{\Schatten{\infty}{}, \beta}{m \times n}}(r^{1/2} \, \adj{u}) f(r^{1/2} \, \adj{u}) \, \diff u \det(r)^{\beta (n - m + 1)/2 - 1} \, \diff r\\
&= \frac{2^{-m}}{\KugVol{\Schatten{\infty}{}, \beta}{m \times n}} \int_{\Pos_{m, \beta}} \int_{\Orth_{n, m; \beta}} f(r^{1/2} \, \adj{u}) \, \diff u \Ind_{(0, I_m]}(r) \det(r)^{\beta (n - m + 1)/2 - 1} \, \diff r\\
&= \Erw[f(R^{1/2} \, \adj{U})],
\end{align*}
where \(R\) and \(U\) have the claimed joint distribution. (This is obvious for \(U\); concerning \(R\) one may compare our result with the known density.) In the last integral, \(r^{1/2} \, \adj{u} \in \Kug{\Schatten{\infty}{}, \beta}{m \times n}\) iff \(r \leq I_m\), and note that \((0, I_m] \setminus (0, I_m)\) is negligible.
\end{proof}

Another representation, related to the first one, can be obtained in the special case \(n = d m\) for some \(d \in \NZ\). Here we may identify \(\KZ^{m \times d m} \cong (\KZ^{m \times m})^d\), constisting of \(d\)\=/dimensional row vectors whose entries are from \(\KZ^{m \times m}\). Then, for \(x = (x_1, \dotsc, x_d), y = (y_1, \dotsc, y_d) \in \KZ^{m \times d m}\), \(x \adj{y} = \sum_{i = 1}^d x_i \adj{y_i}\), and hence also \(\langle x, y \rangle = \sum_{i = 1}^d \langle x_i, y_i \rangle\). It uses the \emph{matrix\-/variate type\=/I Dirichlet\-/distribution} \(\DirichletMatr{m, \beta}(a_1, \dotsc, a_d; a_{d + 1})\), with \(a_1, \dotsc, a_{d + 1} > \frac{\beta (m - 1)}{2}\), defined to be the probability distribution on \(\Sym_{m, \beta}^d\) with density (see~\cite[Definition~6.2.1]{GN2000})
\begin{equation*}
\begin{split}
(x_1, \dotsc, x_d) &\mapsto \frac{\Gamma_{m, \beta}\bigl( \sum_{i = 1}^{d + 1} a_i \bigr)}{\prod_{i = 1}^{d + 1} \Gamma_{m, \beta}(a_i)} \prod_{i = 1}^d \det(x_i)^{a_i - 1 - \beta (m - 1)/2}\\
&\qquad \cdot \det\biggl(I_m - \sum_{i = 1}^d x_i \biggr)^{a_{d + 1} - 1 - \beta (m - 1)/2} \Ind_{\Delta_{m, \beta}^d}(x_1, \dotsc, x_d),
\end{split}
\end{equation*}
where \(\Delta_{m, \beta}^d := \bigl\{ (x_1, \dotsc, x_d) \in \Sym_{m, \beta}^d \Colon x_1, \dotsc, x_d > 0 \wedge \sum_{i = 1}^d x_i < I_m \bigr\}\).

\begin{Bem}
With this componentwise interpretation the spectral norm unit ball can be written \(\Kug{\Schatten{\infty}{}, \beta}{m \times d m} = \bigl\{ (x_i)_{i \leq d} \in (\KZ^{m \times m})^d \Colon \sum_{i = 1}^d x_i \adj{x_i} \leq I_m \bigr\}\), which looks like a natural generalization of the usual \(\ellpe{2}{}\)\=/ball \(\Kug{2, \beta}{d} = \bigl\{ (x_i)_{i \leq d} \in \KZ^d \Colon \sum_{i = 1}^d \lvert x_i \rvert^2 \leq 1 \bigr\}\), replacing the scalar components with matrix\-/valued ones. In truth this was the starting point of the present investigations.
\end{Bem}

\begin{Sa}\label{sa:gleichv_b_probdarst2}
A random vector \(X \in \KZ^{m \times d m}\) is distributed uniformly on \(\Kug{\Schatten{\infty}{}, \beta}{m \times d m}\) iff there exist independent random variables \(R = (R_1, \dotsc, R_d) \in \KZ^{m \times d m}\) and \(U_1, \dotsc, U_d \in \KZ^{m \times m}\) such that \(R \sim \DirichletMatr{m, \beta}(\frac{\beta m}{2}^{[d]}; \frac{\beta (m - 1)}{2} + 1)\) (i.e., the parameter \(\frac{m}{2}\) occurs \(d\) times), \(U_i \sim \Gleichv(\Orth_{m; \beta})\) for each \(i \in \{1, \dotsc, d\}\), and
\begin{equation*}
X \GlVert (R_1^{1/2} \, U_1, \dotsc, R_d^{1/2} \, U_d).
\end{equation*}
\end{Sa}

\begin{proof}
Again we apply Proposition~\ref{sa:integraltrafo}, part~1, this time separately to each of the \(d\) components of \(\KZ^{m \times d m} = \KZ^{m \times m} \times \dotsm \times \KZ^{m \times m}\). We have \(X \sim \Gleichv(\Kug{\Schatten{\infty}{}, \beta}{m \times d m})\) iff, for any \(f \colon \KZ^{m \times d m} \to \RZ\) measurable and nonnegative,
\begin{align*}
\Erw[f(X)] &= \frac{1}{\KugVol{\Schatten{\infty}{}, \beta}{m \times d m}} \int_{\RZ^{m \times d m}} \Ind_{\Kug{\Schatten{\infty}{}, \beta}{m \times d m}}(x) f(x) \, \diff x\\
&= \frac{(2^{-m})^d}{\KugVol{\Schatten{\infty}{}, \beta}{m \times d m}} \int_{\Pos_{m, \beta}} \int_{\Orth_{m; \beta}} \dotsi \int_{\Pos_{m, \beta}} \int_{\Orth_{m; \beta}} \Ind_{\Kug{\Schatten{\infty}{}, \beta}{m \times d m}}(r_1^{1/2} \, \adj{u_1}, \dotsc, r_d^{1/2} \, \adj{u_d})\\
&\mspace{100mu} \cdot f(r_1^{1/2} \, \adj{u_1}, \dotsc, r_d^{1/2} \, \adj{u_d}) \prod_{i = 1}^d \det(r_i)^{\beta/2 - 1} \prod_{i = 1}^d (\diff u_i \, \diff r_i)\\
&= \frac{2^{-d m}}{\KugVol{\Schatten{\infty}{}, \beta}{m \times d m}} \int_{\Delta_{m, \beta}^d} \int_{\Orth_{m; \beta}^d} f(r_1^{1/2} \, u_1, \dotsc, r_d^{1/2} \, u_d) \prod_{i = 1}^d \diff u_i \prod_{i = 1}^d \det(r_i)^{\beta/2 - 1} \, \diff r\\
&= \Erw[f(R_1^{1/2} \, U_1, \dotsc, R_d^{1/2} \, U_d)],
\end{align*}
where the desired conclusion is reached by noticing a few points: firstly, \((r_1^{1/2} \, u_1, \dotsc, r_d^{1/2} \, u_d) \in \Kug{\Schatten{\infty}{}, \beta }{m \times d m}\) iff \(\sum_{i = 1}^d r_i \leq I_m\); secondly, we have transformed \(u_i \mapsto \adj{u_i}\), which has no import since we are integrating with respect to the (scaled) Haar measure; thirdly and lastly, the distribution of the \(U_i\) is obvious again, and that of \((R_1, \dotsc, R_d)\) can be read off by comparing with the known density.
\end{proof}

Next we present a Poincar\'e\--Maxwell\--Borel\-/type result. It is useful to keep the following relationships in mind. The standard Gaussian distribution on \(\RZ\) is denoted by \(\Normal{1}\), defined by its Lebesgue\-/density \(x \mapsto (2 \pi)^{-1/2} \ez^{-x^2/2}\). The standard Gaussian distribution on \(\CZ\) is denoted by \(\Normal{2}\) and defined to be the law of \(2^{-1/2} (x + \ie y)\), where \((x, y) \sim \Normal{1} \otimes \Normal{1}\). Note \(\Erw[\lvert X \rvert^2] = 1\) for \(X \sim \Normal{\beta}\). The content of the following proposition is standard from random matrix theory, especially as used in statistics; textbook references are \cite{GN2000, MathProvHau2022}.

\begin{Sa}\label{sa:gausz_gamma_beta_dirichlet}
\begin{compactenum}
\item
Let \(X \sim \Normal{\beta}^{\otimes (m \times n)}\), then \(X \adj{X} \sim \Gamma_{m, \beta}(\frac{\beta n}{2}, \frac{2}{\beta} I_m)\), where the matrix\-/variate gamma\-/distribution \(\Gamma_{m, \beta}(\alpha, b)\) with \(\alpha > \frac{\beta (m - 1)}{2}\) and \(b \in \Pos_{m, \beta}\) is defined on \(\Pos_{m, \beta}\) by its density
\begin{equation*}
x \mapsto \frac{\det(x)^{\alpha - 1 - \beta (m - 1)/2}}{\det(b)^\alpha \Gamma_{m, \beta}(\alpha)} \, \ez^{-\Spur(\inv{b} x)}.
\end{equation*}
\item
\(\BetaMatr{m, \beta}(\alpha_1, \alpha_2) = \DirichletMatr{m, \beta}(\alpha_1; \alpha_2)\).
\item
Let \(d \in \NZ\), \(\alpha_1, \dotsc, \alpha_{d + 1} > \frac{\beta (m - 1)}{2}\), \(b \in \Pos_{m, \beta}\), and \((X_i)_{i \leq d + 1} \sim \bigotimes_{i = 1}^{d + 1} \Gamma_{m, \beta}(\alpha_i, b)\). Define \(Y_i := \bigl( \sum_{j = 1}^{d + 1} X_j \bigr)^{-1/2} X_i \bigl( \sum_{j = 1}^{d + 1} X_j \bigr)^{-1/2}\) for \(i \in [1, d]\) and \(Y_{d + 1} := \sum_{i = 1}^{d + 1} X_i\), then \((Y_i)_{i \leq d}\) and \(Y_{d + 1}\) are independent, and \((Y_i)_{i \leq d} \sim \DirichletMatr{m, \beta}(\alpha_1, \dotsc, \alpha_d; \alpha_{d + 1})\) and \(Y_{d + 1} \sim \Gamma_{m, \beta}\bigl( \sum_{i = 1}^{d + 1} \alpha_i, b \bigr)\).
\item
Consequently, given \(n_1, \dotsc, n_{d + 1} \in \NZ_{\geq m}\), then \((X_i)_{i \leq d} \sim \DirichletMatr{m, \beta}(\frac{\beta n_1}{2}, \dotsc, \frac{\beta n_d}{2}; \frac{\beta n_{d + 1}}{2})\) iff there exists \((G_i)_{i \leq d + 1} \sim \bigotimes_{i = 1}^{d + 1} \Normal{\beta}^{\otimes (m \times n_i)}\) such that \(X_i \GlVert \bigl( \sum_{j = 1}^{d + 1} G_j \adj{G_j} \bigr)^{-1/2} G_i \adj{G_i} \bigl( \sum_{j = 1}^{d + 1} G_j \adj{G_j} \bigr)^{-1/2}\) for each \(i \in [1, d]\).
\end{compactenum}
\end{Sa}

\begin{Thm}\label{sa:pmb}
Consider \(m\) fixed, let \(k \in \NZ\), and let \((X_n)_{n \geq m}\) be a sequence of random variables such that either
\begin{compactitem}
\item \(X_n \sim \Gleichv(\Kug{\Schatten{\infty}{}, \beta}{m \times n})\) for each \(n \geq m\), or
\item \(X_n \sim \Gleichv(\Sph{\Schatten{\infty}{}, \beta}{m, n})\) for each \(n \geq m\), or
\item \(X_n \sim \Gleichv(S_\beta^{m, n})\) for each \(n \geq m\).
\end{compactitem}
Denote the projection of \(X_n\) onto the first \(k\) columns by \(X_n^{(k)}\). Then in any case
\begin{equation*}
(n^{1/2} X_n^{(k)})_{n \geq \max\{m, k\}} \KiVert{} \Normal{\beta}^{\otimes (m \times k)}.
\end{equation*}
\end{Thm}

\begin{proof}
We begin with the case \(X_n \sim \Gleichv(S_\beta^{m, n})\), then \(\adj{X_n} \sim \Gleichv(\Orth_{n, m; \beta})\). Let \((g_n)_{n \geq 1}\) be an i.i.d.\-/sequence of \(\Normal{\beta}^{\otimes m}\)\-/distributed vectors and set \(G_n := (g_1, \dotsc, g_n)\) for \(n \geq 1\), then \(G_n \sim \Normal{\beta}^{\otimes (m \times n)}\), and by Proposition~\ref{sa:simul_gleich_stiefel} we have \(\adj{X_n} \GlVert \adj{G_n} (G_n \adj{G_n})^{-1/2}\). It follows that
\begin{equation*}
X_n^{(k)} \GlVert (G_n \adj{G_n})^{-1/2} G_k.
\end{equation*}
By the strong law of large numbers,
\begin{equation}\label{eq:pmb_stggz}
\frac{1}{n} \, G_n \adj{G_n} = \frac{1}{n} \sum_{i = 1}^n g_i \adj{g_i} \Kfs{n \to \infty} \Erw[g_1 \adj{g_1}] = I_m,
\end{equation}
and therefore,
\begin{equation*}
n^{1/2} X_n^{(k)} \GlVert \Bigl( \frac{1}{n} \, G_n \adj{G_n} \Bigr)^{-1/2} G_k \Kfs{n \to \infty} G_k,
\end{equation*}
which proves the claim.

We proceed with the case \(X_n \sim \Gleichv(\Kug{\Schatten{\infty}{}, \beta}{m \times n})\), then by Proposition~\ref{sa:gleichv_b_probdarst1}
\begin{equation*}
X_n \GlVert R_n^{1/2} \, \adj{U_n}
\end{equation*}
with independent random variables \(R_n \sim \BetaMatr{m, \beta}(\frac{\beta n}{2}, \frac{\beta (m - 1)}{2} + 1)\) and \(U_n \sim \Gleichv(\Orth_{n, m; \beta})\). The first case implies
\begin{equation*}
(n^{1/2} {\adj{U_n}}^{(k)})_{n \geq \max\{m, k\}} \KiVert{} G_k.
\end{equation*}
Using the probabilistic representation of \(\BetaMatr{m, \beta}(\frac{\beta n}{2}, \frac{\beta (m - 1)}{2} + 1)\) from Proposition~\ref{sa:gausz_gamma_beta_dirichlet} we can write
\begin{equation*}
R_n \GlVert (G_n \adj{G_n} + H)^{-1/2} G_n \adj{G_n} (G_n \adj{G_n} + H)^{-1/2},
\end{equation*}
where \(H \sim \Gamma_{m, \beta}(\frac{\beta (m - 1)}{2} + 1, \frac{2}{\beta} I_m)\) is independent from \((g_j)_{j \in \NZ}\). Since \((n^{-1} H)_{n \geq 1} \Kfs{} 0\), together with Equation~\eqref{eq:pmb_stggz} we infer
\begin{equation*}
(R_n)_{n \geq m} \KiWsk{} I_m,
\end{equation*}
and therewith we conclude
\begin{equation*}
n^{1/2} X_n^{(k)} \GlVert R_n^{1/2} \, n^{1/2} {\adj{U_n}}^{(k)} \KiVert{n \to \infty} G_k,
\end{equation*}
as desired.

Lastly we consider the case \(X_n \sim \Gleichv(\Sph{\Schatten{\infty}{}, \beta}{m, n}) = \KegM{\Schatten{\infty}{}, \beta}{m \times n}\). Let \(Y_n \sim \Gleichv(\Kug{\Schatten{\infty}{}, \beta}{m \times n})\), then by Proposition~\ref{sa:k_symm_masz_unabh},
\begin{equation*}
X_n \GlVert \frac{Y_n}{\lVert Y_n \rVert_{\Schatten{\infty}{}}}
\end{equation*}
and by the same proposition and by Remark~\ref{bem:k_symm}, part~1, the distribution of \(\lVert Y_n \rVert_{\Schatten{\infty}{}}\) has density \(r \mapsto \beta m n r^{\beta m n - 1} \Ind_{[0, 1]}(r)\). The latter fact implies \(\lVert Y_n \rVert_{\Schatten{\infty}{}}^{\beta m n} \sim \Gleichv([0, 1])\) and therefore
\begin{equation*}
(\lVert Y_n \rVert_{\Schatten{\infty}{}})_{n \geq m} \Kfs{} 1.
\end{equation*}
(First use \(\lVert Y_n \rVert_{\Schatten{\infty}{}} \GlVert U^{1/(\beta m n)}\) with \(U \sim \Gleichv([0, 1])\) to get convergence in distribution; because the limit is constant, promote to convergence in probability; use Borel\--Cantelli to obtain almost sure convergence.) This finally leads to
\begin{equation*}
n^{1/2} X_n^{(k)} \GlVert \frac{n^{1/2} Y_n^{(k)}}{\lVert Y_n \rVert_{\Schatten{\infty}{}}} \KiVert{n \to \infty} G_k,
\end{equation*}
and the proof is complete.
\end{proof}

\begin{Bem}
Essentially the same result, but using a very different method, for the case of \(\Gleichv(S_\beta^{m, n})\) has been obtained recently by Petrov and Vershik~\cite[Theorem~3.1]{PetrovVershik2021UE}.
\end{Bem}

The next proposition has a nice geometric interpretation: pick two points independently from either \(\Kug{\Schatten{\infty}{}, \beta}{m \times n}\), \(\Sph{\Schatten{\infty}{}, \beta}{m, n}\), or \(S_\beta^{m, n}\), then they will be nearly orthogonal, when the dimension is high. This will be formulated as a distributional limit theorem.

\begin{Thm}\label{sa:zgs_inneresprodukt}
Let \((X_n)_{n \geq m}\) and \((Y_n)_{n \geq m}\) be sequences of random variables such that individually their distributions satisfy the same conditions as in Theorem~\ref{sa:pmb}, and additionally let \(\{X_n, Y_n\}\) be independent for each \(n \geq m\). Then
\begin{equation*}
(n^{1/2} X_n \adj{Y_n})_{n \geq m} \KiVert{} \Normal{\beta}^{\otimes (m \times m)},
\end{equation*}
and
\begin{equation*}
((\beta \inv{m} \, n)^{1/2} \langle X_n, Y_n \rangle)_{n \geq m} \KiVert{} \Normal{1}.
\end{equation*}
\end{Thm}

\begin{proof}
We use the same notation as in the proof of Theorem~\ref{sa:pmb}. Note that the second convergence is an almost immediate consequence of the first one: apply the real part of the trace to get
\begin{equation*}
\frac{(\beta n)^{1/2}}{m^{1/2}} \langle X_n, Y_n \rangle = \frac{\beta^{1/2}}{m^{1/2}} \Re \Spur(n^{1/2} X_n \adj{Y_n}) \KiVert{n \to \infty} \frac{\beta^{1/2}}{m^{1/2}}\Re \Spur(G_m) = \frac{\beta^{1/2}}{m^{1/2}} \sum_{i = 1}^m \Re(g_{i, i}) \sim \Normal{1},
\end{equation*}
since the \(\Re(g_{i, i})\) are i.i.d.\ Gaussian with zero mean and variance \(\frac{1}{\beta}\).

Concerning the first statement, it suffices to consider the case \(X_n, Y_n \sim \Gleichv(S_\beta^{m, n})\); then use the same arguments as in the proof of Theorem~\ref{sa:pmb} to get rid of any `radial' components of the random variables. Write
\begin{equation*}
X_n \GlVert (G_n \adj{G_n})^{-1/2} G_n \quad \text{and} \quad Y_n \GlVert (G_n' \adj{(G_n')})^{-1/2} G_n',
\end{equation*}
where \(G_n' = (g_1', \dotsc, g_n')\) is an independent copy of \(G_n\), to obtain
\begin{align*}
n^{1/2} X_n \adj{Y_n} &\GlVert n^{1/2} (G_n \adj{G_n})^{-1/2} G_n \adj{(G_n')} (G_n' \adj{(G_n')})^{-1/2}\\
&= \Bigl( \frac{1}{n} G_n \adj{G_n} \Bigr)^{-1/2} \Bigl( \frac{1}{\sqrt{n}} G_n \adj{(G_n')} \Bigr) \Bigl( \frac{1}{n} G_n' \adj{(G_n')} \Bigr)^{-1/2}.
\end{align*}
Because of Equation~\eqref{eq:pmb_stggz} the square root terms are accounted for and they do not contribute to the limit. The middle term can be written out,
\begin{equation*}
G_n \adj{(G_n')} = \sum_{j = 1}^n g_j \adj{(g_j')},
\end{equation*}
and this is a sum of i.i.d.\ random variables; the summands are centred because of
\begin{equation*}
\Erw[g_1 \adj{g_1'}] = \Erw[g_1] \adj{\Erw[g_j']} = o \trapo{o} = 0,
\end{equation*}
their covariance matrix equals
\begin{align*}
\Cov[\Vek(g_1 \adj{(g_1')})] = \Erw[\Vek(g_1 \adj{(g_1')}) \adj{\Vek(g_1 \adj{(g_1')})}] = \konj{\Cov[g_1']} \otimes \Cov[g_1] = I_m \otimes I_m = I_{m^2},
\end{align*}
and in the complex case their pseudo\-/covariance matrix (also called relation matrix) equals
\begin{equation*}
\Rel[\Vek(g_1 \adj{(g_1')})] = \Erw[\Vek(g_1 \adj{(g_1')}) \trapo{\Vek(g_1 \adj{(g_1')})}] = \konj{\Rel[g_1']} \otimes \Rel[g_1] = 0 \otimes 0 = 0.
\end{equation*}
Therefore by the multivariate central limit theorem
\begin{equation*}
\frac{1}{\sqrt{n}} G_n \adj{(G_n')} \KiVert{n \to \infty} G_m.
\end{equation*}
This finally yields
\begin{equation*}
n^{1/2} X_n \adj{Y_n} \KiVert{n \to \infty} G_m. \qedhere
\end{equation*}
\end{proof}

\section{Theory of \(K\)\-/symmetric measures}
\label{sec:symmetrischemasze}

Referring back to Remark~\ref{bem:prob_darst}, we see that if \(X \sim \Gleichv(\Kug{\Schatten{\infty}{}, \beta}{m \times n})\), then \(\adj{X} (X \adj{X})^{-1/2} \sim \Gleichv(\Orth_{n, m; \beta})\). But it is also well known that if \(X \sim \Normal{1}^{\otimes (m \times n)}\), then \(\trapo{X} (X \trapo{X})^{-1/2} \sim \Gleichv(\Orth_{n, m; 1})\) too (see~\cite[Theorem~2.2.\ and top of p.~270]{Chikuse1990}, e.g.). This can be generalized. But first we need to introduce a notion of symmetry of a measure.

\begin{Def}\label{def:k_symm_masz}
Let \(d \in \NZ\) and let \(K \subset \RZ^d\) be a star body. A Borel\-/measure \(\mu\) on \(\RZ ^d\) is called \emph{\(K\)\=/symmetric} iff there exists a Borel\-/measure \(\rho\) on \(\RZ_{\geq 0}\) such that for any measurable, nonnegative map \(f \colon \RZ^d \to \RZ\),
\begin{equation*}
\int_{\RZ^d} f(x) \, \diff\mu(x) = \int_{\RZ_{\geq 0}} \int_{\Rand K} f(r \theta) \, \diff\KegM{K}{}(\theta) \, \diff\rho(r),
\end{equation*}
where recall that \(\KegM{K}{}\) denotes the normalized cone measure on \(\Rand K\).
\end{Def}

\begin{Bem}\label{bem:k_symm}
\begin{asparaenum}
\item
By the polar integration formula~\eqref{eq:kegmasz_polarint} the uniform distribution on \(K\) is \(K\)\=/symmetric with \(\frac{\diff\rho}{\diff r}(r) = d r^{d - 1} \Ind_{[0, 1]}(r)\); and \(\KegM{K}{}\) itself is \(K\)\=/symmetric with \(\rho = \delta_1\).

\item
Given a \(K\)\=/symmetric measure \(\mu\), the `radial measure' \(\rho\) as introduced above is unique. (This can be proved by inserting \(f = \Ind_A\) with \(A = \{r \theta \Colon r \in A_1, \theta \in \Rand{K}\}\) for some Borel\-/set \(A_1 \subset \RZ_{\geq 0}\).)
\end{asparaenum}
\end{Bem}

That Definition~\ref{def:k_symm_masz} captures the symmetry of \(K\) in a natural way is argued by the following statement which tells us that the density of a \(K\)\=/symmetric measure is constant on the dilates of \(\Rand{K}\).

\begin{Sa}\label{sa:k_symm_absstet}
Let \(d\), \(K\) be as in Definition~\ref{def:k_symm_masz}, and let \(\mu\) be absolutely continuous w.r.t.\ \(\vol_d\). Then \(\mu\) is \(K\)\=/symmetric iff its \(\vol_d\)\=/density is of the form \(f \circ \lvert \cdot \rvert_K\) with some measurable, nonnegative \(f \colon \RZ_{\geq 0} \to \RZ\); and then \(\frac{\diff\rho(r)}{\diff r} = d \vol_d(K) f(r) r^{d - 1}\).
\end{Sa}

\begin{proof}
\(\Rightarrow\): Call \(F := \frac{\diff\mu}{\diff\!\vol_d}\). Let \(h \colon \RZ^d \to \RZ\) be measurable and nonnegative. Then, with a suitable measure \(\rho\),
\begin{align}
\int_{\RZ_{\geq 0}} \int_{\Rand{K}} h(r \theta) \, \diff\KegM{K}{}(\theta) \, \diff\rho(r) &= \int_{\RZ^d} h(x) \, \diff\mu(x) = \int_{\RZ^d} h(x) F(x) \, \diff x \label{eq:k_symm_polar1}\\
&= d \vol_d(K) \int_{\RZ_{\geq 0}} \int_{\Rand{K}} h(r \theta) F(r \theta) \, \diff\KegM{K}{}(\theta) \, r^{d - 1} \, \diff r. \notag
\end{align}
Choose \(h = h_1 \circ \lvert \cdot \rvert_K\) with some measurable, nonnegative \(h_1 \colon \RZ_{\geq 0} \to \RZ\), then \(h(r \theta) = h_1(r)\) for any \(r > 0\) and \(\theta \in \Rand{K}\), and therewith
\begin{equation}\label{eq:k_symm_polar2}
\int_{\RZ_{\geq 0}} h_1(r) \, \diff\rho(r) = d \vol_d(K) \int_{\RZ_{\geq 0}} h_1(r) \int_{\Rand{K}} F(r \theta) \, \diff\KegM{K}{}(\theta) \, r^{d - 1} \, \diff r.
\end{equation}
Define \(f(r) := \int_{\Rand{K}} F(r \theta) \, \diff\KegM{K}{}(\theta)\) for \(r \geq 0\), then Equation~\ref{eq:k_symm_polar2} tells us \(\frac{\diff\rho}{\diff r}(r) = d \vol_d(K) f(r) r^{d - 1}\), and going back to Equation~\ref{eq:k_symm_polar1}, again with \(h\) arbitrary, we have
\begin{align*}
\int_{\RZ^d} h(x) F(x) \, \diff x &= \int_{\RZ_{\geq 0}} \int_{\Rand{K}} h(r \theta) \, \diff\KegM{K}{}(\theta) \, d \vol_d(K) f(r) r^{d - 1} \, \diff r\\
&= d \vol_d(K) \int_{\RZ_{\geq 0}} \int_{\Rand{K}} h(r \theta) f(\lvert r \theta \rvert_K) \, \diff\KegM{K}{}(\theta) \, r^{d - 1} \, \diff r\\
&= \int_{\RZ^d} h(x) f(\lvert x \rvert_K) \, \diff x,
\end{align*}
and from this we obtain \(F = f \circ \lvert \cdot \rvert_K\), as claimed.

\(\Leftarrow\): Let \(h \colon \RZ^d \to \RZ\) be measurable and nonnegative. Then,
\begin{align*}
\int_{\RZ^d} h(x) \, \diff\mu(x) &= \int_{\RZ^d} h(x) f(\lvert x \rvert_K) \, \diff x\\
&= d \vol_d(K) \int_{\RZ_{\geq 0}} \int_{\Rand{K}} h(r \theta) f(\lvert r \theta \rvert_K) r^{d - 1} \, \diff\KegM{K}{}(\theta) \, \diff r\\
&= \int_{\RZ_{\geq 0}} \int_{\Rand{K}} h(r \theta) \, \diff\KegM{K}{}(\theta) \, d \vol_d(K) f(r) r^{d - 1} \, \diff r,
\end{align*}
hence the claim follows with \(\frac{\diff \rho}{\diff r}(r) := d \vol_d(K) f(r) r^{d - 1}\).
\end{proof}

\begin{Sa}\label{sa:k_symm_masz_unabh}
Let \(d\), \(K\), \(\mu\), \(\rho\) be as in Definition~\ref{def:k_symm_masz}, with \(\mu\) a probabiliy measure, and let \(X\) be a \(\RZ^d\)\=/valued random variable. Then \(X \sim \mu\) iff there exist independent random variables \(R\) and \(\Theta\) such that \(R \sim \rho\), \(\Theta \sim \KegM{K}{}\), and \(X \GlVert R \Theta\). In that case \(\lvert X \rvert_K \GlVert R\).

If in addition \(\mu\{o\} = 0\), then \(\frac{X}{\lvert X \rvert_K} \GlVert \Theta\), and \(\lvert X \rvert_K\) and \(\frac{X}{\lvert X \rvert_K}\) are stochastically independent.
\end{Sa}

\begin{proof}
The representation \(X \GlVert R \Theta\) together with the properties of \(R\) and \(\Theta\) are a mere rewriting of Definition~\ref{def:k_symm_masz}. Because of \(\Theta \sim \KegM{K}{}\) we have \(\Theta \in \Rand{K}\) almost surely, hence \(\lvert \Theta \rvert_K = 1\) almost surely, also \(R \geq 0\) almost surely, and thus \(\lvert X \rvert_K \GlVert \lvert R \Theta \rvert_K = R \lvert \Theta \rvert_K = R\).

Note \(\Wsk[X = o] = \mu\{o\}\) and \(X = o \Longleftrightarrow \lvert X \rvert_K = 0\), hence if \(\mu\{o\} = 0\), then \(\frac{X}{\lvert X \rvert_K}\) is almost surely defined. In the latter case, let \(f \colon \RZ_{\geq 0} \to \RZ\) and \(g \colon \Rand{K} \to \RZ\) be measurable and nonnegative, then,
\begin{align*}
\Erw\Bigl[ f(\lvert X \rvert_K) g\Bigl( \frac{X}{\lvert X \rvert_K} \Bigr) \Bigr] &= \int_{\RZ_{\geq 0}} \int_{\Rand{K}} f(\lvert r \theta \rvert_K) g\Bigl( \frac{r \theta}{\lvert r \theta \rvert_K} \Bigr) \, \diff\KegM{K}{}(\theta) \, \diff\rho(r)\\
&= \int_{\RZ_{\geq 0}} \int_{\Rand{K}} f(r) g(\theta) \, \diff\KegM{K}{}(\theta) \, \diff\rho(r)\\
&= \int_{\RZ_{\geq 0}} f(r) \, \diff\rho(r) \int_{\Rand{K}} g(\theta) \, \diff\KegM{K}{}(\theta).
\end{align*}
From that follow the distribution of \(\frac{X}{\lvert X \rvert_K}\) and the independence of \(\lvert X \rvert_K\) and \(\frac{X}{\lvert X \rvert_K}\).
\end{proof}

Thus equipped we turn to probability distributions on \(\KZ^{m \times n}\).

\begin{Sa}\label{sa:simul_gleich_stiefel}
Let \(K \subset \KZ^{m \times n}\) be a star body and let \(X\) be an \(\KZ^{m \times n}\)\=/valued random variable whose distribution is \(K\)\=/symmetric with radial distribution \(\rho\).
\begin{compactenum}
\item \(X \adj{X}\) is almost surely invertible, iff \(X \neq 0\) almost surely, iff \(\rho\{0\} = 0\).
\item If \(\rho\{0\} = 0\) and additionally \(K\) is right\-/unitarily invariant (that is \(K u = K\) for all \(u \in \Orth_{n; \beta}\)), then \(X \adj{X}\) and \(\adj{X} (X \adj{X})^{-1/2}\) are stochastically independent, and \(\adj{X} (X \adj{X})^{-1/2} \sim \Gleichv(\Orth_{n, m; \beta})\).
\item If \(K\) is right\-/unitarily invariant and the distribution of \(X\) has \(\vol_{\beta m n}\)\=/density \(h \circ \lvert \cdot \rvert_K\), then the distribution of \(X \adj{X}\) has density proportional to \(r \mapsto h(\lvert r^{1/2} \rvert_{\lvert K \rvert}) \det(r)^{\beta (n - m + 1)/2 - 1}\), where \(\lvert K \rvert := \{(x \adj{x})^{1/2} \Colon x \in K\} \subset \Sym_{m, \beta}\).
\end{compactenum}
\end{Sa}

\begin{proof}
\begin{asparaenum}
\item
Write \(X \GlVert R \Theta\) with \(R \sim \rho\) and \(\Theta \sim \KegM{K}{}\) independent, according to Proposition~\ref{sa:k_symm_masz_unabh}; then \(X = 0\) iff \(\lvert X \rvert_K \GlVert R = 0\), hence \(\Wsk[X = 0] = \Wsk[R = 0] = \rho\{0\}\).

{\Absatz}Next note that \(X \adj{X}\) is invertible iff \(X\) has full rank, iff \(s_m(X) > 0\); using the representation of \(X\) his yields
\begin{equation*}
\begin{split}
\Wsk[s_m(X) > 0] &= \Wsk[s_m(R \Theta) > 0] = \Wsk[R s_m(\Theta) > 0]\\
&= \Wsk[R > 0 \wedge s_m(\Theta) > 0] = \Wsk[R > 0] \Wsk[s_m(\Theta) > 0].
\end{split}
\end{equation*}
Trivially \(\Wsk[R > 0] = 1 - \rho\{0\}\); take \(U \sim \Gleichv([0, 1])\) independent of \(\Theta\), then \(U^{1/(\beta m n)} \Theta \sim \Gleichv(K)\) (this follows from Remark~\ref{bem:k_symm}, part~1) and \(\Wsk[U^{1/(\beta m n)} > 0] = 1\), and by doing the same calculation as before,
\begin{equation*}
1 = \frac{\vol_{\beta m n}(\{x \in K \Colon s_m(x) > 0\})}{v_{\beta m n}(K)} = \Wsk[s_m(U^{1/(\beta m n)} \Theta) > 0] = \Wsk[s_m(\Theta) > 0].
\end{equation*}
This establishes \(\Wsk[s_m(X) > 0] = 1 - \rho\{0\}\), thus \(X \adj{X}\) is almost surely invertible iff \(\rho\{0\} = 0\).

\item
Let \(f \colon \!\Orth_{n, m; \beta} \to \RZ\) and \(g \colon \!\Pos_{m, \beta} \to \RZ\) be measurable and nonnegative, then
\begin{align*}
\Erw[f(\adj{X} (X \adj{X})^{-1/2}) g(X \adj{X})] &= \int_{\RZ_{\geq 0}} \int_{\Rand{K}} f(r \adj{\theta} (r^2 \theta \adj{\theta})^{-1/2}) g(r^2 \theta \adj{\theta}) \, \diff\KegM{K}{}(\theta) \, \diff\rho(r)\\
&= \int_{\Rand{K}} f(\adj{\theta} (\theta \adj{\theta})^{-1/2}) \int_{\RZ_{\geq 0}} g(r^2 \theta \adj{\theta}) \, \diff\rho(r) \, \diff\KegM{K}{}(\theta)\\
&= \frac{1}{\vol_{\beta m n}(K)} \int_K f(\adj{x} (x \adj{x})^{-1/2}) \int_{\RZ_{\geq 0}} g\Bigl( \frac{r^2 x \adj{x}}{\lvert x \rvert_K^2} \Bigr) \diff\rho(r) \, \diff x\\
&= \frac{2^{-m}}{\vol_{\beta m n}(K)} \int_{\Pos_{m, \beta}} \int_{\Orth_{n, m; \beta}} \Ind_K(t^{1/2} \adj{u}) f(u t^{1/2} (t^{1/2} \adj{u} u t^{1/2})^{-1/2}) \\
&\quad \cdot \int_{\RZ_{\geq 0}} g\Bigl( \frac{r^2 t^{1/2} \adj{u} u t^{1/2}}{\lvert t^{1/2} \adj{u} \rvert_K^2} \Bigr) \diff\rho(r) \, \diff u \det(t)^{\beta (n - m + 1)/2 - 1} \, \diff t.
\end{align*}
We have \(t^{1/2} \adj{u} \in K\) iff \(t^{1/2} \in \lvert K \rvert\); indeed, concerning ``\(\Rightarrow\)'' let \(t^{1/2} \adj{u} \in K\), then \((t^{1/2} \adj{u} u t^{1/2})^{1/2} = t^{1/2} \in \lvert K \rvert\). Concerning ``\(\Leftarrow\)'' let \(t^{1/2} \in \lvert K \rvert\), then there is some \(x \in K\) such that \(t^{1/2} = (x \adj{x})^{1/2}\); let \(x = \lvert x \rvert \adj{v}\) be its polar decomposition, then \(t^{1/2} = \lvert x \rvert\); and as \(K\) is right\-/unitarily invariant we get \(x \hat{v} \adj{\hat{u}} \in K\), and \(x \hat{v} \adj{\hat{u}} = \lvert x \rvert \adj{v} \hat{v} \adj{\hat{u}} = \lvert x \rvert \adj{u} = t^{1/2} \adj{u}\). This also implies \(\Ind_K(t^{1/2} \adj{u}) = \Ind_{\lvert K \rvert}(t^{1/2})\). Hence we may continue,
\begin{align*}
\Erw[f(\adj{X} (X \adj{X})^{-1/2}) g(X \adj{X})] &= \frac{2^{-m}}{\vol_{\beta m n}(K)} \int_{\Pos_{m, \beta}} \int_{\Orth_{n, m; \beta}} \Ind_{\lvert K \rvert}(t^{1/2}) f(u)\\
&\quad \cdot \int_{\RZ_{\geq 0}} g\Bigl( \frac{r^2 t}{\lvert t^{1/2} \rvert_K^2} \Bigr) \diff\rho(r) \, \diff u \det(t)^{\beta (n - m + 1)/2 - 1} \, \diff t\\
&= \frac{2^{-m}}{\vol_{\beta m n}(K)} \int_{\Orth_{n, m; \beta}} f(u) \, \diff u\\
&\quad \cdot \int_{\Pos_{m, \beta}} \Ind_{\lvert K \rvert}(t^{1/2}) \int_{\RZ_{\geq 0}} g\Bigl( \frac{r^2 t}{\lvert t^{1/2} \rvert_K^2} \Bigr) \, \diff\rho(r) \det(t)^{\beta (n - m + 1)/2 - 1} \, \diff t.
\end{align*}
This reveals the claimed independence of \(\adj{X} (X \adj{X})^{-1/2}\) and \(X \adj{X}\), and the distribution of \(\adj{X} (X \adj{X})^{-1/2}\).

\item
Let the distribution of \(X\) have \(\vol_{\beta m n}\)\=/density \(h \circ \lvert \cdot \rvert_K\) according to Proposition~\ref{sa:k_symm_absstet} and let \(g\) be as before, then,
\begin{align*}
\Erw[g(X \adj{X})] &= \int_{\KZ^{m \times n}} g(x \adj{x}) h(\lvert x \rvert_K) \, \diff x\\
&= 2^{-m} \int_{\Pos_{m, \beta}} \int_{\Orth_{n, m; \beta}} g(r^{1/2} \adj{u} u r^{1/2}) h(\lvert r^{1/2} \adj{u} \rvert_K) \, \diff u \det(r)^{\beta (n - m + 1)/2 - 1} \, \diff r\\
&= 2^{-m} \, \frac{\OrthVol_{n, m; \beta}}{2^{\beta m (m - 1)/4}} \int_{\Pos_{m, \beta}} g(r) h(\lvert r^{1/2} \rvert_{\lvert K \rvert}) \det(r)^{\beta (n - m + 1)/2 - 1} \, \diff r.
\end{align*}
This last line yields the claimed density of the distribution of \(X \adj{X}\). \qedhere
\end{asparaenum}
\end{proof}

\begin{Bem}\label{bem:simul_gleich_stiefel}
\begin{asparaenum}
\item Because \(\Kug{\Schatten{p}{}, \beta}{m \times n}\) is both left- and right\-/unitarily invariant, all of Proposition~\ref{sa:simul_gleich_stiefel} is applicable to it.

\item The classical case \(X \sim \Normal{\beta}^{\otimes (m \times n)}\) is recovered by noticing its density \linebreak\((\frac{\beta}{2 \pi})^{\beta m n/2} \ez^{-\beta \Spur(x \adj{x})/2} = (\frac{\beta}{2 \pi})^{\beta m n/2} \ez^{-\beta \lVert x \rVert_{\Schatten{2}{}}^2/2}\), that is, \(\Normal{\beta}^{\otimes (m \times n)}\) is \(\Kug{\Schatten{2}{}, \beta}{m \times n}\)\=/symmetric. As an aside, \(X \adj{X}\) is Wishart\-/distributed, which fact can also be inferred from part~3 given above.

\item Part~3 of Proposition~\ref{sa:simul_gleich_stiefel} is always satisfied for \(X \sim \Gleichv(\Kug{\Schatten{p}{}, \beta}{m \times n})\) (density \(\inv{(\KugVol{\Schatten{p}{}, \beta}{m \times n})} \Ind_{[0, 1]} \circ \lVert \cdot \rVert_{\Schatten{p}{}}\)).
\end{asparaenum}
\end{Bem}

Proposition~\ref{sa:simul_gleich_stiefel} partly overlaps with~\cite[Theorem~8.2.5]{GN2000}, as follows from the next statement. But in Proposition~\ref{sa:simul_gleich_stiefel} \(X\) need not have a \(\vol_{\beta m n}\)\=/density in order to ensure almost sure invertibility of \(X \adj{X}\) and to achieve \(\adj{X} (X \adj{X})^{-1/2} \sim \Gleichv(\Orth_{n, m; \beta})\). (Of course, left or right unitary invariance is not sufficient for invertibility of \(X \adj{X}\): just take \(X := V \diag(1, \dotsc, 1, 0) \adj{U}\) with \(U \sim \Gleichv(\Orth_{n; \beta})\) and \(V \sim \Gleichv(\Orth_{m; \beta})\), for instance.)

In the sequel \(\Perm{m}\) denotes the symmetric group of degree \(m \in \NZ\).

\begin{Sa}\label{sa:unitinvar}
A \(\KZ^{m \times n}\)\=/valued random variable \(X\) has a two\-/sided unitarily invariant distribution, that is \(X \GlVert v X u\) for all \(u \in \Orth_{n; \beta}\) and \(v \in \Orth_{m; \beta}\), if and only if there exist independent random variables \(S\), \(U\), and \(V\) such that \(S\) takes values in \(\RZ^m\), \(U \sim \Gleichv(\Orth_{n; \beta})\), \(V \sim \Gleichv(\Orth_{m; \beta})\), and \(X \GlVert V \diag(S) \adj{U}\).

Moreover, if \(X\) has a two\-/sided unitarily invariant distribution, then \(S\) can be constructed such that \(\Wsk[S \in \RZ_{\geq 0}^m] = 1\), it has exchangeable coordinates and then \((s_{\pi(i)}(X))_{i \leq m} \GlVert S\) with \(\pi \sim \Gleichv(\Perm{m})\) independent of \(X\), or such that \(\Wsk[S \in W^m] = 1\) (recall \(W^m = \{(x_i)_{i \leq m} \in \RZ^m \Colon 0 \leq x_m \leq \dotsb \leq x_1\}\)) and then \(s(X) \GlVert S\).
\end{Sa}

\begin{proof}
\(\Leftarrow\): This is trivial because \(\Gleichv(\Orth_{d; \beta})\) is a Haar measure and thus \(w W \GlVert W\) for \(W \sim \Gleichv(\Orth_{d; \beta})\) and any \(w \in \Orth_{d; \beta}\).

\(\Rightarrow\): Call \(\mu := \Wsk \circ \inv{X}\); then we have, for any \(f \colon \KZ^{m \times n} \to \RZ\) measurable and nonneagtive and any \(u \in \Orth_{n; \beta}\) and \(v \in \Orth_{m; \beta}\),
\begin{equation*}
\Erw[f(X)] = \Erw[f(v X u)] = \int_{\KZ^{m \times n}} f(v x u) \, \diff\mu(x).
\end{equation*}
Integrate on both sides over the unitary groups to obtain
\begin{align*}
\frac{\OrthVol_{m; \beta} \OrthVol_{n; \beta}}{2^{\beta m (m - 1)/4 + \beta n (n - 1)/4}} \Erw[f(X)] &= \int_{\Orth_{m; \beta}} \int_{\Orth_{n; \beta}} \int_{\KZ^{m \times n}} f(v x u) \, \diff\mu(x) \diff u \, \diff v\\
&= \int_{\KZ^{m \times n}} \int_{\Orth_{m; \beta}} \int_{\Orth_{n; \beta}} f(v x u) \, \diff u \, \diff v \, \diff\mu(x).
\end{align*}
For any \(x \in \KZ^{m \times n}\) choose a singular value decomposition \(x = v_x \diag(s(x)) \adj{u_x}\) and get
\begin{align*}
\int_{\Orth_{m; \beta}} \int_{\Orth_{n; \beta}} f(v x u) \, \diff u \, \diff v &= \int_{\Orth_{m; \beta}} \int_{\Orth_{n; \beta}} f(v v_x \diag(s(x)) \adj{u_x} u) \, \diff u \, \diff v\\
&= \int_{\Orth_{m; \beta}} \int_{\Orth_{n; \beta}} f(v \diag(s(x)) \adj{u}) \, \diff u \, \diff v,
\end{align*}
where we have used unitary invariance of the respective Haar measures and also its invariance under adjoining. This leads to (abbreviating the normalization constant by \(c\))
\begin{align*}
\Erw[f(X)] &= c \int_{\KZ^{m \times n}} \int_{\Orth_{m; \beta}} \int_{\Orth_{n; \beta}} f(v \diag(s(x)) \adj{u}) \, \diff u \, \diff v \, \diff\mu(x)\\
&= c \int_{W^m} \int_{\Orth_{m; \beta}} \int_{\Orth_{n; \beta}} f(v \diag(\sigma) \adj{u}) \, \diff u \, \diff v \, \diff(\mu \circ \inv{s})(\sigma),
\end{align*}
and this yields \(X = V \diag(S) \adj{U}\) with independent \(S\), \(U\), and \(V\) such that \(U \sim \Gleichv(\Orth_{n; \beta})\), \(V \sim \Gleichv(\Orth_{m; \beta})\), and \(S \sim \mu \circ \inv{s} = \Wsk{} \circ \inv{s(X)}\), so \(s(X) \GlVert S\).

Now let \(P \sim \Gleichv(\Perm{m})\) be independent of all other random variables and set \(T := (S_{P(i)})_{i \leq m}\), then \(T \in \RZ_{\geq 0}^m\) almost surely, \(T\) has exchangeable coordinates, and \(T \GlVert (s_{P(i)}(X))_{i \leq m}\), and we also have
\begin{align*}
\Erw[f(V T \adj{U})] &= c \, \frac{1}{m!} \sum_{\pi \in \Perm{m}} \int_{W^m} \int_{\Orth_{m; \beta}} \int_{\Orth_{n; \beta}} f(v \diag((\sigma_{\pi(i)})_{i \leq m}) \adj{u}) \, \diff u \, \diff v \, \diff(\mu \circ \inv{s})(\sigma)\\
&= c \, \frac{1}{m!} \sum_{\pi \in \Perm{m}} \int_{W^m} \int_{\Orth_{m; \beta}} \int_{\Orth_{n; \beta}} f(v \diag(\sigma) \adj{u}) \, \diff u \, \diff v \, \diff(\mu \circ \inv{s})(\sigma)\\
&= c \int_{W^m} \int_{\Orth_{m; \beta}} \int_{\Orth_{n; \beta}} f(v \diag(\sigma) \adj{u}) \, \diff u \, \diff v \, \diff(\mu \circ \inv{s})(\sigma)\\
&= \Erw[f(X)],
\end{align*}
where from the first to second line we have used \(\diag((\sigma_{\pi(i)})_{i \leq m}) = \tau_\pi \diag(\sigma) \trapo{\diag(\tau_\pi, I_{n - m})}\) with \(\tau_\pi = (\delta_{\pi(i), j})_{(i, j) \in n \times n}\) and unitary invariance, as in Remark~\ref{bem:weylkammer}. This proves \(X \GlVert V T \adj{U}\).
\end{proof}

\begin{Bem}
Proposition~\ref{sa:unitinvar} tells us that if the distribution of \(X\) is two\-/sided unitarily invariant, then it is alredy completely determined by the distribution of the singular values of \(X\), in the sense that \(X \GlVert V \diag(s(X)) \adj{U}\) with independent \(U \sim \Gleichv(\Orth_{n; \beta})\) and \(V \sim \Gleichv(\Orth_{m; \beta})\).
\end{Bem}

\begin{Sa}\label{sa:bqsymm_orthinvar}
Any \(\Kug{\Schatten{p}{}, \beta}{m \times n}\)\=/symmetric measure on \(\KZ^{m \times n}\) is two\-/sided unitarily invariant.
\end{Sa}

\begin{proof}
Let \(\mu\) be a \(\Kug{\Schatten{p}{}, \beta}{m \times n}\)\=/symmetric measure on \(\KZ^{m \times n}\) with radial measure \(\rho\), and let \(u \in \Orth_{n; \beta}\) and \(v \in \Orth_{m; \beta}\). Keeping in mind that singular values, and hence also \(\lVert \cdot \rVert_{\Schatten{p}{}}\) and \(\Kug{\Schatten{p}{}, \beta}{m \times n}\), are two\-/sided unitarily invariant, let \(f \colon \KZ^{m \times n} \to \RZ\) be measurable and nonnegative, then we have for any \(r \in \RZ_{\geq 0}\)
\begin{align*}
\int_{\Sph{\Schatten{p}{}, \beta}{m, n}} f(r v \theta u) \, \diff\KegM{\Schatten{p}{}, \beta}{m \times n}(\theta) &= \frac{1}{\KugVol{\Schatten{p}{}, \beta}{m \times n}} \int_{\KZ^{m \times n}} \Ind_{\Kug{\Schatten{p}{}, \beta}{m \times n}}(x) f\Bigl( \frac{r v x u}{\lVert x \rVert_{\Schatten{p}{}}} \Bigr) \, \diff x\\
&= \frac{1}{\KugVol{\Schatten{p}{}, \beta}{m \times n}} \int_{\KZ^{m \times n}} \Ind_{\Kug{\Schatten{p}{}, \beta}{m \times n}}(\adj{v} x \adj{u}) f\Bigl( \frac{r x}{\lVert \adj{v} x \adj{u} \rVert_{\Schatten{p}{}}} \Bigr) \, \diff x\\
&= \frac{1}{\KugVol{\Schatten{p}{}, \beta}{m \times n}} \int_{\KZ^{m \times n}} \Ind_{\Kug{\Schatten{p}{}, \beta}{m \times n}}(x) f\Bigl( \frac{r x}{\lVert x \rVert_{\Schatten{p}{}}} \Bigr) \, \diff x\\
&= \int_{\Sph{\Schatten{p}{}, \beta}{m, n}} f(r \theta) \, \diff\KegM{\Schatten{p}{}, \beta}{m \times n}(\theta).
\end{align*}
Now intgerate w.r.t.\ \(\rho\) in order to get
\begin{align*}
\int_{\KZ^{m \times n}} f(v x u) \, \diff\mu(x) &= \int_{\RZ_{\geq 0}} \int_{\Sph{\Schatten{p}{}, \beta}{m, n}} f(r v \theta u) \, \diff\KegM{\Schatten{p}{}, \beta}{m \times n}(\theta) \, \diff\rho(r)\\
&= \int_{\RZ_{\geq 0}} \int_{\Sph{\Schatten{p}{}, \beta}{m, n}} f(r \theta) \, \diff\KegM{\Schatten{p}{}, \beta}{m \times n}(\theta) \, \diff\rho(r) = \int_{\KZ^{m \times n}} f(x) \, \diff\mu(x),
\end{align*}
and the claim follows.
\end{proof}

The next theorem is a major generalization of the probabilistic representation of singular values of random matrices as used in \cite{KPTh2020_3} and \cite{KPTh2020_2}.

\begin{Thm}\label{sa:gleichvert_schattenq_kugel_sphaere}
Let \(X\) be a \(\KZ^{m \times n}\)\=/valued random variable.
\begin{compactenum}
\item
\(X\) has a \(\Kug{\Schatten{p}{}, \beta}{m \times n}\)\-/symmetric distribution with radial distribution \(\rho\) if and only if there exist independent random variables \(R\), \(U\), \(V\), and \(Y\) such that \(R \sim \rho\), \(U \sim \Gleichv(\Orth_{n; \beta})\), \(V \sim \Gleichv(\Orth_{m; \beta})\), and \(Y = (Y_i)_{i \leq m}\) is \(\RZ^m\)\-/valued with Lebesgue\-/density
\begin{equation}\label{eq:dichte_y}
g(y_1, \dotsc, y_m) := \frac{1_{\RZ_{> 0}^m}(y)}{Z_{m, n, p, \beta}} \ez^{-\beta n \lVert y \rVert_{p/2}^{p/2}} \prod_{i = 1}^m y_i^{\beta (n - m + 1)/2 - 1} \prod_{1 \leq i < j \leq m} \lvert y_i - y_j \rvert^\beta
\end{equation}
(here \(Z_{m, n, p, \beta} \in \RZ_{> 0}\) is the appropriate normalization constant, and we interpret \(\ez^{-\beta n \lVert y \rVert_{p/2}^{p/2}} := \Ind_{\Kug{\infty, 1}{m}}(y)\) for \(p = \infty\)), and
\begin{equation*}
X \GlVert R \, V \, \dfrac{\diag\bigl( (Y_i^{1/2})_{i \leq m} \bigr)}{\lVert Y \rVert_{p/2}^{1/2}} \, \adj{U}.
\end{equation*}

\item
Special cases of 1.:
\begin{compactenum}
\item \(X \sim \Gleichv(\Kug{\Schatten{p}{}, \beta}{m \times n})\) iff \(R \GlVert W^{1/(\beta m n)}\), where \(W \sim \Gleichv([0, 1])\) is independent of all other random variables.
\item \(X \sim \KegM{\Schatten{p}{}, \beta}{m \times n}\) iff \(R = 1\) almost surely.
\end{compactenum}

\item
For the particular value \(p = \infty\) we have the following additional representations:
\begin{compactenum}
\item \(X \sim \Gleichv(\Kug{\Schatten{\infty}{}, \beta}{m \times n})\) if and only if \(X \GlVert V \diag((Y_i^{1/2})_{i \leq m}) \adj{U}\), where \(U\), \(V\), and \(Y\) have the same meanings as in 1.
\item \(X \sim \KegM{\Schatten{\infty}{}, \beta}{m \times n}\) if and only if \(X \GlVert V \diag\bigl( (Y_i^{1/2})_{i \leq m} \bigr) \adj{U}\), where \(U\) und \(V\) have the same meanings as in~1., and \(Y = \trapo{(Y_1, \dotsc, Y_m)}\) is \(\RZ^m\)\-/valued with \(Y_1 = 1\) almost surely and \((Y_2, \dotsc, Y_m)\) has Lebesgue\-/density proportional to
\begin{equation*}
(y_2, \dotsc, y_m) \mapsto g(1, y_2, \dotsc, y_m).
\end{equation*}
\end{compactenum}
\end{compactenum}
\end{Thm}

\begin{proof}
For ease of notation we abbreviate normalization constants (whose value also may change from one occurrence to the other) and set \(h(x) := \Ind_{\RZ_{> 0}^m}(x) \cdot \prod_{i = 1}^m x_i^{\beta (n - m + 1) - 1} \cdot \prod_{1 \leq i < j \leq m} \lvert x_i^2 - x_j^2 \rvert^\beta\).

\begin{asparaenum}
\item
\(\Rightarrow\): From Proposition~\ref{sa:k_symm_masz_unabh} we know \(X \GlVert R \Theta\) where \(R\) and \(\Theta\) are independent, \(R \sim \rho\) and \(\Theta \sim \KegM{\Schatten{p}{}, \beta}{m \times n}\). Since \(\KegM{\Schatten{p}{}, \beta}{m \times n}\) is \(\Kug{\Schatten{p}{}, \beta}{m \times n}\)\=/symmetric (see Remark~\ref{bem:k_symm}), from Propositions~\ref{sa:unitinvar} and~\ref{sa:bqsymm_orthinvar} we know \(\Theta \GlVert V \diag(S) \adj{U}\) with independent variables \(S\), \(U\), and \(V\) such that \(S \GlVert (s_{P(i)}(\Theta))_{i \leq m}\), \(U \sim \Gleichv(\Orth_{n; \beta})\), and \(V \sim \Gleichv(\Orth_{m; \beta})\), where \(P \sim \Gleichv(\Perm{m})\) is independent of all others. Thence it suffices to establish \(S \GlVert \lVert Y \rVert_{p/2}^{-1/2} (Y_i^{1/2})_{i \leq m}\) where \(Y\) has the claimed distribution. We further define \(x_\pi := (x_{\pi(i)})_{i \leq m}\) for \(x \in \KZ^m\) and \(\pi \in \Perm{m}\). So let \(f \colon \RZ^m \to \RZ\) be measurable and nonnegative, then
\begin{equation*}
\Erw[f(S)] = \Erw[f(s(\Theta)_P)] = \frac{1}{m!} \sum_{\pi \in \Perm{m}} \int_{\Sph{\Schatten{p}{}, \beta}{m, n}} f(s(\theta)_\pi) \, \diff\KegM{\Schatten{p}{}, \beta}{m \times n}(\theta);
\end{equation*}
express the integral over \(\Sph{\Schatten{p}{}, \beta}{m, n}\) as one over \(\Kug{\Schatten{p}{}, \beta}{m \times n}\) and use the positive homogeneity of singular values to get
\begin{equation*}
\Erw[f(S)] = \frac{C}{m!} \sum_{\pi \in \Perm{m}} \int_{\Kug{\Schatten{p}{}, \beta}{m \times n}} f\Bigl( \frac{s(x)_\pi}{\lVert x \rVert_{\Schatten{p}{}}} \Bigr) \, \diff x;
\end{equation*}
now use Proposition~\ref{sa:integraltrafo}, part~2, noting that the integrand already depends only on the singular values of \(x\),
\begin{equation*}
\Erw[f(S)] = \frac{C_1}{m!} \sum_{\pi \in \Perm{m}} \int_{\Kug{p, 1}{m}} f\Bigl( \frac{s(\diag(\sigma))_\pi}{\lVert \sigma \rVert_p} \Bigr) \, h(\sigma) \, \diff\sigma;
\end{equation*}
as in Remark~\ref{bem:weylkammer} we split \(\RZ_{> 0}^m = \bigcup_{\pi' \in \Perm{m}} \tau_{\pi'} W^m\) (up to null sets), transform \(\tau_{\pi'} W^m\) to \(W^m\) while bearing the symmetry of \(\lVert \cdot \rVert_p\) and \(h\) in mind, and use \(s(\diag(\sigma_{\pi'})) = \sigma\) for \(\sigma \in W^m\), so
\begin{align*}
\int_{\Kug{p, 1}{m}} f\Bigl( \frac{s(\diag(\sigma))_\pi}{\lVert \sigma \rVert_p} \Bigr) \, h(\sigma) \, \diff\sigma &= \sum_{\pi' \in \Perm{m}} \int_{\tau_{\pi'} W^m} \Ind_{\Kug{p, 1}{m}}(\sigma) f\Bigl( \frac{s(\diag(\sigma))_\pi}{\lVert \sigma \rVert_p} \Bigr) \, h(\sigma) \, \diff\sigma\\
&= \sum_{\pi' \in \Perm{m}} \int_{W^m} \Ind_{\Kug{p, 1}{m}}(\sigma) f\Bigl( \frac{s(\diag(\sigma_{\pi'}))_\pi}{\lVert \sigma \rVert_p} \Bigr) \, h(\sigma) \, \diff\sigma\\
&= \sum_{\pi' \in \Perm{m}} \int_{W^m} \Ind_{\Kug{p, 1}{m}}(\sigma) f\Bigl( \frac{\sigma_\pi}{\lVert \sigma \rVert_p} \Bigr) \, h(\sigma) \, \diff\sigma\\
&= m! \int_{\tau_{\inv{\pi}} W^m} \Ind_{\Kug{p, 1}{m}}(\sigma) f\Bigl( \frac{\sigma}{\lVert \sigma \rVert_p} \Bigr) \, h(\sigma) \, \diff\sigma;
\end{align*}
therefore
\begin{align*}
\Erw[f(S)] &= C_1 \sum_{\pi \in \Perm{m}} \int_{\tau_{\inv{\pi}} W^m} \Ind_{\Kug{p, 1}{m}}(\sigma) f\Bigl( \frac{\sigma}{\lVert \sigma \rVert_p} \Bigr) \, h(\sigma) \, \diff\sigma\\
&= C_1 \int_{\RZ^m} \Ind_{\Kug{p, 1}{m}}(\sigma) f\Bigl( \frac{\sigma}{\lVert \sigma \rVert_p} \Bigr) \, h(\sigma) \, \diff\sigma.
\end{align*}
Introduce polar coordinates in \(\RZ^m\) and use the homogeneity of \(h\) to arrive at
\begin{equation*}
\Erw[f(S)] = C_2 \int_{\Sph{p, 1}{m}} f(\theta) h(\theta) \, \diff\KegM{p, 1}{m}(\theta),
\end{equation*}
which shows that the distribution of \(S\) has \(\KegM{p, 1}{m}\)\=/density proportional to \(h\). It remains to prove that \(\lVert Y \rVert_{p/2}^{-1/2} (Y_i^{1/2})_{i \leq m}\) also has \(\KegM{p, 1}{m}\)\=/density proportional to \(h\); so let \(f\) be as before, then via substitution \(x := (y_i^{1/2})_{i \leq m}\) and polar integration, and noting that \(g\) and \(h\) contain the indicator function \(\Ind_{\RZ_{> 0}^m}\),
\begin{align*}
\Erw\biggl[ f\biggl( \frac{(Y_i^{1/2})_{i \leq m}}{\lVert Y \rVert_{p/2}^{1/2}} \biggr) \biggr] &= \int_{\RZ^m} f\biggl( \frac{(y_i^{1/2})_{i \leq m}}{\lVert y \rVert_{p/2}^{1/2}} \biggr) g(y) \, \diff y\\
&= 2^m \int_{\RZ^m} f\Bigl( \frac{x}{\lVert x \rVert_p} \Bigr) g((x_i^2)_{i \leq m}) \prod_{i = 1}^m x_i \, \diff x\\
&= C \int_{\RZ^m} f\Bigl( \frac{x}{\lVert x \rVert_p} \Bigr) \ez^{-\beta n \lVert x \rVert_p^p} h(x) \, \diff x\\
&= C_1 \int_{\RZ_{\geq 0}} \int_{\Sph{p, 1}{m}} f\Bigl( \frac{r \theta}{\lVert r \theta \rVert_p} \Bigr) \ez^{-\beta n \lVert r \theta \rVert_p^p} h(r \theta) r^{m - 1} \, \diff\KegM{p, 1}{m}(\theta) \, \diff r\\
&= C_1 \int_{\RZ_{\geq 0}} r^{\beta m n - 1} \ez^{-\beta n r^p} \, \diff r \int_{\Sph{p, 1}{m}} f(\theta) h(\theta) \, \diff\KegM{p, 1}{m}(\theta),
\end{align*}
where for the last step we have invested the homogeneity of \(h\); and we have finished.

{\Absatz}\(\Leftarrow\): Because of Proposition~\ref{sa:k_symm_masz_unabh} it suffices to prove \(\Theta := \lVert Y \rVert_{p/2}^{-1/2} \, V \diag((Y_i^{1/2})_{i \leq m}) \adj{U} \sim \KegM{\Schatten{p}{}, \beta}{m \times n}\), which makes sense since \(\lVert \Theta \rVert_{\Schatten{p}{}} = \lVert Y \rVert_{p/2}^{-1/2} \lVert \diag((Y_i^{1/2})_{i \leq m}) \rVert_{\Schatten{p}{}} = 1\). From the proof of \(\Rightarrow\) we already know that \(\lVert Y \rVert_{p/2}^{-1/2} (Y_i^{1/2})_{i \leq m}\) has \(\KegM{p, 1}{m}\)\=/density proportional to \(h\). So let \(f \colon \Sph{\Schatten{p}{}, \beta}{m, n} \to \RZ\) be measurable and nonnegative, then, by immediately converting back to \(\RZ^m\) we get
\begin{align*}
\Erw[f(\Theta)] &= C \int_{\Orth_{m; \beta}} \int_{\Orth_{n; \beta}} \int_{\Sph{p, 1}{m}} f(v \diag(\theta) \adj{u}) h(\theta) \, \diff\KegM{p, 1}{m}(\theta) \, \diff u \, \diff v\\
&= C_1 \int_{\Orth_{m; \beta}} \int_{\Orth_{n; \beta}} \int_{\Kug{p, 1}{m}} f\Bigl( \frac{v \diag(x) \adj{u}}{\lVert x \rVert_p} \Bigr) \frac{h(x)}{\lVert x \rVert_p^{\beta m n - m}} \, \diff x \, \diff u \, \diff v.
\end{align*}
Use \(\lVert x \rVert_p = \lVert v \diag(x) \adj{u} \rVert_{\Schatten{p}{}}\) and Proposition~\ref{sa:integraltrafo}, part~2, in order to revert to \(\KZ^{m \times n}\), so
\begin{equation*}
\Erw[f(\Theta)] = C_2 \int_{\Kug{\Schatten{p}{}, \beta}{m \times n}} f\Bigl( \frac{x}{\lVert x \rVert_{\Schatten{p}{}}} \Bigr) \lVert x \rVert_{\Schatten{p}{}}^{m - \beta m n} \, \diff x;
\end{equation*}
finally introduce polar coordinates on \(\KZ^{m \times n}\) to achieve
\begin{align*}
\Erw[f(\Theta)] &= C_3 \int_{\RZ_{> 0}} \int_{\Sph{\Schatten{p}{}, \beta}{m, n}} \Ind_{\Kug{\Schatten{p}{}, \beta}{m \times n}}(r \theta) f\Bigl( \frac{r \theta}{\lVert r \theta \rVert_{\Schatten{p}{}}} \Bigr) \lVert r \theta \rVert_{\Schatten{p}{}}^{m - \beta m n} \, r^{\beta m n - 1} \, \diff\KegM{\Schatten{p}{}, \beta}{m \times n}(\theta) \, \diff r\\
&= \frac{C_3}{m} \int_{\Sph{\Schatten{p}{}, \beta}{m, n}} f(\theta) \, \diff\KegM{\Schatten{p}{}, \beta}{m \times n}(\theta),
\end{align*}
and this proves the claim.

\item
We refer to Remark~\ref{bem:k_symm}: a \(K\)\=/symmetric measure equals \(\Gleichv(K)\) iff \(\frac{\diff\rho(r)}{\diff r} = d r^{d - 1} \Ind_{[0, 1]}(r)\), and it equals \(\KegM{K}{}\) iff \(\rho = \delta_1\). Since we know \(R \sim \rho\), the claim is immediate for \(\KegM{\Schatten{p}{}, \beta}{m \times n}\), and in the case of \(\Gleichv(\Kug{\Schatten{p}{}, \beta}{m \times n})\) observe \(R^{\beta m n} \sim \Gleichv([0, 1])\).

\item
\begin{asparaenum}
\item
We start anew from Proposition~\ref{sa:integraltrafo}, part~2; let \(f \colon \KZ^{m \times n} \to \RZ\) be measurable and nonnegative, then, via the usual substitution \(s = (y_i^{1/2})_{i \leq m}\),
\begin{align*}
\Erw[f(X)] &= \frac{1}{\KugVol{\Schatten{\infty}{}, \beta}{m \times n}} \int_{\KZ^{m \times n}} \Ind_{\Kug{\Schatten{\infty}{}, \beta}{m \times n}}(x) f(x) \, \diff x\\
&= C \int_{\RZ_{> 0}^m} \int_{\Orth_{m; \beta}} \int_{\Orth_{n; \beta}} \Ind_{\Kug{\Schatten{\infty}{}, \beta}{m \times n}}(v \diag(s) \adj{u}) f(v \diag(s) \adj{u}) h(s) \, \diff u \, \diff v \, \diff s\\
&= C 2^{-m} \int_{\RZ_{> 0}^m} \int_{\Orth_{m; \beta}} \int_{\Orth_{n; \beta}} \Ind_{\Kug{\infty, 1}{m}}(y) f(v \diag((y_i^{1/2})_{i \leq m}) \adj{u})\\
&\qquad \cdot h((y_i^{1/2})_{i \leq m}) \prod_{i = 1}^m y_i^{-1/2} \, \diff u \, \diff v \, \diff y\\
&= C_1 \int_{\RZ_{> 0}^m} \int_{\Orth_{m; \beta}} \int_{\Orth_{n; \beta}} f(v \diag((y_i^{1/2})_{i \leq m}) \adj{u}) g(y) \, \diff u \, \diff v \, \diff y,
\end{align*}
where recall our convention \(\ez^{-\beta n \lVert y \rVert_\infty^{\infty}} = \Ind_{\Kug{\infty, 1}{m}}(y)\). The last line implies the claim.

\item
We start from~2.; let \(f \colon \Sph{\Schatten{\infty}{}, \beta}{m, n} \to \RZ\) be measurable and nonnegative, then,
\begin{align*}
\Erw[f(X)] &= C \int_{\RZ_{> 0}^m} \int_{\Orth_{m; \beta}} \int_{\Orth_{n; \beta}} f\biggl( \frac{v \diag((y_i^{1/2})_{i \leq m}) \adj{u}}{\lVert y \rVert_\infty^{1/2}} \biggr) g(y) \, \diff u \, \diff v \, \diff y\\
&= C_1 \int_{\RZ_{> 0}} \int_{\Sph{\infty, 1}{m}} \int_{\Orth_{m; \beta}} \int_{\Orth_{n; \beta}} f\biggl( \frac{v \diag(((r \theta_i)^{1/2})_{i \leq m}) \adj{u}}{\lVert r \theta \rVert_\infty^{1/2}} \biggr)\\
&\qquad \cdot g(r \theta) r^{m - 1} \, \diff u \, \diff v \, \diff\KegM{\infty, 1}{m}(\theta) \, \diff r\\
&= C_1 \int_{(0, 1]} r^{\beta m n/2 - 1} \, \diff r \int_{\Sph{\infty, 1}{m}} \int_{\Orth_{m; \beta}} \int_{\Orth_{n; \beta}} f\bigl( v \diag((\theta_i^{1/2})_{i \leq m}) \adj{u} \bigr)\\
&\qquad \cdot g(\theta) \, \diff u \, \diff v \, \diff\KegM{\infty, 1}{m}(\theta),
\end{align*}
where we have introduced polar coordinates and used the homogeneity of \(g\), to wit \(g(r \theta) = \Ind_{[0, 1]}(r) r^{\beta m n/2 - m} \, g(\theta)\). Now essentially \(\Sph{\infty, 1}{m} \cap \RZ_{> 0}^m\) consists of those faces of \([0, 1]^m\) where at least one coordinate equals \(1\), and they all are permutations of \(\{1\} \times [0, 1]^{m - 1}\); the matrix \(\diag(\theta)\) is permuted accordingly, but as we are integrating w.r.t.\ the Haar\-/measures on \(\Orth_{m; \beta}\) and \(\Orth_{n; \beta}\), the value of the integral over each of the faces is the same; cf.\ Remark~\ref{bem:weylkammer} for more details. Thus we get
\begin{align*}
\Erw[f(X)] &= C_2 \int_{[0, 1]^{m - 1}} \int_{\Orth_{m; \beta}} \int_{\Orth_{n; \beta}} f(v \diag(1, \theta_2^{1/2}, \dotsc, \theta_m^{1/2}) \adj{u})\\
&\qquad \cdot g(1, \theta_2, \dotsc, \theta_m) \, \diff u \, \diff v \, \diff\theta
\end{align*}
and the result follows. \qedhere
\end{asparaenum}
\end{asparaenum}
\end{proof}

\section{Sanov\-/type LDPs for singular values}
\label{sec:sanovldp}

\subsection{Prerequisites}
\label{sec:ldp_voraussetzungen}

At first glance it may seem an odd choice that Theorem~\ref{sa:gleichvert_schattenq_kugel_sphaere} contains a representation for the squares of the singular values instead of the singular values themselves. The main reason is the application of results which lead to large deviations principles.

Having a broad readership in mind, we shall recall the basic definitions from the theory of large deviations principles; more information may be found in textbooks like \cite{DZ2010}, \cite{denHol2000}, \cite{RAS2015}, or the survey article \cite{jp2024}, which also provides the reader with the current state of the art concerning large deviations theory in geometric functional analysis. In what follows let \(E\) be a Polish space (a complete separable metrizable topological space) and by \(\Borel_E\) denote its Borel \(\sigma\)\=/algebra.

\begin{Def}[rate function]\label{def:ratenfunktion}
A \emph{rate function} is a lower semicontinuous map \(\mathcal{I} \colon E \to [0, \infty]\) which is not identical \(\infty\). A rate function \(\mathcal{I}\) is called a \emph{good rate function} iff for any \(y \in \RZ_{\geq 0}\) the sublevel set \(\inv{\mathcal{I}}([0, y])\) is compact.
\end{Def}

\begin{Def}[large deviations principle]\label{def:ldp}
A sequence \((\xi_n)_{n \in \NZ}\) of \(E\)\-/valued random variables is said to satisfy a \emph{large deviations principle} (LDP) with speed \((s_n)_{n \in \NZ}\) and (good) rate function \(\mathcal{I} \colon E \to [0, \infty]\), iff \((s_n)_{n \in \NZ}\) is a sequence in \(\RZ_{> 0}\) with \(\lim_{n \to \infty} s_n = \infty\), \(\mathcal{I}\) is a (good) rate function in the sense of Definition~\ref{def:ratenfunktion}, and for any \(A \in \Borel_E\) the following chain of inequalities holds true,
\begin{equation*}
-\inf\bigl( \mathcal{I}(\offen{A}) \bigr) \leq \liminf_{n \to \infty} \frac{1}{s_n} \log \Wsk[\xi_n \in A] \leq \limsup_{n \to \infty} \frac{1}{s_n} \log \Wsk[\xi_n \in A] \leq -\inf\bigl( \mathcal{I}(\schluss{A}) \bigr),
\end{equation*}
where \(\offen{A}\) and \(\schluss{A}\) denote the interior and closure of \(A\) resp.
\end{Def}

\begin{Def}[exponentially tight]\label{def:exp_straff}
A sequence \((\xi_n)_{n \in \NZ}\) of \(E\)\=/valued random variables is \emph{exponentially tight} with speed \((s_n)_{n \in \NZ}\) iff \((s_n)_{n \in \NZ}\) is as in Definition~\ref{def:ldp} and there holds true,
\begin{equation*}
\inf_{C} \limsup_{n \to \infty} \frac{1}{s_n} \log \Wsk[\xi_n \notin C] = -\infty,
\end{equation*}
where \(C\) ranges over all compact subsets of \(E\).
\end{Def}

There are several tools for proving the existence of an LDP; one which starts form first principles is the following.

\begin{Sa}
Let \(\mathcal{T}\) be a basis of the topology on \(E\). Let \((\xi_n)_{n \in \NZ}\) be a sequence of \(E\)\=/valued random variables and let \((s_n)_{n \in \NZ}\) be a sequence in \(\RZ_{> 0}\) with \(\lim_{n \to \infty} s_n = \infty\). If there exists a function \(\mathcal{I} \colon E \to [0, \infty]\) such that, for all \(a \in E\),
\begin{equation*}
\inf_{\substack{U \in \mathcal{T} \\ a \in U}} \liminf_{n \to \infty} \frac{1}{s_n} \log \Wsk[\xi_n \in U] = -\mathcal{I}(a) = \inf_{\substack{U \in \mathcal{T} \\ a \in U}} \limsup_{n \to \infty} \frac{1}{s_n} \log \Wsk[\xi_n \in U],
\end{equation*}
and if in addition \((\xi_n)_{n \in \NZ}\) is exponentially tight with speed \((s_n)_{n \in \NZ}\), then \((\xi_n)_{n \in \NZ}\) satisfies a large deviations principle with speed \((s_n)_{n \in \NZ}\) and good rate function \(\mathcal{I}\).
\end{Sa}

A very useful result to transform an already existing LDP on \(E\) to one on another Polish space \(F\) is the following.

\begin{Sa}[contraction principle]\label{sa:kontraktionsprinzip}
Let \((\xi_n)_{n \in \NZ}\) be a sequence of \(E\)\=/valued random variables that satisfies a large deviations principle with speed \((s_n)_{n \in \NZ}\) and good rate function \(\mathcal{I} \colon E \to [0, \infty]\), and let \(f \colon E \to F\) be a continuous map. Then the sequence \((f(\xi_n))_{n \in \NZ}\) satisfies a large deviations principle in \(F\) with the same speed \((s_n)_{n \in \NZ}\) and good rate function \(\mathcal{J} \colon F \to [0, \infty]\) given by
\begin{equation*}
\mathcal{J}(y) := \inf_{x \in \inv{f}\{y\}} \mathcal{I}(x),
\end{equation*}
where \(\inf(\emptyset) := \infty\).
\end{Sa}

Given a Polish space \(E\), define \(\WMasz(E)\) to be the set of all probability measures on \(E\). The standard topology on \(\WMasz(E)\) is the \emph{weak topology,} that is, the coarsest topology such that the functional
\begin{equation*}
\WMasz(E) \ni \mu \mapsto \int_E f \, \diff\mu
\end{equation*}
is continuous for any given \(f \in \CeBe(E)\), the set of all bounded continuous real\-/valued functions on \(E\). It follows that for \(\mu \in \WMasz(E)\) a basis of open neighborhoods consists of sets \(\mathcal{O}_{f_1, \dotsc, f_d; \epsilon}\), where
\begin{equation*}
\mathcal{O}_{f_1, \dotsc, f_d; \epsilon} := \biggl\{ \nu \in \WMasz(E) \Colon \forall i \in \{1, \dotsc, d\} \colon \biggl\lvert \int_E f_i \, \diff\nu - \int_E f_i \, \diff\mu \biggr\rvert < \epsilon \biggr\},
\end{equation*}
with \(d \in \NZ\), \(\epsilon \in \RZ_{> 0}\), and \(f_1, \dotsc, f_d \in \CeBe(E)\).

The weak topology on \(\WMasz(E)\) may be metrized, for instance with the Lévy\--Prokhorov metric or with the bounded Lipschitz metric \cite[Chapter~11.3]{Dudley2002}, and endowed with such a metric \(\WMasz(E)\) becomes a Polish space itself. For this reason it makes sense to consider LDPs for sequences of random probability measures, and in the sequel any set of probability measures is to be understood as endowed with the weak topology.

\subsection{Results}
\label{sec:ldp_ergebnisse}

For any array \(((X_{n, i})_{i \leq m(n)})_{n \in \NZ}\) of (say) real\-/valued random variables the empirical measures \(\mu_n := \frac{1}{m(n)} \sum_{i = 1}^{m(n)} \delta_{X_{n, i}}\) (where \(\delta_a\) denotes the Dirac measure supported on \(a\)) form a sequence of random probability measures, and the following Theorems~\ref{sa:ldp_singulaerwerte} and~\ref{sa:ldp_singulaerwerte_unendl} treat the situation when the input array consists of the (scaled) singular values of a suitably distributed random matrix.

\begin{Thm}[\(p < \infty\)]\label{sa:ldp_singulaerwerte}
Let \(p \in (0, \infty)\), let \(m \in \NZ\) depend on \(n \in \NZ\) such that \(m \leq n\) and \(c := \lim_{n \to \infty} \frac{m}{n} \in (0, 1]\) exists, and let \((X^{(n)})_{n \in \NZ}\) be a sequence of random variables such that for each \(n \in \NZ\) either \(X^{(n)} \sim \KegM{\Schatten{p}{}, \beta}{m \times n}\) or \(X^{(n)} \sim \Gleichv(\Kug{\Schatten{p}{}, \beta}{m \times n})\). Then the sequence \((\mu_n)_{n \in \NZ}\) of empirical measures \(\mu_n := \frac{1}{m} \sum_{i = 1}^m \delta_{m^{1/p} s_i(X^{(n)})}\) satisfies a large deviations principle on \(\WMasz(\RZ_{\geq 0})\) with speed \(\beta m n\) and good rate function \(\mathcal{I}_{c, p} \colon \WMasz(\RZ_{\geq 0}) \to [0, \infty]\) given by
\begin{equation*}
\mathcal{I}_{c, p}(\mu) := \begin{cases} -\frac{c}{2} \int_{\RZ_{\geq 0}^2} \log\lvert x^2 - y^2 \rvert \, \diff\mu^{\otimes 2}(x, y) - (1 - c) \int_{\RZ_{\geq 0}} \log(x) \, \diff\mu(x) + \frac{\log(\ez p)}{p} + B_{c, p}, & \text{if } m_p(\mu) \leq 1, \\ \infty & \text{else,} \end{cases}
\end{equation*}
which possesses a unique global mimimizer \(\mu_{c, p} \in \WMasz(\RZ_{\geq 0})\) with compact support and some Lebesgue density. Here
\begin{equation*}
B_{c, p} := \lim_{n \to \infty} \frac{1}{\beta m n} \log(Z_{m, n, p, \beta}) \in \RZ,
\end{equation*}
with \(Z_{m, n, p, \beta}\) as in Theorem~\ref{sa:gleichvert_schattenq_kugel_sphaere}, and \(m_p(\mu) := \int_{\RZ_{\geq 0}} x^p \, \diff\mu(x)\).
\end{Thm}

\begin{Bem}\label{bem:ldp_singulaerwerte}
Unfortunately neither the minimizer \(\mu_{c, p}\) nor the constant \(B_{c, p}\) are explicitly known for most values of \(c\) and \(p\). The exceptions are \(c = 1\) for arbitrary \(p \in (0, \infty)\) since this corresponds to the quadratic case already solved in \cite{KPTh2020_2}, and \(p = 2\) for arbitrary \(c \in (0, 1]\) since here the theory for Wishart matrices is applicable; further information for the latter case is given in Remark~\ref{bem:minimierer_p2} below; and for selected values of \(c\) the density of \(\mu_{c, 2}\) is plotted in Figure~\ref{fig:dichtenZwei}.
\end{Bem}

\begin{figure}
\centering\includegraphics[scale=0.5]{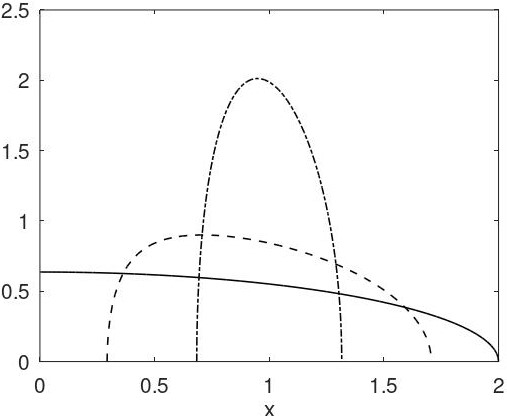}
\caption{Plot of the density of \(\mu_{c, 2}\) for \(c = 0.1\) (dash\-/dotted), \(c = 0.5\) (dashed), and \(c = 1\) (solid).}
\label{fig:dichtenZwei}
\end{figure}

\begin{Bem}\label{bem:ldp_kaufmannthaele}
Recently Kaufmann and Thäle \cite{KaufThaele2022} have proved large deviations principles for the empirical spectral measures of random square matrices whose distribution is a weighted mixture of \(\KegM{\Schatten{p}{}, \beta}{n \times n}\) and \(\Gleichv(\Kug{\Schatten{p}{}, \beta}{n \times n})\) in the spirit of Barthe, Guédon, Mendelson, and Naor~\cite{BGMN2005}. Since those measures are \(\Kug{\Schatten{p}{}, \beta}{n \times n}\)\-/symmetric they allow a representation according to Theorem~\ref{sa:gleichvert_schattenq_kugel_sphaere} with an appropriate radial component \(R_n\). According to the outline given in Section~\ref{sec:ldp_beweis_endl} below, \(R_n\) only enters in the last step, and as is seen in the proof of Proposition~\ref{sa:ldp_endl_vollkugel} an LDP for \((R_n)_{n \in \NZ}\) needs to be established, which is done in \cite[Lemma~6.6]{KaufThaele2022}. We do not go any further into this here.
\end{Bem}

\begin{Thm}[\(p = \infty\)]\label{sa:ldp_singulaerwerte_unendl}
Let \(m \in \NZ\) depend on \(n \in \NZ\) such that \(m \leq n\) and \(c := \lim_{n \to \infty} \frac{m}{n} \in (0, 1]\) exists, and let \((X^{(n)})_{n \in \NZ}\) be a sequence of random variables such that for each \(n \in \NZ\), \(X^{(n)} \sim \Gleichv(\Kug{\Schatten{\infty}{}, \beta}{m \times n})\). Then the sequence \((\mu_n)_{n \in \NZ}\) of empirical measures \(\mu_n := \frac{1}{m} \sum_{i = 1}^m \delta_{s_i(X^{(n)})}\) satisfies a large deviations principle on \(\WMasz(\RZ_{\geq 0})\) with speed \(\beta m n\) and good rate function \(\mathcal{I}_{c, \infty} \colon \WMasz(\RZ_{\geq 0}) \to [0, \infty]\) given by
\begin{equation*}
\mathcal{I}_{c, \infty}(\mu) := \begin{cases} -\frac{c}{2} \int_{[0, 1]^2} \log\lvert x^2 - y^2 \rvert \, \diff\mu^{\otimes 2}(x, y) - (1 - c) \int_{[0, 1]} \log(x) \, \diff\mu(x) + B_{c, \infty}, & \text{if } \supp(\mu) \subset [0, 1], \\ \infty & \text{else,} \end{cases}
\end{equation*}
which possesses a unique global mimimizer \(\mu_{c, \infty} \in \WMasz(\RZ_{\geq 0})\) given by its Lebesgue density,
\begin{equation*}
\frac{\diff \mu_{c, \infty}(x)}{\diff x} = \frac{1 + c}{c \pi x} \, \frac{\bigl( x^2 - (\frac{1 - c}{1 + c})^2 \bigr)^{1/2}}{(1 - x^2)^{1/2}} \Ind_{\bigl[ \frac{1 - c}{1 + c}, 1 \bigr]}(x).
\end{equation*}
Here
\begin{equation*}
B_{c, \infty} := \lim_{n \to \infty} \frac{1}{\beta m n} \log(Z_{m, n, \infty, \beta}) = \frac{c}{2} \log(c) - \frac{(1 - c)^2}{4 c} \log(1 - c) - \frac{(1 + c)^2}{4 c} \log(1 + c)
\end{equation*}
with \(Z_{m, n, \infty, \beta}\) as in Theorem~\ref{sa:gleichvert_schattenq_kugel_sphaere}.
\end{Thm}

\begin{Bem}
In the extreme case \(c = 1\), \(\mu_{1, \infty}\) has the density
\begin{equation*}
x \mapsto \frac{2}{\pi} \, (1 - x^2)^{-1/2} \Ind_{[0, 1]}(x)
\end{equation*}
which corresponds to the absolute arcsine distribution, and this is consistent with~\cite{KPTh2020_2}. For selected values of \(c\) the density of \(\mu_{c, \infty}\) is plotted in Figure~\ref{fig:dichtenUnend}.
\end{Bem}

\begin{figure}
\centering
\includegraphics[scale=0.5]{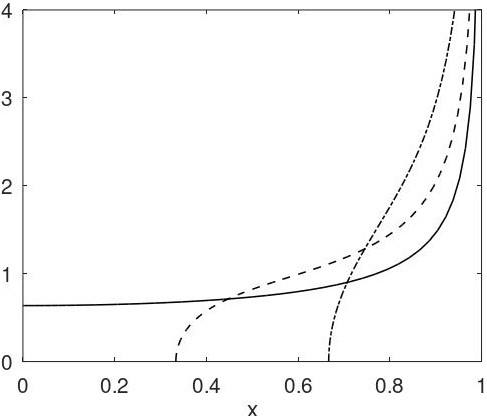}
\caption{Plot of the density of \(\mu_{c, \infty}\) for \(c = 0.2\) (dash\-/dotted), \(c = 0.5\) (dashed), and \(c = 1\) (solid).}
\label{fig:dichtenUnend}
\end{figure}

The following statement is a strong law of large numbers for the empirical measures of singular values and a corollary of the large deviations principles of Theorems~\ref{sa:ldp_singulaerwerte} and~\ref{sa:ldp_singulaerwerte_unendl} and parallels the results \cite[Corollaries~1.4, 1.6]{KPTh2020_2}.

\begin{Sa}\label{sa:ggz_singulaerwerte}
Under the premises of Theorems~\ref{sa:ldp_singulaerwerte} or~\ref{sa:ldp_singulaerwerte_unendl} we have, w.r.t.\ the weak topology,
\begin{equation*}
(\mu_n)_{n \in \NZ} \Kfs{} \mu_{c, p}.
\end{equation*}
\end{Sa}

\begin{proof}
Let \(\dw\) be some metric that induces the weak topology on \(\WMasz(\RZ_{\geq 0})\). Via the Borel\--Cantelli lemma it suffices to show, for any \(\epsilon \in \RZ_{> 0}\),
\begin{equation*}
\sum_{n = 1}^\infty \Wsk[\dw(\mu_n, \mu_{c, p}) \geq \epsilon] < \infty.
\end{equation*}
Let \(\epsilon \in \RZ_{> 0}\), then Theorem~\ref{sa:ldp_singulaerwerte} or~\ref{sa:ldp_singulaerwerte_unendl}, resp., yields
\begin{equation*}
\limsup_{n \to \infty} \frac{1}{\beta m n} \log \Wsk[\dw(\mu_n, \mu_{c, p}) \geq \epsilon] \leq -\inf\{\mathcal{I}_{c, p}(\nu) \Colon \dw(\nu, \mu_{c, p}) \geq \epsilon\} =: -M.
\end{equation*}
We will argue \(M > 0\) below; once we have established that, there follows
\begin{equation*}
\Wsk[\dw(\mu_n, \mu_{c, p}) \geq \epsilon] \leq \ez^{-M \beta m n/2}
\end{equation*}
for all but finitely many \(n\), and clearly the majorant is summable.

In order to establish \(M > 0\) we may assume \(M < \infty\), else there is nothing to prove. By definition of the infimum there exists some \(\nu \in \WMasz(\RZ_{\geq 0})\) with \(\dw(\nu, \mu_{c, p}) \geq \epsilon\) such that \(\mathcal{I}_{c, p}(\nu) \leq M + 1\). Because \(\mathcal{I}_{c, p}\) is a good rate function the sublevel set \(\inv{\mathcal{I}_{c, p}}\bigl( [0, M + 1] \bigr)\) is compact and therefore \(C := \{\nu \in \WMasz(\RZ_{\geq 0}) \Colon \dw(\nu, \mu_{c, p}) \geq \epsilon \wedge \mathcal{I}_{c, p}(\nu) \leq M + 1\}\) is nonempty and compact, thus \(\mathcal{I}_{c, p}\) attains its miminmum on \(C\) in some \(\nu_\epsilon \in C\), and by construction we get \(\mathcal{I}_{c, p}(\nu_\epsilon) = M\). But we also know \(\dw(\nu_\epsilon, \mu_{c, p}) \geq \epsilon\), in particular \(\nu_\epsilon \neq \mu_{c, p}\), and as \(\mu_{c, p}\) is the unique point where \(\mathcal{I}_{c, p}\) attains its global minimum \(0\), this implies \(M = \mathcal{I}_{c, p}(\nu_\epsilon) > 0\).
\end{proof}

Although we are not able to give a closed expression for \(\KugVol{\Schatten{p}{}, \beta}{m \times n}\) for \(p \notin \{2, \infty\}\), we are still able to obtain some asymptotics of the volume radius, as stated below.

\begin{Sa}\label{sa:volumen_bsp}
Let \(p \in (0, \infty)\), consider \(m\) as depending on \(n\) such that \(c := \lim_{n \to \infty} \frac{m}{n} \in (0, 1]\) exists, and let \(B_{c, p}\) be as in Theorem~\ref{sa:ldp_singulaerwerte}. Then we have
\begin{equation*}
\lim_{n \to \infty} (\beta n)^{1/2 + 1/p} (\KugVol{\Schatten{p}{}, \beta}{m \times n})^{1/(\beta m n)} = (2 \pi \ez^{3/2})^{1/2} \Bigl( \frac{\ez p}{c} \Bigr)^{1/p} \ez^{B_{c, p}} \, (1 - c)^{(1 - c)^2/(4 c)} \, c^{-c/4},
\end{equation*}
where in the case \(c = 1\) we interpret \((1 - c)^{(1 - c)^2} := 1\).
\end{Sa}

\begin{proof}
Refer to Proposition~\ref{sa:integraltrafo}, part~2, and denote \(c_n := \frac{2^{\beta m/2} \, \OrthVol_{m; \beta} \OrthVol_{n; \beta}}{2^{\beta m n/2} \, m! \OrthVol_{1; \beta}^m \OrthVol_{n - m; \beta}}\). Then because of \(\Ind_{\Kug{\Schatten{p}{}, \beta}{m \times n}} = \Ind_{\Kug{p, 1}{m}}{} \circ s\) we have
\begin{align*}
\KugVol{\Schatten{p}{}, \beta}{m \times n} &= c_n \int_{\Kug{p, 1}{m} \cap \RZ_{> 0}^m} \prod_{i = 1}^m x_i^{\beta (n - m + 1) - 1} \prod_{1 \leq i < j \leq m} \lvert x_i^2 - x_j^2 \rvert^\beta \, \diff x\\
&= 2^{-m} \, c_n \int_{\Kug{p/2, \beta}{m} \cap \RZ_{> 0}^m} \prod_{i = 1}^m x_i^{\beta (n - m + 1)/2 - 1} \prod_{1 \leq i < j \leq m} \lvert x_i - x_j \rvert^\beta \, \diff x,
\end{align*}
where we have transformed \(y_i = x_i^2\) and immediately renamed \(y_i\) as \(x_i\). As the integrand is positive\-/homogeneous of degree \(\frac{\beta m n}{2} - m\), via polar integration the last expression equals
\begin{align*}
\KugVol{\Schatten{p}{}, \beta}{m \times n} &= 2^{-m} \, c_n \, \frac{n^{\beta m n/p}}{\Gamma\bigl( \frac{\beta m n}{p} + 1 \bigr)} \int_{\RZ_{> 0}^m} \ez^{-\beta n \lVert x \rVert_{p/2}^{p/2}} \prod_{i = 1}^m x_i^{\beta (n - m + 1)/2 - 1} \prod_{1 \leq i < j \leq m} \lvert x_i - x_j \rvert^\beta \, \diff x\\
&= 2^{-m} \, c_n \, \frac{n^{\beta m n/p}}{\Gamma\bigl( \frac{\beta m n}{p} + 1 \bigr)} \, Z_{m, n, p, \beta}.
\end{align*}
Stirling's formula implies \(\Gamma(\frac{\beta m n}{p} + 1)^{1/(\beta m n)} = (\frac{\beta m n}{\ez p})^{1/p} (1 + \smallO(1))\), and also \((2^{-m})^{1/(\beta m n)} = 2^{-1/(\beta n)} = 1 + \smallO(1)\). Furthermore Proposition~\ref{sa:ldp_y} tells us
\begin{equation*}
\lim_{n \to \infty} Z_{m, n, p, \beta}^{1/(\beta m n)} = \lim_{n \to \infty} \exp\Bigl( \frac{1}{\beta m n} \log Z_{m, n, p, \beta} \Bigr) = \ez^{B_{c, p}}.
\end{equation*}
It remains to determine the asymptotics of \(c_n^{1/(\beta m n)}\). For that substitute the explicit expression for \(\OrthVol_{n; \beta}\) given in~\eqref{eq:vol_stiefel} and use the asymptotics stated in Lemma~\ref{lem:produkt_gamma} to arrive at (skipping the details here)
\begin{equation*}
c_n^{1/(\beta m n)} = (2 \pi \ez^{3/2})^{1/2} \, (1 - c)^{(1 - c)^2/(4 c)} \, c^{-c/4} \, (\beta n)^{-1/2} \, (1 + \smallO(1)).
\end{equation*}
Putting things together yields the result.
\end{proof}

\begin{Bem}
In the new preprint~\cite{Naor2024} Naor shows the asymptotics
\begin{equation*}
(\KugVol{\Schatten{p}{}, 1}{m \times n})^{1/(m n)} \asymp \min\{m, n\}^{-1/p} \max\{m, n\}^{-1/2},
\end{equation*}
where \(\asymp\) denotes asymptotic equivalence, i.e., boundedness from above and from below up to dimension\-/independent constant factors. If those bounds were known, they could be used together with Proposition~\ref{sa:volumen_bsp} to get bounds for \(B_{c, p}\), and vice versa.
\end{Bem}

\subsection{Proof of Theorem~\ref{sa:ldp_singulaerwerte}}
\label{sec:ldp_beweis_endl}

The proof of Theorem~\ref{sa:ldp_singulaerwerte} parallels that of \cite[Theorem~1.1]{KPTh2020_2} and accordingly proceeds in several steps. As already mentioned in the overview (Section~\ref{sec:ueberblick} \textit{ad fin.}), the adaption amounts to more than merely substituting \(m\) for \(n\) in the right places, because the now dimension\-/dependent term \(\prod_{i = 1}^m y_i^{\beta (n - m + 1)/2 - 1}\) in the density~\eqref{eq:dichte_y} prompts special treatment involving delicate estimations; all these turn up in Step~2 below. Apart from that, while going through the original paper~\cite{KPTh2020_2} carefully, the authors discovered a few technical gaps and inaccuracies there; we will refer to them in the appropriate places in our proof and provide amendments (which of course have to be readapted in the Hermitian case).

For the convenience of the reader we repeat the outline of the proof in \cite{KPTh2020_2} here, together with an addition for the present Theorem~\ref{sa:ldp_singulaerwerte} (Step~4):
\begin{compactenum}
\item Establish an LDP for \((\nu_n)_{n \geq 1}\), where \(\nu_n := \frac{1}{m} \sum_{i = 1}^m \delta_{Y_i^{(n)}}\) and \(Y^{(n)} = (Y_i^{(n)})_{i \leq m}\) has density~\eqref{eq:dichte_y}.
\item Establish an LDP for \(\bigl( (\nu_n, m_{p/2}(\nu_n)) \bigr)_{n \in \NZ}\), where \(\nu_n\) is as above.
\item Use the contraction principle to prove an LDP for \((\mu_n^2)_{n \in \NZ}\) with \(\mu_n^2 = \frac{1}{m} \sum_{i = 1}^m \delta_{m^{2/p} Y_i^{(n)} / \lVert Y^{(n)} \rVert_{p/2}} \GlVert \frac{1}{m} \sum_{i = 1}^m \delta_{m^{2/p} s_i(X^{(n)})^2}\) for the case of \(X^{(n)} \sim \KegM{\Schatten{p}{}, \beta}{m \times n}\).
\item Contract again to obtain an LDP for \((\mu_n^\kappa)_{n \in \NZ}\) with \(\mu_n^\kappa := \frac{1}{m} \sum_{i = 1}^m \delta_{m^{1/p} s_i(X^{(n)})}\), still for \(X^{(n)} \sim \KegM{\Schatten{p}{}, \beta}{m \times n}\).
\item Contract one last time to get an LDP for \((\mu_n)_{n \geq 1}\) with \(\mu_n := \frac{1}{m} \sum_{i = 1}^m \delta_{m^{1/p} s_i(X^{(n)})}\) where \(X^{(n)} \sim \Gleichv(\Kug{\Schatten{p}{}}{m \times n})\).
\end{compactenum}
The main justification for this procedure are Proposition~\ref{sa:unitinvar} and Theorem~\ref{sa:gleichvert_schattenq_kugel_sphaere} which together yield
\begin{equation*}
s(X^{(n)})_\pi \GlVert R \, \frac{((Y_i^{(n)})^{1/2})}{\lVert Y^{(n)} \rVert_{p/2}^{1/2}},
\end{equation*}
where \(\pi \sim \Gleichv(\Perm{m})\) is independent of \(X^{(n)}\), and therewith
\begin{equation*}
\frac{1}{m} \sum_{i = 1}^m \delta_{m^{1/p} s_i(X^{(n)})} \GlVert \frac{1}{m} \sum_{i = 1}^m \delta_{m^{1/p} R ((Y_i^{(n)})^{1/2})_{i \leq m} / \lVert Y^{(n)} \rVert_{p/2}^{1/2}};
\end{equation*}
also note
\begin{equation*}
m_{p/2}(\nu_n) = \int_{\RZ_{\geq 0}} y^{p/2} \, \diff\nu_n(y) = \frac{1}{m} \sum_{i = 1}^m (Y_i^{(n)})^{p/2} = \frac{\lVert Y^{(n)} \rVert_{p/2}^{p/2}}{m}.
\end{equation*}
That \(\pi\) indeed is immaterial is argued by the result below.

\begin{Lem}\label{lem:permutation}
Let \(E\) be a topological space and \(d \in \NZ\), let \(S = (S_i)_{i \leq d}\) be an \(E^d\)\-/valued random variable and let \(\pi \sim \Gleichv(\Perm{d})\) be independent of \(S\). Then
\begin{equation*}
\frac{1}{d} \sum_{i = 1}^d \delta_{S_i} \GlVert \frac{1}{d} \sum_{i = 1}^d \delta_{S_{\pi(i)}}.
\end{equation*}
\end{Lem}

\begin{proof}
Let \(f \colon \WMasz(E) \to \RZ\) be measurable and nonnegative, then
\begin{align*}
\Erw\biggl[ f\biggl( \frac{1}{d} \sum_{i = 1}^d \delta_{S_{\pi(i)}} \biggr) \biggr] &= \int_{E^d} \frac{1}{d!} \sum_{\tau \in \Perm{d}} f\biggl( \frac{1}{d} \sum_{i = 1}^d \delta_{s_{\tau(i)}} \biggr) \, \diff(\Wsk \circ \inv{S})(s)\\
&= \int_{E^d} \frac{1}{d!} \sum_{\tau \in \Perm{d}} f\biggl( \frac{1}{d} \sum_{i = 1}^d \delta_{s_i} \biggr) \, \diff(\Wsk \circ \inv{S})(s)\\
&= \int_{E^d} f\biggl( \frac{1}{d} \sum_{i = 1}^d \delta_{s_i} \biggr) \, \diff(\Wsk \circ \inv{S})(s) = \Erw\biggl[ f\biggl( \frac{1}{d} \sum_{i = 1}^d \delta_{S_i} \biggr) \biggr],
\end{align*}
where we have invested the elementary fact that addition is commutative and thus \(\sum_{i = 1}^d \delta_{s_{\tau(i)}} = \sum_{i = 1}^d \delta_{s_i}\) for any \(s \in E^d\) and \(\tau \in \Perm{d}\).
\end{proof}

In order to keep the wording simple, in the sequel we are going to work under the premises of Theorem~\ref{sa:ldp_singulaerwerte}, that is, \(p \in (0, \infty)\), \(m\) varies with \(n\) such that \(m \leq n\) and \(c := \lim_{n \to \infty} \frac{m}{n} \in (0, 1]\) exists, for each \(n \in \NZ\) \(Y^{(n)} = (Y_i^{(n)})_{i \leq m}\) is a random vector with density~\eqref{eq:dichte_y}, and \(\nu_n := \frac{1}{m} \sum_{i = 1}^m \delta_{Y_i^{(n)}}\).

\paragraph{Step~1}
This is dealt with readily using a result from \cite{HiaiPetz2000}.

\begin{Sa}\label{sa:ldp_y}
The limit \(B_{c, p} = \lim_{n \to \infty} \frac{1}{\beta m n} \log(Z_{m, n, p, \beta})\) exists and is finite and does not depend on \(\beta\), and the sequence \((\nu_n)_{n \in \NZ}\) of empirical measures satisfies a large deviations principle on \(\WMasz(\RZ_{\geq 0})\) at speed \(\beta m n\) with good rate function \(\mathcal{J}_{c, p} \colon \WMasz(\RZ_{\geq 0}) \to [0, \infty]\) given by
\begin{equation*}
\mathcal{J}_{c, p}(\mu) := -\frac{c}{2} \int_{\RZ_{\geq 0}^2} \log\lvert x - y \rvert \, \diff\mu^{\otimes 2}(x, y) + \int_{\RZ_{\geq 0}} \Bigl( x^{p/2} - \frac{1 - c}{2} \log(x) \Bigr) \, \diff\mu(x) + B_{c, p},
\end{equation*}
which has a unique global minimizer \(\nu_{c, p}\) with compact support and a Lebesgue density.
\end{Sa}

\begin{proof}
This follows almost immediately from \cite[Theorem~5.5.1]{HiaiPetz2000}. Said Theorem implies the existence of \(B_\beta := \lim_{n \to \infty} \frac{1}{n^2} \log(Z_{m, n, p, \beta})\), which might depend on \(\beta\) somehow, and \((\nu_n)_{n \in \NZ}\) satisfies an LDP at speed \(n^2\) with good rate function \(J_{c, p, \beta} = \beta c J_{c, p} + B_\beta\) where
\begin{equation*}
J_{c, p}(\mu) = -\frac{c}{2} \int_{\RZ_{> 0}^2} \log\lvert x - y \rvert \, \diff\mu^{\otimes 2}(x, y) + \int_{\RZ_{> 0}} \Bigl( x^{p/2} - \frac{1 - c}{2} \log(x) \Bigr) \, \diff\mu(x),
\end{equation*}
which no longer depends on \(\beta\); hence a measure \(\mu\) minimizes \(J_{c, p, \beta}\) iff it minimizes \(J_{c, p}\), and therefore the minimizer \(\nu_{c, p}\) does not depend on \(\beta\). Being the rate function of an LDP, \(J_{c, p, \beta}\) satisfies \(J_{c, p, \beta}(\nu_{c, p}) = 0\), hence \(B_{\beta} = -\beta c J_{c, p}(\nu_{c, p})\), and \(B_{c, p} := -J_{c, p}(\nu_{c, p})\) does not depend on \(\beta\) and it satisfies
\begin{equation*}
B_{c, p} = \frac{1}{\beta c} \, B_\beta = \lim_{n \to \infty} \frac{n}{m} \cdot \lim_{n \to \infty} \frac{1}{\beta n^2} \log(Z_{m, n, p, \beta}) = \lim_{n \to \infty} \frac{1}{\beta m n} \log(Z_{m, n, p, \beta}).
\end{equation*}
Writing out the defining inequalities of the LDP, we have, for any measurable set \(A\),
\begin{equation*}
-\inf(J_{c, p, \beta}(\offen{A})) \leq \liminf_{n \to \infty} \frac{1}{n^2} \log \Wsk[\nu_n \in A] \leq \limsup_{n \to \infty} \frac{1}{n^2} \log \Wsk[\nu_n \in A] \leq -\inf(J_{c, p, \beta}(\schluss{A}));
\end{equation*}
and since \(J_{c, p, \beta} = \beta c J_{c, p} + \beta c B_{c, p} = \beta c \mathcal{J}_{c, p}\) and \(\frac{1}{\beta c} = \lim_{n \to \infty} \frac{n}{\beta m}\), we can divide the LDP by \(\beta c\) to get
\begin{equation*}
-\inf(\mathcal{J}_{c, p}(\offen{A})) \leq \liminf_{n \to \infty} \frac{1}{\beta m n} \log \Wsk[\nu_n \in A] \leq \limsup_{n \to \infty} \frac{1}{\beta m n} \log \Wsk[\nu_n \in A] \leq -\inf(\mathcal{J}_{c, p}(\schluss{A})),
\end{equation*}
so \((\nu_n)_{n \in \NZ}\) satisfies the claimed LDP. The compactness of the support of the minimizer \(\nu_{c, p}\) is argued by \cite[Theorem~5.3.3]{HiaiPetz2000}, and the existence of a Lebesgue density by \cite[Theorem~IV.2.5]{SaffTotik1997} as the following argument makes clear: up to a constant factor the rate function \(\mathcal{J}_{c, p}\) can be considered an energy functional with external potential \(Q(x) := \frac{1}{c} x^{p/2} - \frac{1 - c}{2 c} \log(x)\) on \(\RZ_{\geq 0}\) (\(\infty\) else), or weight function \(w(x) = \ez^{-Q(x)} = x^{(1 - c)/(2 c)} \, \ez^{-x^{p/2}/c}\) (\(0\) on \(\RZ_{\leq 0}\)); then \(w\) is positive on the interval \(I = \RZ_{> 0}\) and infinitely often differentiable there, and so all premises of \cite[Theorem~IV.2.5]{SaffTotik1997} are met.
\end{proof}

\begin{Bem}\label{bem:ldp_y}
There seem to be no special results available for weight functions of the form \(w(x) = x^\theta \ez^{-\lambda x^{p/2}} \Ind_{\RZ_{\geq 0}}(x)\) unless \(p = 2\) (for which see Remark~\ref{bem:minimierer_p2} below). The general theory set out in \cite{SaffTotik1997} tells us that \(\supp(\nu_{c, p}) = [a, b]\) with \(a, b \in \RZ_{\geq 0}\), that \(a = 0\) only for \(c = 1\), and else \(a\) and \(b\) satisfy the simultaneous equations
\begin{align*}
\int_a^b \Bigl( \frac{p x^{p/2 - 1}}{2 c} - \frac{1 - c}{2 c x} \Bigr) (x - a)^{-1/2} (b - x)^{1/2} \, \diff x &= -\pi,\\
\int_a^b \Bigl( \frac{p x^{p/2 - 1}}{2 c} - \frac{1 - c}{2 c x} \Bigr) (x - a)^{1/2} (b - x)^{-1/2} \, \diff x &= \pi.
\end{align*}
The integrals can be expressed in terms of the Gaussian hypergeometric functions, but that does not help any further. We also cannot determine the density of \(\nu_{c, p}\) as the only pertinent result \cite[Theorem~IV.3.2]{SaffTotik1997} cannot be applied: upon transferring \([a, b]\) to \([0, 1]\) via an affine map the conditions on the external field are violated.
\end{Bem}

\paragraph{Step~2}
We repeat the observation made in \cite{KPTh2020_2} that the moment map \(m_{p/2}\) is not continuous w.r.t.\ the weak topology, thus contraction is not applicable, and therefore an LDP for \(\bigl( (\nu_n, m_{p/2}(\nu_n)) \bigr)_{n \in \NZ}\) must be proved directly.

\begin{Sa}\label{sa:ldp_empir_paar}
The sequence of pairs \(\bigl( (\nu_n, m_{p/2}(\nu_n)) \bigr)_{n \in \NZ}\) satisfies an LDP on \(\WMasz(\RZ_{\geq 0}) \times \RZ_{\geq 0}\) with speed \(\beta m n\) and good rate function \(\mathcal{J}_{c, p}^2 \colon \WMasz(\RZ_{\geq 0}) \times \RZ_{\geq 0} \to [0, \infty]\) given by
\begin{equation*}
\mathcal{J}_{c, p}^2(\mu, M) = \begin{cases} -\frac{c}{2} \int_{\RZ_{\geq 0}^2} \log\lvert x - y \rvert \, \diff\mu^{\otimes 2}(x, y) - \frac{1 - c}{2} \int_{\RZ_{\geq 0}} \log(x) \, \diff\mu(x) + M + B_{c, p} & \text{if } m_{p/2}(\mu) \leq M, \\ \infty & \text{if } m_{p/2}(\mu) > M, \end{cases}
\end{equation*}
which attains its global minimum precisely at \((\nu_{c, p}, m_{p/2}(\nu_{c, p}))\), where \(\nu_{c, p}\) is the same as in Proposition~\ref{sa:ldp_y}.
\end{Sa}

Seeing that the proof of the corresponding \cite[Theorem~4.1]{KPTh2020_2} is rather lengthy, here we do not spell out all details anew but restrict ourselves mostly to the adaptions necessitated by the new term \(\prod_{i = 1}^m y_i^{\beta (n - m + 1)/2 - 1}\) in density~\eqref{eq:dichte_y} for \(m \neq n\).

\cite[Lemma~4.2]{KPTh2020_2} is carried over with the obvious modifications (i.e., replace \(\RZ\) by \(\RZ_{\geq 0}\) and \(p\) by \(\frac{p}{2}\)) since it does not use the density of \(Y^{(n)}\).

\cite[Lemma~4.4]{KPTh2020_2} too can be adopted with the obvious modifications for the same reason as before; in particular the definition of \(D_n\) now reads
\begin{equation*}
D_n := \bigl( m^{2/p} (M - m_{p/2}(\mu))^{2/p} - m^{2/p - 2}, m^{2/p} (M - m_{p/2}(\mu))^{2/p} \bigr).
\end{equation*}
But \cite[Lemma~4.5]{KPTh2020_2} needs to be modified to the form given below.

\begin{Lem}\label{lem:ldp_untere_abschaetzung}
Let \(\mu \in \WMasz(\RZ_{\geq 0})\) be supported on an interval \([a, b]\) with \(0 \leq a < b < \infty\) and let \(M \in (m_{p/2}(\mu), \infty)\); assume that \(\mu\) has Lebesgue density \(h\) which is continuous on \([a, b]\) and satisfies \(\inf_{x \in [a, b]} h(x) > 0\). Let \(d \in \NZ\), \(f_1, \dotsc, f_d \in \CeBe(\RZ_{\geq 0})\) and \(\delta, \epsilon \in \RZ_{> 0}\). Then
\begin{multline*}
\liminf_{n \to \infty} \frac{1}{\beta m n} \log \Wsk\bigl[ \nu_n' \in \mathcal{O}_{\epsilon, d}(\mu) \wedge \bigl\lvert m_{p/2}(\nu_n') - m_{p/2}(\mu) \bigr\rvert < \delta \wedge Y_m^{(n)} \in D_n \bigr]\\
\geq \frac{c}{2} \int_{\RZ_{\geq 0}^2} \log\lvert x - y \rvert \, \diff\mu^{\otimes 2}(x, y) + \frac{1 - c}{2} \int_{\RZ_{\geq 0}} \log(x) \, \diff\mu(x) - M - B_{c, p},
\end{multline*}
where \(\nu_n' := \frac{1}{m - 1} \sum_{i = 1}^{m - 1} \delta_{Y_i^{(n)}}\).
\end{Lem}

\begin{proof}
Differently from \cite{KPTh2020_2}, the objects \(a_i^{(m - 1)}\), \(b_i^{(m - 1)}\) and \(\Delta_{m - 1}\) now depend rather on \(m\) than on \(n\). Since \([a, b] \subset \RZ_{\geq 0}\), we additionally have \(\xi_i^{(m - 1)} = b_i^{(m - 1)}\). This yields
\begin{align*}
\Wsk\bigl[ &\nu_n' \in \mathcal{O}_{\epsilon, d}(\mu) \wedge \bigl\lvert m_{p/2}(\nu_n') - m_{p/2}(\mu) \bigr\rvert < \delta \wedge Y_m^{(n)} \in D_n \bigr]\\
&\geq \frac{1}{Z_{m, n, p, \beta}} \int_{\Delta_{m - 1} \times D_n} \ez^{-\beta n \lVert y \rVert_{p/2}^{p/2}} \prod_{i = 1}^m y_i^{\beta (n - m + 1)/2 - 1} \prod_{1 \leq i < j \leq m} \lvert y_j - y_i \rvert^\beta \, \diff y\\
&= \frac{1}{Z_{m, n, p, \beta}} \int_{\Delta_{m - 1} \times D_n} \ez^{-\beta n \sum_{i = 1}^{m - 1} y_i^{p/2}} \prod_{i = 1}^{m - 1} y_i^{\beta (n - m + 1)/2 - 1} \prod_{1 \leq i < j \leq m - 1} \lvert y_j - y_i \rvert^\beta\\
&\mspace{100mu} \cdot \ez^{-\beta n y_m^{p/2}} y_m^{\beta (n - m + 1)/2 - 1} \prod_{i = 1}^{m - 1} \lvert y_m - y_i \rvert^\beta \, \diff y\\
&\geq \frac{1}{Z_{m, n, p, \beta}} \int_{\Delta_{m - 1} \times D_n} \ez^{-\beta n \sum_{i = 1}^{m - 1} (b_i^{(m - 1)})^{p/2}} \prod_{i = 1}^{m - 1} (\eta_i^{(m - 1)})^{\beta (n - m + 1)/2 - 1}\\
&\mspace{100mu} \cdot \prod_{1 \leq i < j \leq m - 1} \lvert a_j^{(m - 1)} - b_i^{(m - 1)} \rvert^\beta \, \ez^{-\beta n y_m^{p/2}} y_m^{\beta (n - m + 1)/2 - 1} \prod_{i = 1}^{m - 1} \lvert y_m - y_i \rvert^\beta \, \diff y\\
&\geq \frac{1}{Z_{m, n, p, \beta}} \, \ez^{-\beta n \sum_{i = 1}^{m - 1} (b_i^{(m - 1)})^{p/2}} \prod_{i = 1}^{m - 1} (\eta_i^{(m - 1)})^{\beta (n - m + 1)/2 - 1} \prod_{1 \leq i < j \leq m - 1} \lvert a_j^{(m - 1)} - b_i^{(m - 1)} \rvert^\beta\\
&\mspace{100mu} \cdot \ez^{-\beta n m (M - m_{p/2}(\mu))} \eta_m^{\beta (n - m + 1)/2 - 1} \vol_m(\Delta_{m - 1} \times D_n),
\end{align*}
where
\begin{equation*}
\eta_i^{(m - 1)} := \begin{cases} a_i^{(m - 1)} & \text{if } \frac{\beta (n - m + 1)}{2} \geq 1, \\ b_i^{(m - 1)} & \text{if } \frac{\beta (n - m + 1)}{2} < 1 \end{cases}
\end{equation*}
and
\begin{equation*}
\eta_m := \begin{cases} \inf(D_n) = m^{2/p} \bigl( M - m_{p/2}(\mu) \bigr)^{2/p} - m^{2/p - 2} & \text{if } \frac{\beta (n - m + 1)}{2} \geq 1, \\ \sup(D_n) = m^{2/p} \bigl( M - m_{p/2}(\mu) \bigr)^{2/p} & \text{if } \frac{\beta (n - m + 1)}{2} < 1. \end{cases}
\end{equation*}
(In the case \(c < 1\) only the first case occurs for all \(n\) large enough.) Now we use \(\lim_{n \to \infty} \frac{1}{\beta m n} \log(Z_{m, n, p, \beta}) = B_{c, p}\) by Proposition~\ref{sa:ldp_y}, also
\begin{equation*}
\lim_{n \to \infty} \frac{1}{m} \sum_{i = 1}^{m - 1} (b_i^{(m - 1)})^{p/2} = m_{p/2}(\mu),
\end{equation*}
and
\begin{equation*}
\lim_{n \to \infty} \frac{\beta (n - m + 1)/2 - 1}{\beta m n} \sum_{i = 1}^{m - 1} \log(\eta_i^{(m - 1)}) = \frac{1 - c}{2} \int_{\RZ_{\geq 0}} \log(x) \, \diff\mu(x),
\end{equation*}
as well as
\begin{equation*}
\liminf_{n \to \infty} \frac{1}{m n} \sum_{1 \leq i < j \leq m - 1} \log\lvert a_j^{(m - 1)} - b_i^{(m - 1)} \rvert \geq \frac{c}{2} \int_{\RZ_{\geq 0}^2} \log\lvert x - y \rvert \, \diff\mu^{\otimes 2}(x, y);
\end{equation*}
there is also
\begin{align*}
\lim_{n \to \infty} \frac{\beta (n - m + 1)/2 - 1}{\beta m n} \log(\eta_m) &= \frac{1 - c}{2} \lim_{n \to \infty} \frac{1}{m} \log\bigl( m^{2/p} \bigl( M - m_{p/2}(\mu) \bigr)^{2/p} - r(n) m^{2/p - 2} \bigr)\\
&= \frac{1 - c}{2} \lim_{n \to \infty} \Bigl( \frac{2 \log(m)}{p m} + \frac{1}{m} \log\Bigl( \bigl( M - m_{p/2}(\mu) \bigr)^{2/p} - \frac{r(n)}{m^2} \Bigr) \Bigr)\\
&= 0,
\end{align*}
where \(r(n) \in \{0, 1\}\) according to the definition of \(\eta_m\). Finally we know \(\liminf_{n \to \infty} \frac{1}{\beta m n} \log(\vol_m(\Delta_m \times D_n)) \geq 0\) and hence obtain
\begin{align*}
\liminf_{n \to \infty} \frac{1}{\beta m n} \log \Wsk\bigl[ &\nu_n' \in \mathcal{O}_{\epsilon, d}(\mu) \wedge \bigl\lvert m_{p/2}(\nu_n') - m_{p/2}(\mu) \bigr\rvert < \delta \wedge Y_m^{(n)} \in D_n \bigr]\\
&\geq -B_{c, p} - m_{p/2}(\mu) + \frac{1 - c}{2} \int_{\RZ_{\geq 0}} \log(x) \, \diff\mu(x)\\
&\quad + \frac{c}{2} \int_{\RZ_{\geq 0}^2} \log\lvert x - y \rvert \, \diff\mu^{\otimes 2}(x, y) - (M - m_{p/2}(\mu)),
\end{align*}
which equals the claimed expression.
\end{proof}

The proof of \cite[Lemma~4.6]{KPTh2020_2} could in principle be adopted without much comment, but for the present article the authors have decided to give a detailed proof, for two reasons: firstly, it is not at all clear that an approximating sequence of measures \((\mu_k)_{k \in \NZ}\) as postulated exists indeed, because in contrast to \cite{KPTh2020_2} and to \cite{HiaiPetz2000} as referenced therein, our candidate for the rate function, \(\mathcal{J}_{c, p}^2\), contains the additional term \(\int_{\RZ_{\geq 0}} \log(x) \, \diff\mu(x)\); moreover we work with probability measures on \(\RZ_{\geq 0}\), not \(\RZ\), so it is not immediately clear how the approximating measures are to be constructed, and again neither \cite{KPTh2020_2} nor \cite{HiaiPetz2000} provide any details. Secondly, there is an error of sign in \cite{KPTh2020_2}: the display before the line `This standard fact can be verified (\dots)' must read `\(\limsup_{k \to \infty} \bigl( -\int_\RZ \int_\RZ \ldots \bigr) \leq -\int_\RZ \int_\RZ \ldots\),' and the reference to \cite[p.\ 214]{HiaiPetz2000} (read: `p.\ 216') does not seem directly applicable.

In the following few lemmas we attempt a careful proof of all partial results omitted or only hinted at, and mentioned above. The first one deals with the concavity of the free entropy under convolution, mentioned in \cite[p.\ 216]{HiaiPetz2000}.

\begin{Lem}\label{lem:freieentropie_faltung}
Define the functional \(\Sigma \colon \WMasz(\CZ) \to \schluss{\RZ}\) by \(\Sigma(\mu) := \int_{\CZ^2} \log\lvert x - y \rvert \, \diff\mu^{\otimes 2}(x, y)\), wherever the integral is defined. Let \(\mu, \nu \in \WMasz(\CZ)\) be compactly supported, then,
\begin{equation*}
\Sigma(\nu * \mu) \geq \Sigma(\mu).
\end{equation*}
Via trivial extension, the statement remains true for \(\mu, \nu \in \WMasz(\RZ)\), or for \(\mu \in \WMasz(\RZ_{\geq 0})\) and \(\nu \in \WMasz(\RZ)\) such that \(\nu * \mu \in \WMasz(\RZ_{\geq 0})\).
\end{Lem}

\begin{proof}
We proceed in two steps: first assume \(\nu\) to be discrete with finite support, then approximate arbitrary \(\nu\) by discrete measures.

\textit{Step~1, \(\nu\) discrete.} Assume we can write
\begin{equation*}
\nu = \sum_{i = 1}^n a_i \delta_{x_i}
\end{equation*}
with some \(n \in \NZ\), \(x_1, \dotsc, x_n \in \CZ\), and \(a_1, \dotsc, a_n \in \RZ_{\geq 0}\) with \(\sum_{i = 1}^n a_i = 1\). By the definition of convolution we have, for any \(A \in \Borel_\CZ\),
\begin{align*}
(\nu * \mu)(A) &= \int_\CZ \int_\CZ 1_A(x + y) \, \diff\mu(y) \, \diff\nu(x)\\
&= \int_\CZ \mu(A - x) \, \diff\nu(x) = \sum_{i = 1}^n a_i \mu(A - x_i),
\end{align*}
that means,
\begin{equation*}
\nu * \mu = \sum_{i = 1}^n a_i \mu_{x_i},
\end{equation*}
where we have defined ad hoc \(\mu_{x_i} \colon A \mapsto \mu(A - x_i)\). Since \(\Sigma\) is concave by \cite[Proposition~5.3.2]{HiaiPetz2000}, this readily yields
\begin{equation*}
\Sigma(\nu * \mu) \geq \sum_{i = 1}^n a_i \Sigma(\mu_{x_i}) = \sum_{i = 1}^n a_i \Sigma(\mu) = \Sigma(\mu),
\end{equation*}
where we additionally have used that \(\Sigma\) is invariant under translation, which follows from the fact that its definition depends only on the distance \(\lvert x - y\rvert\).

\textit{Step~2, \(\nu\) arbitrary.} Since \(\CZ\) is Polish, so is \(\WMasz(\CZ)\); specifically, a dense subset is given by convex combinations of Dirac measures, that is, discrete measures with finite support. So choose an approximating sequence \((\nu_n)_{n \in \NZ}\) of such discrete measures converging to \(\nu\). First we prove \((\nu_n * \mu)_{n \in \NZ} \to \nu * \mu\) which we do with Lévy's continuity theorem; so let \(\phi_\rho\) denote the characteristic function of a measure \(\rho\),\footnote{For measures \(\rho \in \WMasz(\CZ)\) this is defined via identification \(\CZ \cong \RZ^2\), i.e., \(\phi_\rho(t) := \int_\CZ \ez^{\ie \Re(\konj{t} z)} \, \diff\mu(z)\).} then we know \((\phi_{\nu_n}(t))_{n \in \NZ} \to \phi_\nu(t)\) for every \(t \in \CZ\), therefore also
\begin{equation*}
\phi_{\nu_n * \mu}(t) = \phi_{\nu_n}(t) \phi_\mu(t) \xrightarrow[n \to \infty]{} \phi_\nu(t) \phi_\mu(t) = \phi_{\nu * \mu}(t)
\end{equation*}
for every \(t \in \CZ\). Again by \cite[Proposition~5.3.2]{HiaiPetz2000}, \(\Sigma\) is upper semicontinuous, so together with Step~1 we get
\begin{equation*}
\Sigma(\nu * \mu) \geq \limsup_{n \to \infty} \Sigma(\nu_n * \mu) \geq \limsup_{n \to \infty} \Sigma(\mu) = \Sigma(\mu). \qedhere
\end{equation*}
\end{proof}

The next lemma establishes the existence of a ``good approximating sequence'' for \(\mu\) such that Lemma~\ref{lem:ldp_untere_abschaetzung} can be applied. In order to abbreviate notation somewhat we introduce the kernel function
\begin{equation*}
F \colon \RZ_{\geq 0}^2 \to (-\infty, \infty], \quad (x, y) \mapsto -c \log\lvert x - y \rvert - \frac{1 - c}{2} \log(x y);
\end{equation*}
then we can rewrite
\begin{equation*}
-c \int_{\RZ_{\geq 0}^2} \log\lvert x - y \rvert \, \diff\mu^{\otimes 2}(x, y) - (1 - c) \int_{\RZ_{\geq 0}} \log(x) \, \diff\mu(x) = \int_{\RZ_{\geq 0}^2} F(x, y) \, \diff\mu^{\otimes 2}(x, y),
\end{equation*}
whenever the integrals are defined.

\begin{Lem}\label{lem:approx_maszfolge}
Let \(\mu \in \WMasz(\RZ_{\geq 0})\) with \(\mu(\{0\}) = 0\), \(m_{p/2}(\mu) < \infty\), and \(\int_{\RZ_{\geq 0}^2} F \, \diff\mu^{\otimes 2} < \infty\). Then there exists a sequence \((\mu_k)_{k \in \NZ}\) in \(\WMasz(\RZ_{\geq 0})\) such that,
\begin{compactenum}[(i)]
\item \((\mu_k)_{k \in \NZ} \to \mu\) (weakly),
\item \((m_{p/2}(\mu_k))_{k \in \NZ} \to m_{p/2}(\mu)\),
\item \(\limsup_{k \to \infty} \int_{\RZ_{\geq 0}^2} F \, \diff\mu_k^{\otimes 2} \leq \int_{\RZ_{\geq 0}^2} F \, \diff\mu^{\otimes 2}\), and
\item for any \(k \in \NZ\), \(\mu_k\) is supported on a compact interval \([a_k, b_k]\) with \(0 < a_k < b_k < \infty\) and has a Lebesgue density \(h_k\) which is continuous on \([a_k, b_k]\) and satisfies \(\min h_k([a_k, b_k]) > 0\).
\end{compactenum}
\end{Lem}

\begin{proof}
The construction of \((\mu_k)_{k \in \NZ}\) proceeds in three steps, where in each step \(\mu\) is assumed to behave closer to the desired properties; the steps are going to be pasted together via triangle inequality. We will follow the course sketched in \cite[p.\ 216]{HiaiPetz2000}.

\textit{Step~1.} Make no further assumptions on \(\mu\), and construct \(\mu_k\) to have compact support within \(\RZ_{> 0}\). To that end define, for any \(k \in \NZ\), \(k \geq k_0\),
\begin{equation*}
\mu_k \colon \Borel(\RZ_{\geq 0}) \to \RZ, \quad A \mapsto \frac{\mu(A \cap [\frac{1}{k}, k])}{\mu([\frac{1}{k}, k])},
\end{equation*}
where \(k_0 \in \NZ\) is chosen such that \(\mu([\frac{1}{k}, k]) > 0\) for all \(k \geq k_0\); obviously this is possible since \(\mu([\frac{1}{k}, k]) \to \mu((0, \infty)) = 1\) as \(k \to \infty\) where we have needed \(\mu(\{0\}) = 0\). Now for any \(k \geq k_0\), \(\mu_k\) has support contained in \([\frac{1}{k}, k]\), and it remains to check conditions (i)--(iii).

\textit{Ad (i).} We already have pointed out \(\mu([\frac{1}{k}, k]) \to 1\); similarly, for any \(A \in \Borel(\RZ_{\geq 0})\), since \((A \cap [\frac{1}{k}, k])_{k \geq k_0}\) is increasing we get \(\mu(A \cap [\frac{1}{k}, k]) \to \mu(A \cap (0, \infty)) = \mu(A)\) and therefore,
\begin{equation*}
(\mu_k(A))_{k \geq k_0} \to \mu(A),
\end{equation*}
so we even have strong convergence of \((\mu_k)_{k \geq k_0}\) to \(\mu\).

\textit{Ad (ii).} For any \(x \in \RZ_{\geq 0}\), \(x^{p/2} \Ind_{[1/k, k]}(x)\) is nonnegative and monotonically increasing with limit \(x^{p/2}\), hence monotone convergence yields,
\begin{equation*}
m_{p/2}(\mu_k) = \frac{1}{\mu([\frac{1}{k}, k])} \int_{\RZ_{\geq 0}} x^{p/2} \Ind_{[1/k, k]}(x) \, \diff\mu(x) \xrightarrow[k \to \infty]{} \frac{1}{1} \int_{\RZ_{\geq 0}} x^{p/2} \, \diff\mu(x) = m_{p/2}(\mu).
\end{equation*}

\textit{Ad (iii).} Similarly to (ii) we have \(F \Ind_{[1/k, k]^2} \to F\) \(\mu^{\otimes 2}\)\=/almost everywhere; also, the negative part of \(F\) can be bounded from below by \(-C (x^{p/2} + y^{p/2} + x^{p/2} \, y^{p/2})\) with some \(C > 0\), which is \(\mu^{\otimes 2}\)\=/integrable since we assume \(m_{p/2}(\mu) < \infty\); and because of \(\int_{\RZ_{\geq 0}^2} F \, \diff\mu^{\otimes 2} < \infty\) also \(F^+\), the positive part of \(F\), is \(\mu^{\otimes 2}\)\=/integrable and we have \(F \Ind_{[1/k, k]^2} \leq F^+\) \(\mu^{\otimes 2}\)\=/almost everywhere. The conclusion now follows via Fatou's lemma, that is,
\begin{align*}
\limsup_{k \to \infty} \int_{\RZ_{\geq 0}^2} F \, \diff\mu_k^{\otimes 2} &= \limsup_{k \to \infty} \frac{1}{\mu([\frac{1}{k}, k])^2} \int_{\RZ_{\geq 0}^2} F \Ind_{[1/k, k]^2} \, \diff\mu^{\otimes 2}\\
&\leq \frac{1}{1} \int_{\RZ_{\geq 0}^2} F \, \diff\mu^{\otimes 2}.
\end{align*}

\textit{Step~2.} Next we assume that \(\mu\) has support contained in an interval \([a, b]\) with \(0 < a < b < \infty\), and we want \(\mu_k\) to be compactly supported within \((0, \infty)\) and to have a continuous Lebesgue density.

Fix a nonnegative function \(f \in \CeCe(\RZ)\) supported in \([-1, 1]\) and such that \(\int_\RZ f(x) \, \diff x = 1\) (any smoothing kernel will do), and for \(k \geq k_0\) define \(f_k \colon x \mapsto k f(k x)\) and \(\mu_k := f_k * \mu\),\footnote{This means \(\nu_k * \mu\) where \(\diff \nu_k(x) := f_k(x) \, \diff x\).} where this time \(k_0 \in \NZ\) is chosen such that \(\frac{1}{k_0} \leq \frac{a}{2}\). By the basic properties of convolution it follows that \(\mu_k\) is supported in \([a - \frac{1}{k}, b + \frac{1}{k}] \subset [\frac{a}{2}, b + \frac{a}{2}]\) and has the continuous Lebesgue density
\begin{equation*}
h_k \colon \RZ_{\geq 0} \to \RZ, \quad x \mapsto \int_{\RZ_{\geq 0}} f_k(x - y) \, \diff\mu(y).
\end{equation*}
Again we have to assert (i)--(iii).

\textit{Ad (i).} Like in the proof of Lemma~\ref{lem:freieentropie_faltung} we will invest Lévy's continuity theorem. Note\footnote{Like before this means \(\phi_{\nu_k}\) where \(\diff\nu_k(x) := f_k(x) \, \diff x\).}
\begin{equation*}
\phi_{f_k}(t) = \int_\RZ f_k(x) \ez^{\ie t x} \, \diff x = \int_\RZ f(x) \ez^{\ie t x/k} \, \diff x \xrightarrow[k \to \infty]{} \int_\RZ f(x) \, \diff x = 1
\end{equation*}
for all \(t \in \RZ\) by dominated convergence; this implies
\begin{equation*}
\phi_{\mu_k}(t) = \phi_{f_k * \mu}(t) = \phi_{f_k}(t) \phi_\mu(t) \xrightarrow[k \to \infty]{} \phi_\mu(t)
\end{equation*}
for all \(t \in \RZ\), thus \((\mu_k)_{k \geq k_0} \to \mu\).

\textit{Ad (ii).} This is a direct consequence of (i) since all measures involved are supported in the common compact set \([\frac{a}{2}, b + \frac{a}{2}]\), and restricted to that set the function \(x \mapsto x^{p/2}\) is bounded and continuous.

\textit{Ad (iii).} With the notation from Lemma~\ref{lem:freieentropie_faltung} and immediately using it we have
\begin{equation*}
\Sigma(\mu_k) = \Sigma(f_k * \mu) \geq \Sigma(\mu),
\end{equation*}
and therewith,
\begin{align*}
\int_{\RZ_{\geq 0}^2} F \, \diff\mu_k^{\otimes 2} &= -c \Sigma(\mu_k) - (1 - c) \int_{\RZ_{\geq 0}} \log(x) \, \diff\mu_k(x)\\
&\leq -c \Sigma(\mu) - (1 - c) \int_{\RZ_{\geq 0}} \log(x) \, \diff\mu_k(x).
\end{align*}
So it suffices to show
\begin{equation*}
\limsup_{k \to \infty} \biggl( -\int_{\RZ_{\geq 0}} \log(x) \, \diff\mu_k(x) \biggr) \leq -\int_{\RZ_{\geq 0}} \log(x) \, \diff\mu(x).
\end{equation*}
But like in (ii) this readily follows from (i) as also the logarithm restricted to \([\frac{a}{2}, b + \frac{a}{2}]\) is continuous and bounded, and therefore we even have existence of the limit.

\textit{Step~3.} Lastly assume that \(\mu\) is compactly supported in \([a, b] \subset (0, \infty)\) and has a continuous Lebesgue density \(h\). Call \(\upsilon := \Gleichv([a, b])\), and for any \(k \in \NZ\) define \(\mu_k := (1 - \frac{1}{k}) \mu + \frac{1}{k} \, \upsilon\). Clearly each \(\mu_k\) is supported on \([a, b]\) and has Lebesgue density \(h_k := (1 - \frac{1}{k}) h + \frac{1}{k (b - a)} \Ind_{[a, b]}\) which is continuous on \([a, b]\) and bounded from below by \(\frac{1}{k (b - a)}\) on \([a, b]\). Now we verify (i)--(iii).

\textit{Ad (i).} We see \((h_k)_{k \in \NZ} \to h\) pointwise everywhere on \(\RZ_{\geq 0}\), and convergence of densities implies weak convergence of measures.

\textit{Ad (ii).} Since \(m_{p/2}(\upsilon) < \infty\), we get immediately
\begin{equation*}
m_{p/2}(\mu_k) = \Bigl( 1 - \frac{1}{k} \Bigr) m_{p/2}(\mu) + \frac{1}{k} \, m_{p/2}(\upsilon) \xrightarrow[k \to \infty]{} m_{p/2}(\mu).
\end{equation*}

\textit{Ad (iii).} It is easy to see that, because \(\mu\) has a compactly supported Lebesgue density which is bounded from above, the integral \(\int_{\RZ_{\geq 0}^2} F \, \diff\mu^{\otimes 2}\) exists and is finite; the same holds true for \(\upsilon\) and each \(\mu_k\). Furthermore we have
\begin{equation*}
\int_{\RZ_{\geq 0}^2} F \, \diff\mu_k^{\otimes 2} = \Bigl( 1 - \frac{1}{k} \Bigr)^2 \int_{\RZ_{\geq 0}^2} F \, \diff\mu^{\otimes 2} + \frac{2}{k} \Bigl( 1 - \frac{1}{k} \Bigr) \int_{\RZ_{\geq 0}^2} F \, \diff(\mu \otimes \upsilon) + \frac{1}{k^2} \int_{\RZ_{\geq 0}^2} F \, \diff\upsilon^{\otimes 2},
\end{equation*}
and since all integrals on the right\-/hand side are finite it converges to \(\int_{\RZ_{\geq 0}^2} F \, \diff\mu^{\otimes 2}\) as \(k \to \infty\), as desired.

Now to put all the steps together, let \(\mu\) be arbitrary again as the premises of the lemma require. Fix a metric \(d_\text{\upshape w}\) inducing the weak topology on \(\WMasz(\RZ_{\geq 0})\). Let \(k \in \NZ\), then by Step~1 there exists \(\mu_k' \in \WMasz(\RZ_{\geq 0})\) supported on \([a_k, b_k] \subset (0, \infty)\) such that \(d_\text{\upshape w}(\mu, \mu_k') \leq \frac{1}{3 k}\) and \(\lvert m_{p/2}(\mu_k') - m_{p/2}(\mu) \rvert \leq \frac{1}{3 k}\) and \(\int_{\RZ_{\geq 0}^2} F \, \diff(\mu_k')^{\otimes 2} \leq \int_{\RZ_{\geq 0}^2} F \, \diff\mu^{\otimes 2} + \frac{1}{3 k}\).

By Step~2 we can choose \(\mu_k'' \in \WMasz(\RZ_{\geq 0})\) supported in \([\frac{a_k}{2}, b_k + \frac{a_k}{2}]\) with continuous Lebesgue density such that \(d_\text{\upshape w}(\mu_k', \mu_k'') \leq \frac{1}{3 k}\) and \(\lvert m_{p/2}(\mu_k'') - m_{p/2}(\mu_k') \rvert \leq \frac{1}{3 k}\) and \(\int_{\RZ_{\geq 0}^2} F \, \diff(\mu_k'')^{\otimes 2} \leq \int_{\RZ_{\geq 0}^2} F \, \diff(\mu_k')^{\otimes 2} + \frac{1}{3 k}\).

Lastly by Step~3 we can find \(\mu_k \in \WMasz(\RZ_{\geq 0})\)  supported on \([\frac{a_k}{2}, b_k + \frac{a_k}{2}]\) with Lebesgue density \(h_k\) which is continuous on its support and satisfies \(\min h_k([\frac{a_k}{2}, b_k + \frac{a_k}{2}]) > 0\), such that \(d_\text{\upshape w}(\mu_k'', \mu_k) \leq \frac{1}{3 k}\) and \(\lvert m_{p/2}(\mu_k) - m_{p/2}(\mu_k'') \rvert \leq \frac{1}{3 k}\) and \(\int_{\RZ_{\geq 0}^2} F \, \diff\mu_k^{\otimes 2} \leq \int_{\RZ_{\geq 0}^2} F \, \diff(\mu_k'')^{\otimes 2} + \frac{1}{3 k}\).

Finally the triangle inequality yields \(d_\text{\upshape w}(\mu, \mu_k) \leq \frac{1}{k}\) and \(\lvert m_{p/2}(\mu_k) - m_{p/2}(\mu) \rvert \leq \frac{1}{k}\) and \(\int_{\RZ_{\geq 0}^2} F \, \diff\mu_k^{\otimes 2} \leq \int_{\RZ_{\geq 0}^2} F \, \diff\mu^{\otimes 2} + \frac{1}{k}\), and thus \((\mu_k)_{k \in \NZ}\) fulfills all of (i)--(iv).
\end{proof}

Now we have gathered all preparatory results to prove the following statement.

\begin{Lem}\label{lem:ldp_untere_abschaetzung2}
Define the functional \(\mathcal{I} \colon \WMasz(\RZ_{\geq 0}) \times \RZ_{\geq 0} \to \schluss{\RZ}\) by
\begin{equation*}
\mathcal{I}(\mu, M) := \inf_G \liminf_{n \to \infty} \frac{1}{\beta m n} \log \Wsk\bigl[ (\nu_n, m_{p/2}(\nu_n)) \in G \bigr],
\end{equation*}
where \(G\) ranges over basic open subsets of \(\WMasz(\RZ_{\geq 0}) \times \RZ_{\geq 0}\) such that \((\mu, M) \in G\). Then we have
\begin{equation*}
\mathcal{I}(\mu, M) \geq -\mathcal{J}_{c, p}^2(\mu, M)
\end{equation*}
for any \(\mu \in \WMasz(\RZ_{\geq 0})\) and \(M \in \RZ_{\geq 0}\), with \(\mathcal{J}_{c, p}^2\) defined in Proposition~\ref{sa:ldp_empir_paar}.
\end{Lem}

\begin{proof}
First we note a few important properties of \(\mathcal{I}\). \(\mathcal{I}\) is of the form \(\mathcal{I}(\mu, M) = \inf_G I(G)\) with an isotone map \(I\) (i.e., \(U \subset V\) implies \(I(U) \leq I(V)\)). The system of basic open sets \(G \in (\mu, M)\) is a neighborhood basis of \((\mu, M)\), and then the form of \(\mathcal{I}\) yields that we may choose any specific neighborhood basis for any point \((\mu, M)\) without changing the value \(\mathcal{I}(\mu, M)\). Also, as rightly observed in \cite[p.\ 216]{HiaiPetz2000}, \(\mathcal{I}\) is upper semicontinuous.

Now we are going to prove the statement of this lemma. If \(\mu \in \WMasz(\RZ_{\geq 0})\) and \(M \in \RZ_{\geq 0}\) are such that \(\mathcal{J}_{c, p}^2(\mu, M) = \infty\), then trivially there is nothing left to prove. By the definition of \(\mathcal{J}_{c, p}^2\) this definitely is the case if \(m_{p/2}(\mu) > M\), or if \(m_{p/2}(\mu) \leq M\) and \(\mu(\{0\}) > 0\).

So consider such \(\mu \in \WMasz(\RZ_{\geq 0})\) and \(M \in \RZ_{\geq 0}\) that \(\mathcal{J}_{c, p}^2(\mu, M) < \infty\), then \(m_{p/2}(\mu) \leq M\) and \(\mu(\{0\}) = 0\), which also implies \(m_{p/2}(\mu) > 0\) and therefore \(M > 0\). First assume that \(\mu\) is supported on a compact interval \([a, b] \subset (0, \infty)\) with Lebesgue density \(h\) which is continuous on \([a, b]\) and satisfies \(\min h([a, b]) > 0\), and that \(m_{p/2}(\mu) < M\). In this situation we can apply the analogue of \cite[Lemma~4.5]{KPTh2020_2} and Lemma~\ref{lem:ldp_untere_abschaetzung} to get
\begin{align*}
\liminf_{n \to \infty} \frac{1}{\beta m n} &\log \Wsk\bigl[ \nu_n \in \mathcal{O}_{\epsilon, d}(\mu) \wedge \lvert m_{p/2}(\nu_n) - M \rvert < \delta \bigr]\\
&\geq \liminf_{n \to \infty} \frac{1}{\beta m n} \log \Wsk\Bigl[ \nu_n' \in \mathcal{O}_{\epsilon/3, d}(\mu) \wedge \lvert m_{p/2}(\nu_n') - m_{p/2}(\mu) \rvert < \frac{\delta}{3} \wedge Y_m^{(n)} \in D_n \Bigr]\\
&\geq -\mathcal{J}_{c, p}^2(\mu, M)
\end{align*}
for any \(\delta, \epsilon \in \RZ_{> 0}\), \(d \in \NZ\), and \(f_1, \dotsc, f_d \in \CeBe(\RZ_{\geq 0})\). Since
\begin{equation*}
\bigl\{ \mathcal{O}_{\epsilon, d}(\mu) \times (M - \delta, M + \delta) \Colon \delta, \epsilon \in \RZ_{> 0}, d \in \NZ, f_1, \dotsc, f_d \in \CeBe(\RZ_{\geq 0}) \bigr\}
\end{equation*}
is a neighborhood basis at \((\mu, M)\), taking the infimum over neighborhoods yields
\begin{equation*}
\mathcal{I}(\mu, M) \geq -\mathcal{J}_{c, p}^2(\mu, M),
\end{equation*}
as claimed.

Now let \(\mu \in \WMasz(\RZ_{\geq 0})\) and \(M \in \RZ_{\geq 0}\) be general, but still such that \(\mathcal{J}_{c, p}^2(\mu, M) < \infty\). Recalling the definition of the kernel \(F\) introduced before Lemma~\ref{lem:approx_maszfolge}, we can rewrite
\begin{equation*}
\mathcal{J}_{c, p}^2(\mu, M) = \frac{1}{2} \int_{\RZ_{\geq 0}^2} F \, \diff\mu^{\otimes 2} + M + B_{c, p},
\end{equation*}
and therefore we know \(\int_{\RZ_{\geq 0}^2} F \, \diff\mu^{\otimes 2} < \infty\). Let \((\mu_k)_{k \in \NZ}\) be a sequence as asserted by Lemma~\ref{lem:approx_maszfolge}, and for each \(k \in \NZ\) define
\begin{equation*}
M_k := m_{p/2}(\mu_k) + M - m_{p/2}(\mu) + \frac{1}{k},
\end{equation*}
then \(M_k > m_{p/2}(\mu_k)\) and \((M_k)_{k \in \NZ} \to M\). This means that for each \(k \in \NZ\) we are in the special situation treated above, hence
\begin{equation*}
\mathcal{I}(\mu_k, M_k) \geq -\mathcal{J}_{c, p}^2(\mu_k, M_k);
\end{equation*}
then letting \(k \to \infty\) and using the upper semicontinuity of \(\mathcal{I}\),
\begin{equation*}
\mathcal{I}(\mu, M) \geq \limsup_{k \to \infty} \mathcal{I}(\mu_k, M_k) \geq \limsup_{k \to \infty} \bigl( -\mathcal{J}_{c, p}^2(\mu_k, M_k) \bigr) = -\liminf_{k \to \infty} \mathcal{J}_{c, p}^2(\mu_k, M_k).
\end{equation*}
Here now the crucial property (iii) of Lemma~\ref{lem:approx_maszfolge} enters the picture, since this leads to
\begin{align*}
\liminf_{k \to \infty} \mathcal{J}_{c, p}^2(\mu_k, M_k) &= \liminf_{k \to \infty} \biggl( \frac{1}{2} \int_{\RZ_{\geq 0}^2} F \, \diff\mu_k^{\otimes 2} + M_k + B_{c, p} \biggr)\\
&= \frac{1}{2} \liminf_{k \to \infty} \int_{\RZ_{\geq 0}^2} F \, \diff\mu_k^{\otimes 2} + M + B_{c, p}\\
&\leq \frac{1}{2} \limsup_{k \to \infty} \int_{\RZ_{\geq 0}^2} F \, \diff\mu_k^{\otimes 2} + M + B_{c, p}\\
&\leq \frac{1}{2} \int_{\RZ_{\geq 0}^2} F \, \diff\mu^{\otimes 2} + M + B_{c, p} = \mathcal{J}_{c, p}^2(\mu, M),
\end{align*}
and together with the display before this concludes the proof.
\end{proof}

\cite[Lemma~4.7]{KPTh2020_2} actually need not be true as stated; the flaw happens on p.~945, at `Now, if we take \(\epsilon \mathbin{\downarrow} 0\) as well as \(\delta \mathbin{\downarrow} 0\), then': even though the functional \(\nu \mapsto \int_{\RZ^2} F_\alpha(x, y; \gamma) \, \diff\nu^{\otimes 2}(x, y)\) be weakly continuous, this is \emph{not} sufficient for
\begin{equation*}
\lim_{\epsilon \to 0+} \inf_{\nu \in \mathcal{O}_{\epsilon, d}(\mu)} \int_{\RZ^2} F_\alpha(x, y; \gamma) \, \diff\nu^{\otimes 2}(x, y) = \int_{\RZ^2} F_\alpha(x, y; \gamma) \, \diff\mu^{\otimes 2}(x, y),
\end{equation*}
if the neighbourhoods \(\mathcal{O}_{\epsilon, d}(\mu)\) depend on fixed functions \(f_1, \dotsc, f_d \in \CeBe(\RZ)\). But in truth this setup is not needed since we actually wish to establish (reverting to our notation)
\begin{equation*}
\inf_{G} \limsup_{n \to \infty} \frac{1}{\beta m n} \log \Wsk\bigl[ (\nu_n, m_{p/2}(\nu_n)) \in G \bigr] \leq -\mathcal{J}_{c, p}^2(\mu, M),
\end{equation*}
where \(G\) runs through a neighborhood basis of \((\mu, M)\); thus it suffices to consider a specific subset of basic neighborhoods over which to minimize, as then the global infimum can only be smaller. The precise statement is given next.

\begin{Lem}\label{lem:ldp_obere_abschaetzung}
Let \(\mu \in \WMasz(\RZ_{\geq 0})\) and \(M \in \RZ_{> 0}\) with \(m_{p/2}(\mu) \leq M\). Then there exists a family \((\mathcal{O}_\epsilon)_{\epsilon \in \RZ_{> 0}}\) of basic neighborhoods of \(\mu\) such that
\begin{multline*}
\lim_{\delta, \epsilon \to 0+} \limsup_{n \to \infty} \frac{1}{\beta m n} \log \Wsk\bigl[ \nu_n \in \mathcal{O}_\epsilon \wedge m_{p/2}(\nu_n) \in (M - \delta, M + \delta) \bigr]\\
\leq \frac{c}{2} \int_{\RZ_{\geq 0}^2} \log\lvert x - y \rvert \, \diff\mu^{\otimes 2}(x, y) + \frac{1 - c}{2} \int_{\RZ_{\geq 0}} \log(x) \, \diff\mu(x) - M - B_{c, p}.
\end{multline*}
\end{Lem}

Before proving Lemma~\ref{lem:ldp_obere_abschaetzung} we give a purely auxiliary result which by the way nicely illustrates how one can evaluate integrals with respect to cone measures.

\begin{Lem}\label{lem:hom_integral}
Let \(q \in \RZ_{> 0}\) and \(\alpha \in (-1, \infty)\), then
\begin{equation*}
\int_{\Sph{q, 1}{m}} \Ind_{\RZ_{> 0}^m}(\theta) \prod_{i = 1}^m \theta_i^\alpha \, \diff\KegM{q, 1}{m}(\theta) = \frac{\Gamma\bigl( \frac{\alpha + 1}{q} \bigr)^m}{m \KugVol{q, 1}{m} \, q^{m - 1} \, \Gamma\bigl( \frac{m (\alpha + 1)}{q} \bigr)}.
\end{equation*}
\end{Lem}

\begin{proof}
The basic property of the integrand which we are going to exploit is its positive homogeneity. By a standard method the integral is extended to all of \(\RZ_{> 0}^m\) via introducing an exponential term. To be specific, we consider
\begin{equation}\label{eq:hom_integral}
\int_{\RZ_{> 0}^m} \ez^{-\lVert x \rVert_q^q} \prod_{i = 1}^m x_i^\alpha \, \diff x.
\end{equation}
On the one hand use polar integration on \eqref{eq:hom_integral} to obtain
\begin{align*}
\int_{\RZ_{> 0}^m} \ez^{-\lVert x \rVert_q^q} \prod_{i = 1}^m x_i^\alpha \, \diff x &= m \KugVol{q, 1}{m} \int_{\RZ_{> 0}} \int_{\Sph{q, 1}{m}} \Ind_{\RZ_{> 0}^m}(r \theta) \ez^{-\lVert r \theta \rVert_q^q} \prod_{i = 1}^m (r \theta_i)^\alpha \, \diff\KegM{q, 1}{m}(\theta) \, r^{m - 1} \, \diff r\\
&= m \KugVol{q, 1}{m} \int_{\RZ_{> 0}} r^{m \alpha + m - 1} \, \ez^{-r^q} \, \diff r \int_{\Sph{q, 1}{m}} \Ind_{\RZ_{> 0}^m}(\theta) \prod_{i = 1}^m \theta_i^\alpha \, \diff\KegM{q, 1}{m}(\theta)\\
&= \frac{m \KugVol{q, 1}{m} \, \Gamma\bigl( \frac{m (\alpha + 1)}{q} \bigr)}{q} \int_{\Sph{q, 1}{m}} \Ind_{\RZ_{> 0}^m}(\theta) \prod_{i = 1}^m \theta_i^\alpha \, \diff\KegM{q, 1}{m}(\theta).
\end{align*}
On the other hand the integrand of \eqref{eq:hom_integral} has tensor product structure, therefore
\begin{align*}
\int_{\RZ_{> 0}^m} \ez^{-\lVert x \rVert_q^q} \prod_{i = 1}^m x_i^\alpha \, \diff x &= \biggl( \int_{\RZ_{> 0}} x^\alpha \, \ez^{-x^q} \, \diff x \biggr)^m\\
&= \biggl( \frac{\Gamma\bigl( \frac{\alpha + 1}{q} \bigr)}{q} \biggr)^m.
\end{align*}
Equating the two partial results and rearranging lead to the claim.
\end{proof}

\begin{proof}[Proof of Lemma~\ref{lem:ldp_obere_abschaetzung}]
The strategy is the same as in~\cite{KPTh2020_2}. For \(\alpha \in \RZ_{\geq 0}\) and \(\gamma \in [0, 1)\) define the kernels \(F(\cdot, \cdot; \gamma), F_\alpha(\cdot, \cdot; \gamma) \colon \RZ_{\geq 0}^2 \to (-\infty, \infty]\) by
\begin{equation*}
F(x, y; \gamma) := \begin{cases} -\frac{c}{2} \log\lvert x - y \rvert - (1 - \gamma) \frac{1 - c}{4} \log(x y) + \frac{\gamma}{2} (x^{p/2} + y^{p/2}) & \text{if } x > 0 \wedge y > 0 \wedge x \neq y,\\
\infty & \text{else}, \end{cases}
\end{equation*}
and
\begin{equation*}
F_\alpha(x, y; \gamma) = \min\{F(x, y; \gamma), \alpha\}.
\end{equation*}
We note some properties of \(F\) and \(F_\alpha\): for \(\gamma > 0\), \(F(\cdot, \cdot; \gamma)\) and \(F_\alpha(\cdot, \cdot; \gamma)\) are bounded from below, hence their (coinciding) negative parts are \(\mu^{\otimes 2}\)\-/integrable; and for \(\gamma = 0\) they can be bounded from below by \(-C (x^{p/2} + y^{p/2} + (x y)^{p/2})\) with some \(C > 0\), and since we have been supposing \(m_{p/2}(\mu) \leq M < \infty\), also in this case the negative parts are \(\mu^{\otimes 2}\)\-/integrable. From this follows that for any \(\alpha \geq 0\) and \(\gamma \in [0, 1)\), \(F_\alpha(\cdot, \cdot; \gamma)\) is \(\mu^{\otimes 2}\)\-/integrable. Also we can rewrite
\begin{equation*}
F(x, y; \gamma) = F(x, y; 0) + \frac{\gamma}{2} \Bigl( x^{p/2} + y^{p/2} - \frac{1 - c}{2} \log(x y) \Bigr),
\end{equation*}
wherever the value is finite, which clearly implies \(\lim_{\gamma \to 0+} F(x, y; \gamma) = F(x, y; 0)\) pointwise, and the map \(\gamma \to F(x, y; \gamma)\) is increasing on
\begin{equation*}
P := \Bigl\{ (x, y) \in \RZ_{\geq 0}^2 \Colon x^{p/2} + y^{p/2} - \frac{1 - c}{2} \log(x y) \geq 0 \Bigr\}
\end{equation*}
and decreasing on its complement, \(N := \RZ_{\geq 0}^2 \setminus P\); this holds true for \(F_\alpha\) too. Notice that \(P\) and \(N\) do not depend on \(\gamma\). Lastly, for any \(\gamma \in [0, 1)\) and \((x, y) \in \RZ_{\geq 0}^2\) the map \(\alpha \mapsto F_\alpha(x, y; \gamma)\) is increasing and \(\lim_{\alpha \to \infty} F_\alpha(x, y; \gamma) = F(x, y; \gamma)\).

Fix a metric \(\dw\) on \(\WMasz(\RZ_{\geq 0})\) which metrizes the weak topology, then for any \(\epsilon \in \RZ_{> 0}\) choose a basic weak neighborhood \(\mathcal{O}_\epsilon\) of \(\mu\) with \(\mathcal{O}_\epsilon \subset B_\epsilon(\mu)\) (the open ball w.r.t.\ \(\dw\) around \(\mu\) with radius \(\epsilon\)).

Let \(\mathcal{F} \colon \WMasz(\RZ_{\geq 0}) \to (-\infty, \infty]\) be some functional, then the map \(\epsilon \mapsto \inf\bigl( \mathcal{F}(B_\epsilon(\mu)) \bigr)\) is decreasing, hence the improper limit \(\lim_{\epsilon \to 0+} \inf\bigl( \mathcal{F}(B_\epsilon(\mu)) \bigr)\) exists; next, for any \(\epsilon > 0\) we may choose an \(\epsilon'(\epsilon) \in (0, \epsilon]\) such that \(B_{\epsilon'(\epsilon)}(\mu) \subset \mathcal{O}_\epsilon \subset B_\epsilon(\mu)\), then as \(\epsilon \to 0+\), so also \(\epsilon'(\epsilon) \to 0+\), and hence
\begin{equation*}
\lim_{\epsilon \to 0+} \inf(\mathcal{F}(\mathcal{O}_\epsilon)) = \lim_{\epsilon \to 0+} \inf\bigl( \mathcal{F}(B_\epsilon(\mu)) \bigr).
\end{equation*}
If in addition \(\mathcal{F}\) is continuous at \(\mu\), then there follows
\begin{equation}\label{eq:stetigkeit_infimum}
\lim_{\epsilon \to 0+} \inf(\mathcal{F}(\mathcal{O}_\epsilon)) = \mathcal{F}(\mu).
\end{equation}
Indeed, the inequality `\(\leq\)' readily follows from \(\inf(\mathcal{F}(\mathcal{O}_\epsilon)) \leq \mathcal{F}(\mu)\) for any \(\epsilon > 0\) (since \(\mu \in \mathcal{O}_\epsilon\)); and concerning `\(\geq\)' it suffices to prove \(\lim_{\epsilon \to 0+} \inf(\mathcal{F}(\mathcal{O}_\epsilon)) \geq \mathcal{F}(\mu) - \zeta\) for any \(\zeta \in \RZ_{> 0}\). But continuity implies that for any \(\zeta > 0\) there exists \(\epsilon' > 0\) such that \(\mathcal{F}(\nu) \geq \mathcal{F}(\mu) - \zeta\) for any \(\nu \in B_{\epsilon'}(\mu)\); in particular this remains valid for all \(\nu \in  \mathcal{O}_{\epsilon'}\), and by passing to the infimum and then investing \(\lim_{\epsilon \to 0+} \inf(\mathcal{F}(\mathcal{O}_\epsilon)) \geq \inf(\mathcal{F}(\mathcal{O}_{\epsilon'}))\) we arrive at~\eqref{eq:stetigkeit_infimum}.

As observed in \cite{KPTh2020_2}, for any \(\alpha > 0\) and \(\gamma \in (0, 1)\) the map
\begin{equation*}
\nu \mapsto \int_{\RZ_{\geq 0}^2} F_\alpha(x, y; \gamma) \, \diff\nu^{\otimes 2}(x, y)
\end{equation*}
is weakly continuous, and therefore identity~\eqref{eq:stetigkeit_infimum} can be applied.

Now that we have finished the preparations, let \(\delta, \epsilon \in \RZ_{> 0}\), w.l.o.g.\ \(\delta < M\), and define
\begin{equation*}
H := \biggl\{ y \in \RZ_{\geq 0}^m \Colon \frac{1}{m} \sum_{i = 1}^m \delta_{y_i} \in \mathcal{O}_\epsilon \wedge \frac{1}{m} \sum_{i = 1}^m y_i^{p/2} \in (M - \delta, M + \delta) \biggr\};
\end{equation*}
then
\begin{align*}
\Wsk\bigl[ &\nu_n \in \mathcal{O}_\epsilon \wedge m_{p/2}(\nu_n) \in (M - \delta, M + \delta) \bigr]\\
&= \frac{1}{Z_{m, n, p, \beta}} \int_H \ez^{-\beta n \lVert y \rVert_{p/2}^{p/2}} \prod_{i = 1}^m y_i^{\beta (n - m + 1)/2 - 1} \prod_{1 \leq i < j \leq m} \lvert y_j - y_i \rvert^\beta \, \diff y\\
&= \frac{1}{Z_{m, n, p, \beta}} \int_H \exp\biggl( -\beta n \Bigl( 1 - \frac{\gamma (m - 1)}{c n} \Bigr) \lVert y \rVert_{p/2}^{p/2} \biggr) \exp\biggl( -\frac{2 \beta}{c} \sum_{1 \leq i < j \leq m} F(y_i, y_j;  \gamma) \biggr)\\
&\mspace{100mu} \cdot \exp\biggl( \Bigl( \frac{\beta (n - m + 1)}{2} - 1 - \frac{\beta (1 - \gamma) (1 - c) (m - 1)}{2 c} \Bigr) \sum_{i = 1}^m \log(y_i) \biggr) \, \diff y,
\end{align*}
where we have used the definition of \(F\) to rewrite \(\prod_{i < j} \lvert y_j - y_i \rvert^\beta\); we want to estimate the integral from above. First note \(\lim_{n \to \infty} (1 - \frac{\gamma (m - 1)}{c n}) = 1 - \gamma > 0\), hence the first exponential term may be estimated on \(H\) from above by \(\exp\bigl( \beta n (1 - \frac{\gamma (m - 1)}{c n}) m (M - \delta) \bigr)\) for all sufficiently large \(n\) (so also in the sequel). Next there is
\begin{align*}
\sum_{1 \leq i < j \leq m} F(y_i, y_j; \gamma) &\geq \sum_{1 \leq i < j \leq m} F_\alpha(y_i, y_j; \gamma)\\
&= \frac{m^2}{2} \int_{\RZ_{\geq 0}^2} F_\alpha(x_1, x_2; \gamma) \, \diff\mu_y^{\otimes 2}(x_1, x_2) - \frac{m \alpha}{2}\\
&\geq \frac{m^2}{2} \inf_{\nu \in \mathcal{O}_\epsilon} \int_{\RZ_{\geq 0}^2} F_\alpha(x_1, x_2; \gamma) \, \diff\nu^{\otimes 2}(x_1, x_2) - \frac{m \alpha}{2},
\end{align*}
where of course \(\mu_y := \frac{1}{m} \sum_{i = 1}^m \delta_{y_i}\), and this treats the second exponential term. The third exponential term needs special care; even though we have
\begin{equation*}
\lim_{n \to \infty} \frac{1}{\beta n} \Bigl( \frac{\beta (n - m + 1)}{2} - 1 - \frac{\beta (1 - \gamma) (1 - c) (m - 1)}{2 c} \Bigr) = \frac{\gamma (1 - c)}{2} \geq 0,
\end{equation*}
in the case \(c = 1\) the coefficient of \(\sum_{i = 1}^m \log(y_i)\) can be negative and hence naively maximizing the exponent over \(\{y \in \RZ_{\geq 0}^m \Colon m (M - \delta) < \lVert y \rVert_{p/2}^{p/2} < (M + \delta)\}\) would yield the trivial bound \(\infty\). (As in other places the authors of \cite{KPTh2020_2} remain silent on this point although they would have needed it in the proof of their Theorem~1.5.) Instead we have to find a finer estimate for the following integral,
\begin{equation*}
\int_H \prod_{i = 1}^m y_i^{\beta(n - m + 1)/2 - 1 - \beta (1 - \gamma) (1 - c) (m - 1)/(2 c)} \, \diff y.
\end{equation*}
Call \(\alpha_n := \frac{\beta (n - m + 1)}{2} - 1 - \frac{\beta (1 - \gamma) (1 - c) (m - 1)}{2 c}\), then we already have observed above that \(\lim_{n \to \infty} \frac{\alpha_n}{\beta n} = \frac{\gamma (1 - c)}{2}\), hence in the case \(c < 1\) we know \(\alpha_n \geq 0 > -1\) for all \(n\) large enough, and in the case \(c = 1\) we get \(\alpha_n = \frac{\beta (n - m + 1)}{2} - 1 \geq \frac{\beta}{2} - 1 > -1\) for all \(n \in \NZ\), since \(m \leq n\). Therefore we have \(\alpha_n > -1\) for eventually all \(n\) in any case, and we aim for an application of Lemma~\ref{lem:hom_integral}. Using \(H \subset \bigl\{ y \in \RZ_{\geq 0}^m \Colon m (M - \delta) < \lVert y \rVert_{p/2}^{p/2} < m (M + \delta) \bigr\}\), polar integration leads to
\begin{align*}
\int_H \prod_{i = 1}^m y_i^{\alpha_n} \, \diff y &\leq m \KugVol{p/2, 1}{m} \int_{(m (M - \delta))^{2/p}}^{(m (M + \delta))^{2/p}} \int_{\Sph{p/2, 1}{m}} \Ind_{\RZ_{\geq 0}^m}(r \theta) \prod_{i = 1}^m (r \theta_i)^{\alpha_n} \, \diff\KegM{p/2, 1}{m}(\theta) \, r^{m - 1} \, \diff r\\
&= m \KugVol{p/2, 1}{m} \int_{(m (M - \delta))^{2/p}}^{(m (M + \delta))^{2/p}} r^{m (\alpha_n + 1) - 1} \, \diff r \int_{\Sph{p/2, 1}{m}} \Ind_{\RZ_{\geq 0}^m}(\theta) \prod_{i = 1}^m \theta_i^{\alpha_n} \, \diff\KegM{p/2, 1}{m}(\theta)\\
&\leq \frac{(m (M + \delta))^{2 m (\alpha_n + 1)/p}}{m (\alpha_n + 1)} \cdot \frac{\Gamma\bigl( \frac{2 (\alpha_n + 1)}{p} \bigr)^m}{\bigl( \frac{p}{2} \bigr)^{m - 1} \, \Gamma\bigl( \frac{2 m (\alpha_n + 1)}{p} \bigr)},
\end{align*}
where for the last line we have invested Lemma~\ref{lem:hom_integral} (obviously it is immaterial whether we integrate over \(\RZ_{\geq 0}^m\) or \(\RZ_{> 0}^m\)). Summarizing all this we see
\begin{align*}
\Wsk\bigl[ \nu_n \in \mathcal{O}_\epsilon \wedge m_{p/2}(\nu_n) &\in (M - \delta, M + \delta) \bigr]\\
&\leq \frac{1}{Z_{m, n, p, \beta}} \exp\biggl( -\beta m n (M - \delta) \Bigl( 1 - \frac{\gamma (m - 1)}{c n} \Bigr) \biggr)\\
&\quad \cdot \exp\biggl( -\frac{\beta m^2}{c} \inf_{\nu \in \mathcal{O}_\epsilon} \int_{\RZ_{\geq 0}^2} F_\alpha(x, y; \gamma) \, \diff\nu^{\otimes 2}(x, y) + \frac{\alpha \beta m}{c} \biggr)\\
&\quad \cdot \frac{(m (M + \delta))^{2 m (\alpha_n + 1)/p}}{m (\alpha_n + 1)} \cdot \frac{\Gamma\bigl( \frac{2 (\alpha_n + 1)}{p} \bigr)^m}{\bigl( \frac{p}{2} \bigr)^{m - 1} \, \Gamma\bigl( \frac{2 m (\alpha_n + 1)}{p} \bigr)}.
\end{align*}
Proposition~\ref{sa:ldp_y} tells us \(\lim_{n \to \infty} \frac{1}{\beta m n} \log(Z_{m, n, p, \beta}) = B_{c, p}\). The first exponential term needs no comment and for the second we get
\begin{align*}
\lim_{n \to \infty} &\frac{1}{\beta m n} \log\Bigl( \ez^{-\frac{\beta m^2}{c} \inf_{\nu \in \mathcal{O}_\epsilon} \int_{\RZ_{\geq 0}^2} F_\alpha(x, y; \gamma) \, \diff\nu^{\otimes 2}(x, y) + \frac{\alpha \beta m}{c}} \Bigr)\\
&= -\inf_{\nu \in \mathcal{O}_\epsilon} \int_{\RZ_{\geq 0}^2} F_\alpha(x, y; \gamma) \, \diff\nu^{\otimes 2}(x, y).
\end{align*}
Stirling's formula \(\Gamma(z) = \sqrt{2 \pi} \, z^{z - 1/2} \, \ez^{-z + R(z)}\), where \(R(z) = \frac{1}{12 z} \bigl( 1 + \BigO(z^{-2}) \bigr)\), implies
\begin{equation*}
\frac{m^{2 m (\alpha_n + 1)/p} \, \Gamma\bigl( \frac{2 (\alpha_n + 1)}{p} \bigr)^m}{\Gamma\bigl( \frac{2 m (\alpha_n + 1)}{p} \bigr)} = m^{1/2} \Bigl( \frac{\pi p}{\alpha_n + 1} \Bigr)^{(m - 1)/2} \, \ez^{m R(2 (\alpha_n + 1)/p) - R(2 m (\alpha_n + 1)/p)},
\end{equation*}
hence
\begin{align*}
\frac{1}{\beta m n} \log\frac{m^{2 m (\alpha_n + 1)/p} \, \Gamma\bigl( \frac{2 (\alpha_n + 1)}{p} \bigr)^m}{\Gamma\bigl( \frac{2 m (\alpha_n + 1)}{p} \bigr)} &= \frac{\log(m)}{2 \beta m n} + \frac{(m - 1) \log(\pi p)}{2 \beta m n} - \frac{(m + 1) \log(\alpha_n + 1)}{2 \beta m n}\\
&\quad + \frac{R\bigl( \frac{2 (\alpha_n + 1)}{p} \bigr)}{\beta n} - \frac{R\bigl( \frac{2 m (\alpha_n + 1)}{p} \bigr)}{\beta m n}\\
&\xrightarrow[n \to \infty]{} 0,
\end{align*}
where we also have used \(\alpha_n = \BigO(\beta n)\). Lastly there remains
\begin{align*}
\frac{1}{\beta m n} \log\frac{(M + \delta)^{2 m (\alpha_n + 1)/p}}{m (\alpha_n + 1) \bigl( \frac{p}{2} \bigr)^{m - 1}} &= \frac{2 (\alpha_n + 1) \log(M + \delta)}{p \beta n} - \frac{\log(m (\alpha_n + 1))}{\beta m n} - \frac{(m - 1) \log\bigl( \frac{p}{2} \bigr)}{\beta m n}\\
&\xrightarrow[n \to \infty]{} \frac{\gamma (1 - c) \log(M + \delta)}{p}.
\end{align*}
Therewith we get
\begin{multline*}
\limsup_{n \to \infty} \frac{1}{\beta m n} \log\Wsk\bigl[ \nu_n \in \mathcal{O}_\epsilon \wedge m_{p/2}(\nu_n) \in (M - \delta, M + \delta) \bigr]\\
\leq -B_{c, p} - (M - \delta) (1 - \gamma) - \inf_{\nu \in \mathcal{O}_\epsilon} \int_{\RZ_{\geq 0}^2} F_\alpha(x, y; \gamma) \, \diff\nu^{\otimes 2}(x, y) + \frac{\gamma (1 - c)}{p} \log(M + \delta).
\end{multline*}
Now we send \(\delta, \epsilon \to 0+\), so by our discussion at the preparations for the proof concerning the continuity of infima of continuous functionals we have
\begin{multline*}
\lim_{\delta, \epsilon \to 0+} \limsup_{n \to \infty} \frac{1}{\beta m n} \log\Wsk\bigl[ \nu_n \in \mathcal{O}_\epsilon \wedge m_{p/2}(\nu_n) \in (M - \delta, M + \delta) \bigr]\\
\leq -B_{c, p} - M (1 - \gamma) - \int_{\RZ_{\geq 0}^2} F_\alpha(x, y; \gamma) \, \diff\mu^{\otimes 2}(x, y) + \frac{\gamma (1 - c)}{p} \log(M).
\end{multline*}
Next we take \(\gamma \to 0+\), then clearly \(F_\alpha(x, y; \gamma) \to F_\alpha(x, y; 0)\). Recall the sets \(P\) and \(N\) from the beginning of this proof; on \(P\), \(F_\alpha(\cdot, \cdot; \gamma)\) is bounded from below by \(F_\alpha(\cdot, \cdot; 0)\), and on \(N\), it is bounded from above by \(\alpha\); either bound is \(\mu^{\otimes 2}\)\-/quasiintegrable, hence monotone convergence can be applied on \(P\) and on \(N\) separately, and because the negative part of the limiting function \(F_\alpha(\cdot, \cdot; 0)\) is \(\mu^{\otimes 2}\)\-/integrable over \(\RZ_{\geq 0}^2\) we also have convergence of the integral on all of \(\RZ_{\geq 0}^2\), that is,
\begin{equation*}
\lim_{\gamma \to 0+} \int_{\RZ_{\geq 0}^2} F_\alpha(x, y; \gamma) \, \diff\mu^{\otimes 2}(x, y) = \int_{\RZ_{\geq 0}^2} F_\alpha(x, y; 0) \, \diff\mu^{\otimes 2}(x, y).
\end{equation*}
Lastly we let \(\alpha \to \infty\), then \(F_\alpha(\cdot, \cdot; 0)\) is dominated from below by \(F(\cdot, \cdot; 0)\) whose negative part is integrable as we have discussed above. Thus monotone convergence can be used once more, so
\begin{equation*}
\lim_{\alpha \to \infty} \int_{\RZ_{\geq 0}^2} F_\alpha(x, y; 0) \, \diff\mu^{\otimes 2}(x, y) = \int_{\RZ_{\geq 0}^2} F(x, y; 0) \, \diff\mu^{\otimes 2}(x, y).
\end{equation*}
Plugging in the definition of \(F\) we finally have
\begin{multline*}
\lim_{\delta, \epsilon \to 0+} \limsup_{n \to \infty} \frac{1}{\beta m n} \log \Wsk\bigl[ \nu_n \in \mathcal{O}_\epsilon \wedge m_{p/2}(\nu_n) \in (M - \delta, M + \delta) \bigr]\\
\leq -B_{c, p} - M + \int_{\RZ_{\geq 0}^2} \Bigl( \frac{c}{2} \log\lvert x - y \rvert + \frac{1 - c}{4} \log(x y) \Bigr) \diff\mu^{\otimes 2}(x, y),
\end{multline*}
and slight simplification leads to the claimed result.
\end{proof}

\cite[Lemma~4.8]{KPTh2020_2}, which concerns the exponential tightness, carries over analogously, where the reference point is  \cite[Theorem~5.5.1, p.\ 232]{HiaiPetz2000} now.

We also adapt \cite[Remark~4.9]{KPTh2020_2}, that is, the LDP on \(\WMasz(\RZ_{\geq 0}) \times \RZ_{> 0}\); the proof of
\begin{equation*}
\lim_{A \to \infty} \limsup_{n \to \infty} \frac{1}{\beta m n} \log \Wsk\Bigl[ m_{p/2}(\nu_n) < \frac{1}{A} \Bigr] = -\infty
\end{equation*}
needs to be modified as follows: We have, for any \(A \in \RZ_{> 0}\),
\begin{align*}
\Wsk\Bigl[ &m_{p/2}(\nu_n) < \frac{1}{A} \Bigr] = \Wsk\Bigl[ \lVert Y^{(n)} \rVert_{p/2} < \Bigl( \frac{m}{A} \Bigr)^{2/p} \Bigr]\\
&= \frac{1}{Z_{m, n, p, \beta}} \int_{\RZ_{\geq 0}^m} \Ind_{(\frac{m}{A})^{2/p} \Kug{p/2, 1}{m}}(y) \, \ez^{-\beta n \lVert y \rVert_{p/2}^{p/2}} \prod_{i = 1}^m y_i^{\beta (n - m + 1)/2 - 1} \prod_{1 \leq i < j \leq m} \lvert y_j - y_i \rvert^\beta \, \diff y\\
&\leq \frac{1}{Z_{m, n, p, \beta}} \int_{\RZ_{\geq 0}^m} \Ind_{(\frac{m}{A})^{2/p} \Kug{p/2, 1}{m}}(y) \prod_{i = 1}^m y_i^{\beta (n - m + 1)/2 - 1} \prod_{1 \leq i < j \leq m} \lvert y_j - y_i \rvert^\beta \, \diff y\\
&= \frac{1}{Z_{m, n, p, \beta}} \Bigl( \frac{m}{A} \Bigr)^{\beta m n/p} \int_{\RZ_{\geq 0}^m} \Ind_{\Kug{p/2, 1}{m}}(z) \prod_{i = 1}^m z_i^{\beta (n - m + 1)/2 - 1} \prod_{1 \leq i < j \leq m} \lvert z_j - z_i \rvert^\beta \, \diff z,
\end{align*}
where in the last step we have substituted \(z := \bigl( \frac{A}{m} \bigr)^{2/p} y\). The last integral then is transformed with the same method as in the proof of Theorem~\ref{sa:gleichvert_schattenq_kugel_sphaere}, part~1, that is,
\begin{align*}
\int_{\RZ_{\geq 0}^m} \Ind_{\Kug{p/2, 1}{m}}(z) &\prod_{i = 1}^m z_i^{\beta (n - m + 1)/2 - 1} \prod_{1 \leq i < j \leq m} \lvert z_j - z_i \rvert^\beta \, \diff z\\
&= \frac{(\beta n)^{\beta m n/p}}{\Gamma(\frac{\beta m n}{p} + 1)} \int_{\RZ_{\geq 0}^m} \ez^{-\beta n \lVert y \rVert_{p/2}^{p/2}} \prod_{i = 1}^m y_i^{\beta (n - m + 1)/2 - 1} \prod_{1 \leq i < j \leq m} \lvert y_j - y_i \rvert^\beta \, \diff y\\
&= \frac{(\beta n)^{\beta m n/p} \, Z_{m, n, p, \beta}}{\Gamma(\frac{\beta m n}{p} + 1)};
\end{align*}
this leads to
\begin{equation*}
\Wsk\Bigl[ m_{p/2}(\nu_n) < \frac{1}{A} \Bigr] \leq \frac{(\beta m n)^{\beta m n/p}}{A^{\beta m n/p} \, \Gamma(\frac{\beta m n}{p} + 1)},
\end{equation*}
and thence
\begin{equation*}
\frac{1}{\beta m n} \log \Wsk\Bigl[ m_{p/2}(\nu_n) < \frac{1}{A} \Bigr] \leq \frac{\log(\beta m n) - \log(A)}{p} - \frac{1}{\beta m n} \log \Gamma\Bigl( \frac{\beta m n}{p} + 1 \Bigr).
\end{equation*}
Stirling's formula yields \(\lim_{n \to \infty} \bigl( \frac{1}{\beta m n} \log \Gamma(\frac{\beta m n}{p} + 1) - \frac{1}{p} \log(\frac{\beta m n}{p}) \bigr) = -\frac{1}{p}\), hence
\begin{equation*}
\limsup_{n \to \infty} \frac{1}{\beta m n} \log\Wsk\Bigl[ m_{p/2}(\nu_n) < \frac{1}{A} \Bigr] \leq \frac{1 - \log(A)}{p},
\end{equation*}
so the claim follows as \(A \to \infty\).

\paragraph{Step~3}
The result of this step is the LDP for the empirical measure of the squares of the singular values, reached via the contraction principle from the LDP for the empirical pair in Step~2.

\begin{Sa}\label{sa:ldp_quadsingulaerwerte_kegm}
Let \((X^{(n)})_{n \in \NZ}\) be a sequence of random variables where \(X^{(n)} \sim \KegM{\Schatten{p}{}, \beta}{m \times n}\) for each \(n \in \NZ\), and define \(\mu_n^2 := \frac{1}{m} \sum_{i = 1}^m \delta_{m^{2/p} s_i(X^{(n)})^2}\). Then \((\mu_n^2)_{n \in \NZ}\) satisfies a large deviations principle with speed \(\beta m n\) and good rate function \(\mathcal{I}_{c, p}^2 \colon \WMasz(\RZ_{\geq 0}) \to [0, \infty]\) given by
\begin{equation*}
\mathcal{I}_{c, p}^2(\mu) := \begin{cases} -\frac{c}{2} \int_{\RZ_{\geq 0}^2} \log\lvert x - y \rvert \, \diff\mu^{\otimes 2}(x, y) - \frac{1 - c}{2} \int_{\RZ_{\geq 0}} \log(x) \, \diff\mu(x) + \frac{\log(\ez p)}{p} + B_{c, p} & \text{if } m_{p/2}(\mu) \leq 1, \\ \infty & \text{else,} \end{cases}
\end{equation*}
which has a unique global minimizer \(\mu_{c, p}^2\) defined by \(\mu_{c, p}^2(A) = \nu_{c, p}\bigl( m_{p/2}(\nu_{c, p})^{2/p} A \bigr)\), where \(\nu_{c, p}\) is the same as in Proposition~\ref{sa:ldp_y}; in particular \(\mu_{c, p}^2\) has compact support and a Lebesgue density.
\end{Sa}

\begin{proof}
The role model here is \cite[Proposition~5.1]{KPTh2020_2} which we adopt with the obvious adaptions; in particular \(p\) has to be replaced by \(\frac{p}{2}\). The contraction principle yields the LPD with speed \(\beta m n\) and good rate function defined by
\begin{equation*}
\mathcal{I}(\mu) := \inf_{(\nu, M) \in \inv{F_{p/2}}\{\mu\}} \mathcal{J}_{c, p}^2(\nu, M).
\end{equation*}
The value \(\mathcal{I}(\mu) = \infty\) for \(m_{p/2}(\mu) > 1\) is argued as in \cite{KPTh2020_2}. In the case \(m_{p/2}(\mu) \leq 1\), note that actually for any \(M \in \RZ_{> 0}\) there exists a \(\nu \in \WMasz(\RZ_{\geq 0})\) such that \(F_{p/2}(\nu, M) = \mu\); simply take \(\nu := F_{p/2}(\mu, \inv{M})\) (this\---important\---point is not addressed in \cite{KPTh2020_2}). Therefore
\begin{align*}
\mathcal{I}(\mu) &= \inf_{(\nu, M) \in \inv{F_{p/2}}\{\mu\}} \biggl( -\frac{c}{2} \int_{\RZ_{\geq 0}^2} \log\lvert x - y \rvert \, \diff\nu^{\otimes 2}(x, y) - \frac{1 - c}{2} \int_{\RZ_{\geq 0}} \log(x) \, \diff\nu(x) + M + B_{c, p} \biggr)\\
&= \inf_{(\nu, M) \in \inv{F_{p/2}}\{\mu\}} \biggl( -\frac{c}{2} \int_{\RZ_{\geq 0}^2} \log\bigl( M^{2/p} \lvert x - y \rvert \bigr) \, \diff\mu^{\otimes 2}(x, y)\\
&\mspace{120mu} - \frac{1 - c}{2} \int_{\RZ_{\geq 0}} \log(M^{2/p} x) \, \diff\mu(x) + M + B_{c, p} \biggr)\\
&= \inf_{(\nu, M) \in \inv{F_{p/2}}\{\mu\}} \biggl( -\frac{c}{2} \int_{\RZ_{\geq 0}^2} \log\lvert x - y \rvert \, \diff\mu^{\otimes 2}(x, y) - \frac{1 - c}{2} \int_{\RZ_{\geq 0}} \log(x) \, \diff\mu(x)\\
&\mspace{120mu} - \frac{c}{p} \log(M) - \frac{1 - c}{p} \log(M) + M + B_{c, p} \biggr)\\
&= -\frac{c}{2} \int_{\RZ_{\geq 0}^2} \log\lvert x - y \rvert \, \diff\mu^{\otimes 2}(x, y) - \frac{1 - c}{2} \int_{\RZ_{\geq 0}} \log(x) \, \diff\mu(x) + \frac{\log(\ez p)}{p} + B_{c, p},
\end{align*}
because \(\frac{\log(\ez p)}{p}\) is the global minimum of \(-\frac{1}{p} \log(x) + x\) on \(\RZ_{> 0}\). But this concludes \(\mathcal{I} = \mathcal{I}_{c, p}^2\).

Concerning the minimizer of \(\mathcal{I}_{c, p}^2\), from Proposition~\ref{sa:ldp_empir_paar} we know that \(\mathcal{J}_{c, p}^2\) has the global minimizer \((\nu_{c, p}, m_{p/2}(\nu_{c, p}))\) and hence \(\mu_{c, p}^2 = F_{p/2}(\nu_{c, p}, m_{p/2}(\nu_{c, p}))\) is a global minimizer of \(\mathcal{I}_{c, p}^2\). It remains to show unicity. Let \(\mu \in \WMasz(\RZ_{\geq 0})\) with \(\mu \neq \mu_{c, p}^2\), w.l.o.g.\ \(\mathcal{I}_{c, p}^2(\mu) < \infty\); then \(\{\mu\}\) is weakly closed and thus also the fibre \(\inv{F_{p/2}}\{\mu\}\) is closed in \(\WMasz(\RZ_{\geq 0}) \times \RZ_{> 0}\). Because of \(\mathcal{I}_{c, p}^2(\mu) < \infty\) and because \(\mathcal{J}_{c, p}^2\) is a good rate function there exists \((\nu, M) \in \inv{F_{p/2}}\{\mu\}\) such that \(\mathcal{I}_{c, p}^2(\mu) = \mathcal{J}_{c, p}^2(\nu, M)\). But \(\mu \neq \mu_{c, p}^2\) forces \((\nu, M) \neq (\nu_{c, p}, m_{p/2}(\nu_{c, p}))\) and therefore \(\mathcal{J}_{c, p}^2(\nu, M) > 0\).
\end{proof}

\paragraph{Step~4}
The key observation here is stated in the lemma below.

\begin{Lem}\label{lem:ldp_bildmasz}
Let \((\mu_n)_{n \geq 1}\) be a sequence of \(\WMasz(\RZ_{\geq 0})\)\-/valued random variables which satisfies an LDP at speed \((s_n)_{n \geq 1}\) with GRF \(I\), and let \(q \colon \RZ_{\geq 0} \to \RZ_{\geq 0}\) be a measurable bijection with continuous inverse. Then \((\mu_n \circ q)_{n \geq 1}\) satisfies an LDP at speed \((s_n)_{n \geq 1}\) with GRF \(I^q \colon \mu \mapsto I(\mu \circ \inv{q})\). In particular \(\mu \circ q\) is a minimizer of \(I^q\) iff \(\mu\) is a mimizer of \(I\).
\end{Lem}

\begin{proof}
It suffices to prove that the map
\begin{equation*}
F \colon \left\{ \begin{aligned} \WMasz(\RZ_{\geq 0}) &\to \WMasz(\RZ_{\geq 0}) \\ \mu &\mapsto \mu \circ q \end{aligned} \right.
\end{equation*}
is weakly continuous, then apply the contraction principle to \((F(\mu_n))_{n \geq 1}\) which will lead to the desired results. Note that \(F\) is well\-/defined since \(q\) is a measurable isomorphism, and that for the same reason \(F\) is bijective.

Let \(\mu\) in \(\WMasz(\RZ_{\geq 0})\) and let \(V \subset \WMasz(\RZ_{\geq 0})\) be a basic weak neighborhood of \(F(\mu) = \mu \circ q\), that is,
\begin{equation*}
V = \bigcap_{i = 1}^k \biggl\{ \nu \in \WMasz(\RZ_{\geq 0}) \Colon \biggl\lvert \int_{\RZ_{\geq 0}} f_i \, \diff\nu - \int_{\RZ_{\geq 0}} f_i \, \diff(\mu \circ q) \biggr\rvert < \epsilon \biggr\}
\end{equation*}
with some \(k \in \NZ\), \(\epsilon \in \RZ_{> 0}\) and \(f_1, \dotsc, f_k \in \CeBe(\RZ_{\geq 0})\). Now \(f_i \circ \inv{q} \in \CeBe(\RZ_{\geq 0})\) for any \(i \in [1, k]\), therefore
\begin{equation*}
U := \bigcap_{i = 1}^k \biggl\{ \nu \in \WMasz(\RZ_{\geq 0}) \Colon \biggl\lvert \int_{\RZ_{\geq 0}} (f_i \circ \inv{q}) \, \diff\nu - \int_{\RZ_{\geq 0}} (f_i \circ \inv{q}) \, \diff\mu \biggr\rvert < \epsilon \biggr\}
\end{equation*}
is a weak neighborhood of \(\mu\), and for any \(\nu \in U\) we get, for any \(i \in [1, k]\),
\begin{equation*}
\biggl\lvert \int_{\RZ_{\geq 0}} f_i \, \diff(\nu \circ q) - \int_{\RZ_{\geq 0}} f_i \, \diff(\mu \circ q) \biggr\rvert = \biggl\lvert \int_{\RZ_{\geq 0}} (f_i \circ \inv{q}) \, \diff\nu - \int_{\RZ_{\geq 0}} (f_i \circ \inv{q}) \, \diff\mu \biggr\rvert < \epsilon,
\end{equation*}
and thus (recall \(F(\nu) = \nu \circ q\)) \(F(\nu) \in V\).
\end{proof}

The most important consequence of Lemma~\ref{lem:ldp_bildmasz} for us is that not only the empirical measures of the squares of singular values satisfy an LDP, but also of the singular values themselves. This concludes the proof of Theorem~\ref{sa:ldp_singulaerwerte} in the case \(X^{(n)} \sim \KegM{\Schatten{p}{}, \beta}{m \times n}\).

\begin{Sa}\label{sa:ldp_singulaerwerte_kegm}
Let \((X^{(n)})_{n \geq m}\) be a sequence of random variables such that \(X^{(n)} \sim \KegM{\Schatten{p}{}, \beta}{m \times n}\) for each \(n \in \NZ\), and define \(\mu_n := \frac{1}{m} \sum_{i = 1}^m \delta_{m^{1/p} s_i(X^{(n)})}\), then the sequence \((\mu_n)_{n \in \NZ}\) satsifies a large deviations principle with speed \(\beta m n\) and good rate function
\begin{equation*}
\mathcal{I}_{c, p}(\mu) := \begin{cases} -\frac{c}{2} \int_{\RZ_{\geq 0}^2} \log\lvert x^2 - y^2 \rvert \, \diff\mu^{\otimes 2}(x, y) - (1 - c) \int_{\RZ_{\geq 0}} \log(x) \, \diff\mu(x) + \frac{\log(\ez p)}{p} + B_{c, p} & \text{if } m_p(\mu) \leq 1, \\ \infty & \text{else,} \end{cases}
\end{equation*}
which possesses a unique minimizer \(\mu_{c, p}\), given by its Lebesgue density
\begin{equation*}
\frac{\diff\mu_{c, p}(x)}{\diff x} = 2 x \, \frac{\diff\mu_{c, p}^2(x^2)}{\diff x},
\end{equation*}
where \(\mu_{c, p}^2\) is the minimizer from Proposition~\ref{sa:ldp_quadsingulaerwerte_kegm}.
\end{Sa}

\begin{proof}
Apply Lemma~\ref{lem:ldp_bildmasz} to the LDP for \(\bigl( \frac{1}{m} \sum_{i = 1}^m \delta_{m^{2/p} s_i(X^{(n)})^2} \bigr)_{n \in \NZ}\) of Proposition~\ref{sa:ldp_quadsingulaerwerte_kegm} and the map \(q(x) := x^2\), defined on \(\RZ_{\geq 0}\); use that generally \(\delta_x \circ q = \delta_{\inv{q}(x)} = \delta_{x^{1/2}}\) and therefore especially \(\bigl( \frac{1}{m} \sum_{i = 1}^m \delta_{m^{2/p} s_i(X^{(n)})^2} \bigr) \circ q = \frac{1}{m} \sum_{i = 1}^m \delta_{m^{1/p} s_i(X^{(n)})}\); also mind \(m_{p/2}(\mu) = m_p(\mu \circ q)\). The density of \(\mu_{c, p}\) is a simple instance of density transformation.
\end{proof}

\begin{Bem}\label{bem:minimierer_p2}
As already indicated in Remarks~\ref{bem:ldp_singulaerwerte} and~\ref{bem:ldp_y} the minimizer \(\mu_{c, 2}\) is explicitly known. This is because then the rate function \(\mathcal{J}_{c, 2}\) from Proposition~\ref{sa:ldp_y} can be written as
\begin{equation*}
\mathcal{J}_{c, 2}(\mu) = \frac{c}{2} \int_{\CZ} \log\Bigl( \frac{1}{\lvert x - y \rvert w(x) w(y)} \Bigr) \, \diff\mu^{\otimes 2}(x, y) + B_{c, 2},
\end{equation*}
with the weight function
\begin{equation*}
w \colon \left\{ \begin{aligned} \CZ &\to \CZ \\ x &\mapsto x^{(1 - c)/(2 c)} \, \ez^{-x/c} \Ind_{\RZ_{\geq 0}}(x) \end{aligned} \right.
\end{equation*}
This is a \emph{Laguerre weight}, and by \cite[Example~IV.5.4]{SaffTotik1997} the minimizer \(\nu_{c, 2}\) of \(\mathcal{J}_{c, 2}\) has Lebesgue density
\begin{equation*}
\frac{\diff \nu_{c, 2}(x)}{\diff x} = \frac{1}{c \pi x} \Bigl( c - \Bigl( x - \frac{1 + c}{2} \Bigr)^2 \Bigr)^{1/2} \Ind_{\bigl[ \frac{(1 - \sqrt{c})^2}{2}, \frac{(1 + \sqrt{c})^2}{2} \bigr]}(x).
\end{equation*}
We can easily compute \(m_1(\nu_{c, 2}) = \frac{1}{2}\), and therewith Proposition~\ref{sa:ldp_quadsingulaerwerte_kegm} yields the minimizer \(\mu_{c, 2}^2\) with density
\begin{equation*}
\frac{\diff\mu_{c, 2}^2(x)}{\diff x} = \frac{1}{2} \, \frac{\diff\nu_{c, 2}}{\diff x}\Bigl( \frac{x}{2} \Bigr) = \frac{1}{2 \pi c x} \bigl( 4 c - (x - 1 - c)^2 \bigr)^{1/2} \Ind_{[(1 - \sqrt{c})^2, (1 + \sqrt{c})^2]}(x);
\end{equation*}
from this we can identify \(\mu_{c, 2}^2\) as a \emph{Marchenko\--Pastur distribution.} And finally from Proposition~\ref{sa:ldp_singulaerwerte_kegm} we obtain the minimizer \(\mu_{c, 2}\) with density
\begin{equation*}
\frac{\diff\mu_{c, 2}(x)}{\diff x} = \frac{1}{\pi c x} \bigl( 4 c - (x^2 - 1 - c)^2 \bigr)^{1/2} \Ind_{[1 - \sqrt{c}, 1 + \sqrt{c}]}(x);
\end{equation*}
in particular we have for \(c = 1\)
\begin{equation*}
\frac{\diff\mu_{1, 2}(x)}{\diff x} = \frac{1}{\pi} \, (4 - x^2)^{1/2} \Ind_{[0, 2]}(x),
\end{equation*}
which corresponds to a quarter\-/circle distribution, as expected.

Also the constant \(B_{c, 2}\) can be computed explicitly because of
\begin{align*}
Z_{m, n, 2, \beta} &= \int_{\RZ_{\geq 0}^m} \ez^{-\beta n \lVert y \rVert_1} \prod_{i = 1}^m y_i^{\beta (n - m + 1)/2 - 1} \prod_{1 \leq i < j \leq m} \lvert y_j - y_i \rvert^\beta \, \diff y\\
&= (\beta n)^{-\beta m n/2} \int_{\RZ_{\geq 0}^m} \ez^{-\lVert y \rVert_1} \prod_{i = 1}^m y_i^{\beta (n - m + 1)/2 - 1} \prod_{1 \leq i < j \leq m} \lvert y_j - y_i \rvert^\beta \, \diff y\\
&= (\beta n)^{-\beta m n/2} \prod_{i = 1}^m \frac{\Gamma(\frac{\beta (n - m + 1)}{2} + \frac{\beta (i - 1)}{2}) \Gamma(1 + \frac{\beta i}{2})}{\Gamma(1 + \frac{\beta}{2})},
\end{align*}
where we have used the Selberg integral of Laguerre type (see \cite[p.\ 354]{Mehta1991}). Using Lemma~\ref{lem:produkt_gamma} to deal with the products of gamma functions we arrive at
\begin{equation*}
B_{c, 2} = -\frac{3}{4} - \frac{\log(2)}{2} + \frac{c}{4} \log(c) - \frac{(1 - c)^2}{4 c} \log(1 - c).
\end{equation*}
For \(c = 1\) this is consistent with \cite[Theorem~1.5]{KPTh2020_2}.
\end{Bem}

\paragraph{Step~5}
The last step is the establishment of the LDP w.r.t.\ the uniform distribution on the ball, i.e.\ \(X^{(n)} \sim \Gleichv(\Kug{\Schatten{p}{}, \beta}{m \times n})\). Recall \(s(X^{(n)}) \GlVert U^{1/(\beta m n)} s(\Theta^{(n)})\), where \(\Theta^{(n)} \sim \KegM{\Schatten{p}{}, \beta}{m \times n}\) and hence its LDP is provided by Proposition~\ref{sa:ldp_singulaerwerte_kegm}, and \(U \sim \Gleichv([0, 1])\) is independent from \(\Theta^{(n)}\) and satisfies the following well\-/known LDP (see~\cite{AGPTh2018_2}, e.g., for a quite general version).

\begin{Lem}
The sequence \((U^{1/(\beta m n)})_{n \in \NZ}\) satisfies a large deviations principle with speed \(\beta m n\) and good rate function \(\mathcal{I}_U \colon \RZ \to [0, \infty]\) given by
\begin{equation*}
\mathcal{I}_U(x) := \begin{cases} -\log(x) & \text{if } x \in (0, 1], \\ \infty & \text{else.} \end{cases}
\end{equation*}
\end{Lem}

\begin{Sa}\label{sa:ldp_endl_vollkugel}
Let \((X^{(n)})_{n \in \NZ}\) be a sequence of random variables such that \(X^{(n)} \sim \Gleichv(\Kug{\Schatten{p}{}, \beta}{m \times n})\) for each \(n \in \NZ\). Then \((\mu_n)_{n \in \NZ}\) with \(\mu_n := \frac{1}{m} \sum_{i = 1}^m \delta_{m^{1/p} s_i(X^{(n)})}\) satisfies a large deviations principle with speed \(\beta m n\) and the same good rate function \(\mathcal{I}_{c, p}\) as in Proposition~\ref{sa:ldp_singulaerwerte_kegm}.
\end{Sa}

\begin{proof}
Again the argument runs parallel to \cite[Proposition~6.2]{KPTh2020_2}, so we adopt their notations with the obvious adaptations, and thus the contraction principle yields the LDP with speed \(\beta m n\) and the good rate function
\begin{equation*}
\mathcal{I}(\mu) := \inf_{(\nu, u) \in \inv{F}\{\mu\}} \bigl( \mathcal{I}_{c, p}(\nu) - \log(u) \bigr),
\end{equation*}
where \(\mathcal{I}_{c, p}\) is the rate function from Proposition~\ref{sa:ldp_singulaerwerte_kegm}. The cases \(m_p(\mu) > 1\) and \(m_p(\mu) \leq 1 \wedge \mu_p(\nu) > 1\) via the same reasoning lead to \(\mathcal{I}(\mu) = \infty\), and in the remaining case \(m_p(\mu) \leq 1 \wedge m_p(\nu) \leq 1\) these conditions also force \(u \in [m_p(\mu)^{1/p}, 1]\) and thence
\begin{align*}
\mathcal{I}(\mu) &= \inf_{(\nu, u) \in \inv{F}\{\mu\}} \biggl( -\frac{c}{2} \int_{\RZ_{\geq 0}^2} \log\lvert x^2 - y^2 \rvert \, \diff\nu^{\otimes 2}(x, y) - (1 - c) \int_{\RZ_{\geq 0}} \log(x) \, \diff\nu(x)\\
&\mspace{120mu} + \frac{\log(\ez p)}{p} + B_{c, p} - \log(u) \biggr)\\
&= \inf_{(\nu, u) \in \inv{F}\{\mu\}} \biggl( -\frac{c}{2} \int_{\RZ_{\geq 0}^2} \log\Bigl( \frac{\lvert x^2 - y^2 \rvert}{u^2} \Bigr) \, \diff\mu^{\otimes 2}(x, y) - (1 - c) \int_{\RZ_{\geq 0}} \log\Bigl( \frac{x}{u} \Bigr) \, \diff\mu(x)\\
&\mspace{120mu} + \frac{\log(\ez p)}{p} + B_{c, p} - \log(u) \biggr)\\
&= -\frac{c}{2} \int_{\RZ_{\geq 0}^2} \log\lvert x^2 - y^2 \rvert \, \diff\mu^{\otimes 2}(x, y) - (1 - c) \int_{\RZ_{\geq 0}} \log(x) \, \diff\mu(x) + \frac{\log(\ez p)}{p} + B_{c, p},
\end{align*}
proving \(\mathcal{I} = \mathcal{I}_{c, p}\).
\end{proof}

\subsection{Proof of Theorem~\ref{sa:ldp_singulaerwerte_unendl}}
\label{sec:ldp_beweis_unendl}

Recall from Theorem~\ref{sa:gleichvert_schattenq_kugel_sphaere}, part~3, the representation \(X^{(n)} \GlVert V \diag(((Y_i^{(n)})^{1/2})_{i \leq m}) \adj{U}\) whenever \(X^{(n)} \sim \Gleichv(\Kug{\Schatten{\infty}{}, \beta}{m \times n})\), where \(Y^{(n)}\) has Lebesgue density \eqref{eq:dichte_y}. Hence \((s_{\pi(i)}(X^{(n)})^2)_{i \leq m} \GlVert Y^{(n)}\), where \(\pi \sim \Gleichv(\Perm{m})\) is independent of \(X^{(n)}\), and via Lemma~\ref{lem:permutation} we know \(\frac{1}{m} \sum_{i = 1}^m \delta_{s_i(X^{(n)})^2} \GlVert \frac{1}{m} \sum_{i = 1}^m \delta_{Y_i^{(n)}} =: \nu_n\). So we state an LDP for \((\nu_n)_{n \in \NZ}\) first.

\begin{Sa}\label{sa:ldp_singulaerwerte_quad_unend}
If \(\lim_{n \to \infty} \frac{m}{n} = c \in (0, 1]\), then \((\nu_n)_{n \in \NZ}\) satisfies a large deviations principle at speed \(\beta m n\) with good rate function
\begin{equation*}
\mathcal{J}_{c, \infty}(\mu) = \begin{cases} -\frac{c}{2} \int_{[0, 1]^2} \log\lvert x - y \rvert \, \diff\mu^{\otimes 2}(x, y) - \frac{1 - c}{2} \int_{[0, 1]} \log(x) \, \diff\mu(x) + B_{c, \infty} & \text{if } \supp(\mu) \subset [0, 1], \\ \infty & \text{else,} \end{cases}
\end{equation*}
where \(B_{c, \infty}\) is the same as in Theorem~\ref{sa:ldp_singulaerwerte_unendl}; and \(\mathcal{J}_{c, \infty}\) has a unique minimizer \(\nu_{c, \infty}\), which is given explicitly by its Lebesgue density,
\begin{equation*}
\frac{\diff\nu_{c, \infty}(x)}{\diff x} = \frac{1 + c}{2 c \pi x} \, \frac{\bigl( x - (\frac{1 - c}{1 + c})^2 \bigr)^{1/2}}{(1 - x)^{1/2}} \Ind_{[(\frac{1 - c}{1 + c})^2, 1]}(x).
\end{equation*}
\end{Sa}

Before giving a proof for Proposition~\ref{sa:ldp_singulaerwerte_quad_unend} we observe that it implies Theorem~\ref{sa:ldp_singulaerwerte_unendl} almost immediately through an application of Lemma~\ref{lem:ldp_bildmasz}.

Now we prove Proposition~\ref{sa:ldp_singulaerwerte_quad_unend}. First notice that, in contrast to Proposition~\ref{sa:ldp_y}, \cite[Theorem~5.5.1]{HiaiPetz2000} is not applicable directly because they only allow continuous external potentials defined on all \(\RZ_{\geq 0}\), but in our case we have
\begin{equation*}
\Ind_{[0, 1]^m}(y) = \ez^{-\beta n \sum_{i = 1}^m Q(y_i)}
\end{equation*}
with
\begin{equation*}
Q(y) := \begin{cases} 0 & \text{if } y \in [0, 1], \\ \infty & \text{else.} \end{cases}
\end{equation*}
Yet it turns out that their proof can be adapted. Since it requires only minor modifications we present only those and refer the reader to \cite{HiaiPetz2000} for the details.

First we show that the rate function must be infinite at any \(\mu \in \WMasz(\RZ_{\geq 0})\) with \(\supp(\mu) \setminus [0, 1] \neq \emptyset\).

\begin{Lem}\label{lem:ratenfunktion_unendl}
Let \(\mu \in \WMasz(\RZ_{\geq 0})\) with \(\supp(\mu) \setminus [0, 1] \neq \emptyset\). Then
\begin{equation*}
\inf_{G} \limsup_{n \to \infty} \frac{1}{\beta m n} \log \Wsk[\nu_n \in G] = -\infty,
\end{equation*}
where the infimum is taken over basic neighborhoods \(G\) of \(\mu\).
\end{Lem}

\begin{proof}
It suffices to show that there exists a weak neighborhood \(G\) of \(\mu\) such that \(\supp(\nu) \setminus [0, 1] \neq \emptyset\) for all \(\nu \in G\); since \(\supp(\nu_n) \subset [0, 1]\) this implies \(\nu_n \notin G\) for all \(n \in \NZ\) and hence the statement follows.

Because \((1, \infty)\) is open, the evaluation map
\begin{equation*}
\WMasz(\RZ_{\geq 0}) \ni \nu \mapsto \nu((1, \infty))
\end{equation*}
is weakly lower semicontinuous; by the premises on \(\mu\) we know \(\mu((1, \infty)) > 0\), and hence there exists a (basic) weak neighborhood \(G\) of \(\mu\) such that, for any \(\nu \in G\),
\begin{equation*}
\nu((1, \infty)) > \frac{\mu((1, \infty))}{2} > 0.
\end{equation*}
This implies \(\supp(\nu) \setminus [0, 1] \neq \emptyset\) for all \(\nu \in G\).
\end{proof}

From now on we only consider \(\mu \in \WMasz(\RZ_{\geq 0})\) such that \(\supp(\mu) \subset [0, 1]\).

In order to adapt the proof of \cite[Theorem~5.5.1]{HiaiPetz2000}, define the kernel functions on \(\RZ_{\geq 0}^2\),
\begin{gather*}
F(x, y) := \begin{cases} -\frac{c}{2} \log\lvert x - y \rvert - \frac{1 - c}{4} \log(x y) & \text{if } (x, y) \in [0, 1]^2, \\ \infty & \text{else,} \end{cases}\\
\tilde{F}_n(x, y) := \begin{cases} -\frac{m}{2 n} \log\lvert x - y \rvert - \frac{\beta (n - m + 1) - 2}{4 \beta n} \log(x y) & \text{if } (x, y) \in [0, 1]^2, \\ \infty & \text{else,} \end{cases}
\end{gather*}
and for any \(\alpha > 0\) their cutoffs \(F_\alpha := \min\{F, \alpha\}\) and analogously \(\tilde{F}_{n, \alpha}\). Then \(F\) and \(\tilde{F}_n\) are bounded from below and lower semicontinuous, and \(F_\alpha\) and \(\tilde{F}_{n, \alpha}\) are bounded and lower semicontinuous. Therefore for any \(F' \in \{F, \tilde{F}_n, F_\alpha, \tilde{F}_{n, \alpha}\}\) the map
\begin{equation*}
\WMasz(\RZ_{\geq 0}) \ni \nu \mapsto \int_{\RZ_{\geq 0}^2} F' \, \diff\nu^{\otimes 2}
\end{equation*}
is a well\-/defined weakly lower semicontinuous functional.

\cite[Lemma~5.5.2]{HiaiPetz2000} remains true unchanged. The existence of the unique minimizers \(\mu_0\), \(\tilde{\mu}_n\) needs to be argued differently: write
\begin{equation*}
\int_{\RZ_{\geq 0}^2} F \, \diff\mu^{\otimes 2} = \frac{c}{2} \int_{\CZ} \log\Bigl( \frac{1}{\lvert x - y \rvert w(x) w(y)} \Bigr) \diff\mu^{\otimes 2}(x, y)
\end{equation*}
with the weight function
\begin{equation*}
w \colon \left\{ \begin{aligned} \CZ &\to \CZ \\ x &\mapsto x^{(1 - c)/(2 c)} \Ind_{[0, 1]}(x) \end{aligned} \right.,
\end{equation*}
then \(w\) is an admissible weight in the sense of \cite[Definition~I.1.1]{SaffTotik1997}, and by \cite[Theorem~I.1.3]{SaffTotik1997} the energy functional admits a unique minimizer \(\nu_{c, \infty}\); even more, \cite[Example~IV.5.3]{SaffTotik1997} with \(\theta := \frac{1 - c}{1 + c}\) gives the Lebesgue density of \(\nu_{c, \infty}\), and it equals that given in Proposition~\ref{sa:ldp_singulaerwerte_quad_unend}. The analogous argument is valid for \(\tilde{F}_n\), but the only relevant information about \(\tilde{\mu}_n\) which we need is that \(\supp(\tilde{\mu}_n) \subset [0, 1]\), and this already proves that \((\tilde{\mu}_n)_{n \in \NZ}\) is tight, which is the first part of \cite[Lemma~5.5.3]{HiaiPetz2000}, and the argument for the second part does not change.

\cite[Lemma~5.5.4]{HiaiPetz2000} is adapted easily; note that integration extends simply over \([0, 1]^m\), and \(Q = 0\). In truth we know more as
\begin{equation*}
Z_{m, n, \infty, \beta} = \int_{[0, 1]^m} \prod_{i = 1}^m y_i^{\beta (n - m + 1)/2 - 1} \prod_{1 \leq i < j \leq m} \lvert y_j - y_i \rvert^\beta \, \diff y
\end{equation*}
is just a Selberg integral whose explicit value we know and whose asymptotics can be tackled with Lemma~\ref{lem:produkt_gamma}; this yields the limit \(B_{c, \infty} = \lim_{n \to \infty} \frac{1}{\beta m n} \log(Z_{m, n, \infty, \beta})\) stated in Theorem~\ref{sa:ldp_singulaerwerte_unendl}.

The only thing to be noted in the proof of \cite[Lemma~5.5.5]{HiaiPetz2000} is that \(\nu \mapsto \int_{\RZ_{\geq 0}^2} F_\alpha \, \diff\nu^{\otimes 2}\) is only weakly lower semicontinuous, but this suffices to ensure
\begin{equation*}
\inf_G \biggl( -\inf_{\nu \in G} \int_{\RZ_{\geq 0}^2} F_\alpha \, \diff\nu^{\otimes 2} \biggr) \leq -\int_{\RZ_{\geq 0}^2} F_\alpha \, \diff\mu^{\otimes 2},
\end{equation*}
which is all we need. Indeed, for any \(M \in \RZ\) with \(M < \int_{\RZ_{\geq 0}^2} F_\alpha \, \diff\mu^{\otimes 2}\) there exists a basic neighborhood \(G\) of \(\mu\) such that, for all \(\nu \in G\), \(\int_{\RZ_{\geq 0}^2} F_\alpha \, \diff\nu^{\otimes 2} > M\). This implies \(\inf_{\nu \in G} \int_{\RZ_{\geq 0}^2} F_\alpha \, \diff\nu^{\otimes 2} \geq M\), hence \(\sup_G \inf_{\nu \in G} \int_{\RZ_{\geq 0}^2} F_\alpha \, \diff\nu^{\otimes 2} \geq M\), and letting \(M \to \int_{\RZ_{\geq 0}^2} F_\alpha \, \diff\mu^{\otimes 2}\) and taking the negatives proves the claim.

The only subtlety in the proof of \cite[Lemma~5.5.6]{HiaiPetz2000} is that the exponent \(\frac{\beta (n - m + 1)}{2} - 1\) may be positive or negative in the case \(c = 1\) (in the case \(c < 1\) it is positive up to finitely many \(n\)), so \(y_i^{\beta (n - m + 1)/2 - 1}\) must be bounded from below by \(\min\bigl\{ (a_i^{(n)})^{\beta (n - m + 1)/2 - 1}, (b_i^{(n)})^{\beta (n - m + 1)/2 - 1} \bigr\}\); yet this does not influence the validity of the following arguments.

Lastly the exponential tightness of \((\nu_n)_{n \in \NZ}\) is a simple consequence of \(\supp(\nu_n) \subset [0, 1]\) for all \(n \in \NZ\), and this concludes the proof of Proposition~\ref{sa:ldp_singulaerwerte_quad_unend}.

\subsection*{Acknowledgement}
Michael Juhos and Joscha Prochno have been supported by the German Research Foundation (DFG) under project 516672205 and the Austrian Science Fund (FWF) Project P32405 \textit{Asymptotic Geometric Analysis and Applications}. Zakhar Kabluchko has been supported by the German Research Foundation under Germany’s Excellence Strategy EXC~2044~-- 390685587, \textit{Mathematics M\"unster: Dynamics~-- Geometry~-- Structure} and by the DFG priority program SPP~2265 \textit{Random Geometric Systems}.

\bibliographystyle{abbrv}
\bibliography{schatten}

\vspace{2\bigskipamount}

\begin{flushleft}
\textsc{Michael Juhos:} Faculty of Computer Science and Mathematics, University of Passau, Innstra{\ss}e~33, 94032~Passau, Germany\\
\textit{E\=/mail:} \texttt{michael.juhos@uni-passau.de}

\vspace{0.5\bigskipamount}

\textsc{Zakhar Kabluchko:} Faculty of Mathematics, University of M\"unster, Orl\'eans\-/Ring~10, 48149~M\"unster, Germany\\
\textit{E\=/mail:} \texttt{zakhar.kabluchko@uni-muenster.de}

\vspace{0.5\bigskipamount}

\textsc{Joscha Prochno:} Faculty of Computer Science and Mathematics, University of Passau, Innstra{\ss}e~33, 94032~Passau, Germany\\
\textit{E\=/mail:} \texttt{joscha.prochno@uni-passau.de}
\end{flushleft}

\end{document}